\documentclass{article}
\pdfoutput=1
\usepackage{etoolbox}
\newbool{production}
\setbool{production}{false}
\usepackage{preamble}
\title{A notion of homotopy for directed graphs and their flag complexes}
\shorttitle{A notion of homotopy for directed graphs and their flag complexes}
\author[1,*]{Thomas Chaplin}
\author[1,2,3,4]{Heather A. Harrington}
\author[1,5]{Ulrike Tillmann}
\affil[1]{Mathematical Institute, University of Oxford}
\affil[2]{Max Planck Institute of Molecular Cell Biology and Genetics, Dresden Germany}
\affil[3]{Centre for Systems Biology Dresden, Germany}
\affil[4]{Faculty of Mathematics, Technische Universit\"at Dresden, Germany}
\affil[5]{Isaac Newton Institute, University of Cambridge}
\affil[*]{Corresponding author: {\normalfont\texttt{\href{mailto:thomas.chaplin@maths.ox.ac.uk}{thomas.chaplin@maths.ox.ac.uk}}}}

\begin{document}

\maketitle
\begin{abstract}
Directed graphs can be studied by their associated directed flag complex.
The homology of this complex has been successful in applications as a topological invariant for digraphs.
Through comparison with path homology theory, we derive a homotopy-like equivalence relation on digraph maps such that equivalent maps induce identical maps on the homology of the directed flag complex.
Thus, we obtain an equivalence relation on digraphs such that equivalent digraphs have directed flag complexes with isomorphic homology.
With the help of these relations, we can prove a generic stability theorem for the persistent homology of the directed flag complex of filtered digraphs.
In particular, we show that the persistent homology of the directed flag complex of the shortest-path filtration of a weighted directed acyclic graph is stable to edge subdivision.
In contrast, we also discuss some important instabilities that are not present in persistent path homology.
We also derive similar equivalence relations for ordered simplicial complexes at large.
Since such complexes can alternatively be viewed as simplicial sets, we verify that these two perspectives yield identical relations.
\end{abstract}
 \setcounter{tocdepth}{2}
\tableofcontents

\section{Introduction}
Digraphs appear naturally in numerous applications domains, including neuroscience~\cite{reimann2017cliques}, traffic network analysis~\cite{bittner2018comparing} and social network analysis~\cite{oliveira2012overview}.
Frequently these digraphs support a dynamical system such as brain activity, traffic flow and (dis)information spread.
A common hypothesis is that the structure of the digraph is critical in determining the evolution of the dynamical system.
In order to assess this hypothesis, there is a need for interpretable and informative summaries of digraphs, that are amenable to statistical analysis.

A common approach is to associate a combinatorial object to the digraph
built out of substructures in the digraph which are relevant to the application in hand.
For example, one could study the set of all paths~\cite{Grigoryan2012}, all tournaments~\cite{Govc2021} or all directed cliques~\cite{Masulli2016} in the digraph. The interest to us is that each of these combinatorial objects carries sufficient structure to build a chain complex,
the homology of which can be used as an algebraic invariant of the underlying digraph.
These invariants are most useful (such as in the context of persistence) when they are functorial, in the sense that there is a class of digraph maps which induce maps on homology.

In order to assess the discriminative power of these homologies and hence their utility in applications, some natural questions arise:
When do two digraphs give rise to isomorphic homology?
When do two digraph maps induce \emph{the same} maps on homology?

In this work, we make progress towards these two questions in the case of the directed flag complex.
This is a simplicial set in which the $k$-simplies are the directed $(k+1)$-cliques, i.e.
$(k+1)$-tuples of distinct vertices $v_0 \dots v_k$ such that there is an edge $v_i \to v_j$ whenever $i< j$.
We focus on the directed flag complex because it has seen particular success in the field of neuroscience. The homology of the directed flag complex was proposed as a digraph invariant by \citeauthor{Masulli2016}~\cite{Masulli2016}
and, shortly thereafter,
the Betti numbers of the complex
associated to the activation graph of numerically simulated neocortical microcircuitry
was shown to
exhibit significant temporal patterns in response to stimuli~\cite{reimann2017cliques}.
Further applications may be enabled by the incorporation of persistence, 
and the presence of highly performant software for its computation~\cite{Luetgehetmann2020},
which can be used to capture the topological organisation of directed structures, across a range of scales~\cite{Caputi2021}.

In contrast, the directed flag complex has received comparatively little theoretical treatment.
Here, we describe an equivalence relation $\simeq_{\dFl}$ for digraph maps such that if two maps are $\simeq_{\dFl}$-equivalent then they induce identical maps on the homology of the directed flag complex.
In particular, this induces an equivalence relation $\simeq_{\dFl}$ on digraphs themselves, such that if two digraphs are $\simeq_{\dFl}$-equivalent then their directed flag complexes have isomorphic homology.
To illustrate the utility of these relations in the context of persistence, we use the relation $\simeq_{\dFl}$ to show that the persistent homology of the directed flag complex applied to a filtration of digraphs is stable in certain circumstances.

There are numerous homotopy theories for both directed and undirected graphs.
The relation we will develop is primarily derived from path homotopy theory, which was introduced by~\citeauthor{Grigoryan2014} in a series of papers~\cite{Grigoryan2020,Grigoryan2012, Grigoryan2014}.
Initially developed for digraphs, path homotopy generalises a prior homotopy theory for undirected graphs~\cite{Babson2006, Barcelo2001}, viewing undirected graphs as a full subcategory of digraphs.
This theory was further generalised to path complexes in later work~\cite{GRIGORYAN2019106877}.

\subsection{Summary}

Given a digraph $G$, we denote the associated directed flag complex by $\dFl(G)$.
The directed flag complex is a particular example of a more general class of objects called ordered simplicial complexes (Definition~\ref{def:osc}).
In this work, we primarily view an ordered simplicial complex as a special case of a regular path complex (Definition~\ref{def:rpc}).
The latter were first introduced in~\cite{Grigoryan2012} as an invariant of directed graphs.
To any regular path complex one can associate a chain complex and moreover this can be made into a functor
$\Omega: \ascat{WkRPC} \to \ascat{Ch}$,
where $\ascat{WkRPC}$ is a category of regular path complexes and $\ascat{Ch}$ is the category of chain complexes of vector spaces over some background field.
Taking the full subcategory of $\ascat{WkRPC}$, in which objects are restricted to ordered simplicial complexes, yields a category $\ascat{TcOSC}$.
We fully characterise the morphisms in this category as weak simplicial morphisms that never map a directed $3$-clique to a reciprocal pair of edges (Lemma~\ref{lem:osc_weak_path_iff_tc}).
Pre-composing with the functor which takes digraphs to their directed flag complex, we obtain a sequence of functors
\begin{equation}\label{eq:intro_factorisation}
\ascat{TcDgr} \xrightarrow{\dFl} \ascat{TcOSC} \xhookrightarrow{\iota} \ascat{WkRPC} \xrightarrow{\Omega} \ascat{Ch}
\end{equation}
where the morphisms in $\ascat{TcDgr}$ are defined so that $\dFl$ becomes a full and faithful functor. We show that this is a factorisation of the chain complex of $\dFl(G)$ (Lemma~\ref{lem:gens_of_osc}).
This factorisation is the main perspective through which we construct our equivalence relation, $\simeq_{\dFl}$, by comparison with a similar relation for the well-studided category $\ascat{WkRPC}$.

In particular, there is a notion of homotopy equivalence, $\simeq$, between the morphisms of $\ascat{WkRPC}$ such that if $f \simeq g$ then $\Omega(f)$ and $\Omega(g)$ are chain homotopic (see~\cite{GRIGORYAN2019106877}).
The equivalence relation for digraph maps that we will construct is a pull-back of this equivalence relation, i.e.\
\begin{equation}
f \simeq_{\dFl} g \iff (\iota \circ \dFl)(f) \simeq (\iota \circ \dFl)(g).
\end{equation}
Our main contribution is an `intrinsic' characterisation of this pull-back (Section~\ref{sec:dflag_homotopy}).

To be more precise,~\cite{GRIGORYAN2019106877} describes a system of one-step homotopies between morphisms in $\ascat{WkRPC}$.
This is essentially a binary relation on the morphisms of $\ascat{WkRPC}$, which is then completed to the equivalence relation $\simeq$.
As above, we can pull back this binary relation to one for the morphisms of $\ascat{TcDgr}$.
Given two morphisms $f, g\in\ascat{TcDgr}$,
we can provide simple, edge-based conditions that are equivalent to the pair $(f, g)$ belonging to this pull-back binary relation (Corollary~\ref{cor:dfl_homot_characterise}).
In particular, we relate $(f, g)$ if
\begin{enumerate}
\item $x\tooreq y \implies f(x) \tooreq g(y)$; and
\item $x\to y$ and $f(x) = g(y)$ $\implies$ $f(x) = f(y)= g(x) = g(y)$,
\end{enumerate}
where $\tooreq$ indicates that either the vertices coincide or are joined by a directed edge.
The relation $\simeq_{\dFl}$ can then be constructed by completing this to an equivalence relation.

As a first step, in Section~\ref{sec:big_osc_homotopy_sec}, we study the pull-back binary relation for $\ascat{TcOSC}$ along the inclusion functor $\iota: \ascat{TcOSC} \hookrightarrow \ascat{WkRPC}$.
Again, we aim to describe the resulting relation `intrinsically' with the category $\ascat{TcOSC}$.
Corollary~\ref{cor:pullback_TcOSC_WkRPC} achieves this by
describing related morphisms in terms of a cylinder functor for this category.
For an alternative view, $\ascat{TcOSC}$ can also be fully and faithfully embedded in the category of simplicial sets, which itself has a notion of homotopy equivalence.
In Section~\ref{sec:sset}, we verify that the pull-back relation along this embedding coincides with the pull-back along $\iota$.

Finally, in Section~\ref{sec:stability}, we employ the equivalence relation $\simeq_{\dFl}$ to obtain stability results for the persistent homology of the directed flag complex of a filtration of digraphs.
Typically, such stability results take the form of a bound on the interleaving distance.
The relation $\simeq_{\dFl}$ allow us to construct interleavings of digraph filtrations up to $\simeq_{\dFl}$, facilitating the derivation of such bounds.
We summarise this approach with Corollary~\ref{cor:generic_stability}.
In Proposition~\ref{prop:subdiv_dag_stability}, we apply this result to the shortest-path filtration of a weighted directed acyclic graph (DAG), showing that its persistent homology is stable to edge subdivision.
However, persistent homology pipelines using the directed flag complex have some notable instabilities.
We discuss some important examples in Section~\ref{sec:instabilities}, namely edge subdivision of a non-DAG and the addition of an appendage edge.

\subsection{Notation}

Given categories $\C, \D$, the category of functors from $\C$ to $\D$ is denoted \mdf{$\Funct{\C}{\D}$}.
We use \mdf{$\MorXY{\C}{C_1}{C_2}$} to denote the set of morphisms between two objects $C_1, C_2 \in C$ and use \mdf{$\Obj(C)$} to denote the objects of $\C$.
Given functors $F, G: \C \to \D$, we denote a natural transformation by $\mu: F \Rightarrow G$.
We denote components of $\mu$ by $\mu_C$ for each object $C\in\C$, or optionally omit the notation when the component is clear from context.

\subsection{Acknowledgements}
HAH gratefully acknowledges funding from a Royal Society University Research Fellowship.
The authors are members of the Centre for Topological Data Analysis, which is funded by the EPSRC grant `New Approaches to Data Science: Application Driven Topological Data Analysis' \href{https://gow.epsrc.ukri.org/NGBOViewGrant.aspx?GrantRef=EP/R018472/1}{\texttt{EP/R018472/1}}.
For the purpose of Open Access, the authors have applied a CC BY public copyright licence to any  Author Accepted Manuscript (AAM) version arising from this submission. 
 \section{Combinatorial and algebraic complexes}\label{sec:complexes}

\subsection{Path complexes}\label{sec:path_comp}

\begin{defin}
Let $V$ be an arbitrary set, which we call the \mdf{vertices}.
\begin{enumerate}
    \item An \mdf{elementary $k$-path (on V)} is any sequence of $(k+1)$ vertices, written $p = v_0 \dots v_k$, where $v_i \in V$.
    \item Let \mdf{$v_0 \dots \hat{v_i} \dots v_k$} denote the $(k-1)$-path obtained by removing the vertex $v_i$.
    \item Given a function $f:V \to W$ and an elementary path $p= v_0 \dots v_k$, let \mdf{$f(p)$} denote the elementary $k$-path $f(v_0)\dots f(v_k)$.
    \item We say $p$ is \mdf{irregular} if $v_i = v_{i+1}$ for some $i$, otherwise we say $p$ is \mdf{regular}.
    \item We say $p$ is \mdf{simplicial} if each of the $v_i$ are distinct.
    \item If an elementary path $p'$ can be obtained from another elementary path $p$ by removing a subset of vertices, we say $p'$ is a \mdf{face} of $p$.
    \item If an elementary path $p'$ can be obtained from another elementary path $p$ by successively removing the initial or terminal vertex, we say $p'$ is a \mdf{sub-path} of $p$.
\end{enumerate}
\end{defin}

\begin{example}
On the vertex set $V=\{a, b\}$, $p_1 \defeq aab$ and $p_2 \defeq aba$ are both elementary $2$-paths.
The path $p_1$ is irregular and thus non-simplicial, whilst $p_2$ is regular but still non-simplicial.
The path $p_3\defeq aa$ is a sub-path and face of $p_1$ whilst $p_3$ is a face of $p_2$ but not a sub-path.
\end{example}

Next, we describe the notions of a path complex, first introduced in~\cite{Grigoryan2012}, and an ordered simplicial complex.

\begin{defin}\label{def:rpc}
A \mdf{path complex} on $V$ is a set $P$ of elementary paths on $V$, such that
\begin{enumerate}
    \item $P$ contains all singletons, i.e. $v_0 \in P$ for every $v_0 \in V$;
    \item $P$ is closed under truncating paths at either end, i.e.\ if
    $v_0 \dots v_k \in P$ then $v_1 \dots v_k\in P$ and $v_0 \dots v_{k-1}\in P$.
\end{enumerate}
We denote the base vertex set, $\mdf{V(P)} \defeq V$.
We call a path complex \mdf{regular} if all of its constituent paths are regular.
\end{defin}

\begin{defin}\label{def:osc}
An \mdf{ordered simplicial complex} on $V$ is a set $K$ of \textit{simplicial} paths on $V$, such that
\begin{enumerate}
    \item $P$ contains all singletons, i.e. $v_0\in K$ for every $v_0 \in V$;
    \item $P$ is closed under taking ordered subsets, i.e.\ if
    $v_0 \dots v_k \in K$ then for every $i$, $v_0 \dots \hat{v_i} \dots v_k \in K$. 
\end{enumerate}
We denote the base vertex set, $\mdf{V(K)} \defeq V$.
\end{defin}

\begin{rem}
We emphasise that the ordering on the vertices of a path in a simplicial path is \emph{not} inherited from some total order on $V$.
Instead, each simplex carries a total order of its vertices and it is possible for a given subset of vertices to support multiple simplices (each with a different order).
\end{rem}

For brevity, we refer to paths in an ordered simplicial complex as \mdf{simplices} and take `simplicial complex' to mean an ordered simplicial complex.
Clearly any simplicial complex is a regular path complex.
Since we are primarily interested in simplicial complexes, we primarily focus on regular path complexes in this work.
Next, we describe a range of categorical structures for these combinatorial complexes.

\begin{defin}
Given two path complexes, $P_1$ and $P_2$, a map $f: V(P_1) \to V(P_2)$ is a
\begin{enumerate}
    \item \mdf{weak path morphism} if for any $p\in P_1$, $f(p)$ is either irregular or $f(p) \in P_2$;
    \item \mdf{strong path morphism} if for any $p\in P_1$, $f(p) \in P_2$.
\end{enumerate}
We denote the category of all  path complexes with weak path morphisms as \mdf{$\ascat{WkPathC}$}, and with strong path morphisms as \mdf{$\ascat{StPathC}$}.
We denote the full subcategories of all regular path complexes by $\mdf{\ascat{WkRPC}}$ and $\mdf{\ascat{StRPC}}$ respectively.
\end{defin}

\begin{defin}
Given two ordered simplicial complexes, $K_1$ and $K_2$, a map $f: V(K_1) \to V(K_2)$ is a
\begin{enumerate}
    \item \mdf{weak simplicial morphism} if for any $p\in K_1$, $f(p)$ is either non-simplicial or $f(p) \in K_2$;
    \item \mdf{triangle-collapsing simplicial morphism} if $f$ is weak simplicial and furthermore whenever there is a simplex $v_0 v_1 v_2 \in K_1$ such that $f(v_0) = f(v_2)$ then $f(v_0) = f(v_1) = f(v_2)$;
    \item \mdf{strong simplicial morphism} if for any $p \in K_1$, $f(p) \in K_2$.
\end{enumerate}
We denote the category of all simplicial complexes with weak simplicial morphisms as \mdf{$\ascat{WkOSC}$}, with triangle-collapsing simplicial morphisms as \mdf{$\ascat{TcOSC}$}, and with strong simplicial morphisms as \mdf{$\ascat{StOSC}$}.
\end{defin}

\begin{rem}
Given a triangle-collapsing simplicial morphism $f: K_1 \to K_2$ if there is a simplex $p = v_0 \dots v_k$ such that $f(v_i) = f(v_j)$, then \emph{all} the vertices $v_m$, for $i\leq m \leq j$, have the same image under $f$.
\end{rem}

The relationship between strong simplicial and strong path morphism is obvious.

\begin{lemma}
Given two simplicial complexes, $K_1, K_2$, a vertex map $f: V(K_1) \to V(K_2)$ is a strong simplicial morphism if and only if it is a strong path morphism.
Therefore, $\ascat{StOSC}$ is a full subcategory of $\ascat{StRPC}$.
\end{lemma}

A weak simplicial morphism can fail to be a weak path morphism if it maps a simplex to a non-simplicial but regular path.
For example, consider the following simplicial complexes
\begin{equation}
    K_1 = \{ 0, 1, 2, 01, 02, 12, 012 \},\quad
    K_2 = \{ 0, 1, 01, 10 \}.
\end{equation}
Consider the vertex map $f: \{ 0, 1,2 \} \to \{ 0, 1 \}$ which is the identity on $\{0, 1\}$ and maps $2 \mapsto 0$.
It is easy to check that $f$ is a weak simplicial morphism.
However, it is not a weak path morphism because $f(012)=010$ is regular but does not belong to $K_1$.
Thankfully, precluding this scenario is necessary and sufficient to yield a weak path morphism.

\begin{lemma}\label{lem:osc_weak_path_iff_tc}
A weak simplicial morphism $f: K_1 \to K_2$ is a weak path morphism if and only if it is a triangle-collapsing simplicial morphism.
Therefore, $\ascat{TcOSC}$ is the full subcategory of $\ascat{WkRPC}$, with objects restricted to all simplicial complexes.
\end{lemma}
\begin{proof}
Suppose $f$ is triangle-collapsing and take
an arbitrary path $p = v_0 \dots v_k \in K_1$.
Since $f$ is weak simplicial there are two cases.
In the first case $f(p) \in K_2$ and we are done.
In the second case $f(p)$ is non-simplicial, i.e.\ there is some $i < j$ such that $f(v_i) = f(v_j)$.
Since $f$ is triangle collapsing, we must have $f(v_m) = f(v_i)$ for all $i \leq m \leq j$.
In particular, $f(v_{i+1}) = f(v_i)$ and hence $f(p)$ is irregular.
Therefore, $f$ is a weak path morphism.

Now suppose $f$ is not triangle-collapsing.
Hence, there is some simplex $v_0 v_1 v_2 \in K_1$ with $f(v_0) = f(v_2)$ but $f(v_1) \neq f(v_0)$ and $f(v_1) \neq f(v_2)$.
Therefore, $f(p)$ is non-simplicial, and hence not a simplex of $K_2$, but is also regular.
Thus, $f$ is not a weak path morphism.
\end{proof}

\begin{defin}
Given $k\geq 0$ and a path complex $P$, the \mdf{$k$-skeleton of $P$},
is the sub--path complex
\begin{equation}
\mdf{\Sk_k(P)}\defeq \left\{ p \in P \rmv p \text{ is an elementary }l\text{-path for }l\leq k\right\}.
\end{equation}
\end{defin}
\begin{rem}
Any weak path morphism $f: P_1 \to P_2$ induces a weak path morphism $f: \Sk_k(P_1) \to \Sk_k(P_2)$.
Hence, $\Sk_k$ is an endofunctor on $\ascat{WkPathC}$.
This endofunctor restricts to an endofunctor on all the categories introduced in this section.
\end{rem}

\begin{defin}
Given a path complex $P$, removing all irregular paths yields its maximal regular sub--path complex, which we call its \mdf{regularisation}.
This operation constitutes a functor $\mdf{\toreg}:\ascat{WkPathC} \to \ascat{WkRPC}$ and $\mdf{\toreg}:\ascat{StPathC} \to \ascat{StRPC}$.
\end{defin}

\subsection{Regular path complexes from digraphs}\label{sec:background_from_Dgr}

A \mdf{directed graph}, or \mdf{digraph}, is a pair $G=(V, E)$ where $V$ is an arbitrary set and $E \subseteq (V \times V)$.
We call $V$ the set of \mdf{vertices} and $E$ the set of \mdf{(directed) edges}.
We denote the sets of vertices and edges by $\mdf{V(G)}\defeq V$ and $\mdf{E(G) \defeq E}$ respectively.
An edge belonging to $\mdf{\Delta_V} \defeq\left\{ (v, v) \rmv v \in V \right\}$ is called a \mdf{self-loop}.
We call a digraph \mdf{simple} if it contains no self-loops, i.e. $E \subseteq (V\times V)\setminus \Delta_V$.
\textbf{In this work, we assume that all digraphs are simple}.

We write $\mdf{v_0 \to v_1}$ to mean $(v_0, v_1)\in E$ and similarly \mdf{$v_0 \not\to v_1$} to mean $(v_0, v_1)\not\in E$; we write $\mdf{v_0 \tooreq v_1}$ to mean either $v_0 = v_1$ or $v_0 \to v_1$.
Given an edge, $e=(s, t)\in E$, we denote its endpoints by $\mdf{\st(e)} \defeq s$ and $\mdf{\fn(e)} \defeq t$.
We say there is a \mdf{reciprocal edge} on $\{ v_0, v_1\}\subseteq V$ if $v_0 \to v_1$ and $v_1 \to v_0$, and we write \mdf{$v_0 \recip v_1$}.
Finally, if $G$ contains no reciprocal edges then we say $G$ is \mdf{oriented}.

\begin{defin}
Given a directed graph $G$,
a \mdf{directed $k$-clique} is a $k$-tuple of distinct vertices (written $v_0 \dots v_{k-1}$) so that there is an edge $v_i \to v_j$ whenever $i < j$.
The \mdf{directed flag complex of $G$}, \mdf{$\dFl(G)$}, is the simplicial complex on $V$ consisting of all directed cliques in $G$.
\end{defin}

This construction can also be made functorial, but first we must define some categories of digraphs.

\begin{defin}
Given directed graphs $G_1=(V_1, E_1)$ and  $G_2=(V_2, E_2)$, a map $f: V_1 \to V_2$ is a
\begin{enumerate}
    \item \mdf{weak digraph map} if whenever $v \to w$ in $G_1$, either $f(v) = f(w)$ or $f(v) \to f(w)$ in $G_2$;
    \item \mdf{triangle-collapsing digraph map} if $f$ is a weak digraph map and furthermore whenever there is a directed $3$-clique, $v_0 v_1 v_2$, such that $f(v_0) = f(v_2)$, then $f(v_0) = f(v_1) = f(v_2)$;
    \item \mdf{strong digraph map} if whenever $v \to w$ in $G_1$ then $f(v) \to f(w)$ in $G_2$.
\end{enumerate}
We denote the category of all simple digraphs with weak digraph maps as \mdf{$\ascat{WkDgr}$}, with triangle-collapsing digraph maps as \mdf{$\ascat{TcDgr}$} and with strong digraph maps as \mdf{$\ascat{StDgr}$}.
\end{defin}

\begin{lemma}\label{lem:dfl_functoriality}
Given two digraphs $G, H$ and a vertex map $f: V(G) \to V(H)$,
\begin{enumerate}
\item $f$ is a weak digraph map $G \to H$ if and only if it is a weak simplicial morphism $K_G \to \dFl(H)$ 
for some simplicial complex $\Sk_1(\dFl(G)) \subseteq K_G \subseteq \dFl(G)$;
\item $f$ is a strong digraph map $G \to H$ if and only if it is a strong simplicial morphism $K_G \to \dFl(H)$ 
for some simplicial complex $\Sk_1(\dFl(G)) \subseteq K_G \subseteq \dFl(G)$;
\item $f$ is a triangle-collapsing digraph map $G \to H$ if and only if it is a triangle-collapsing simplicial morphism $\dFl(G) \to \dFl(H)$.
\end{enumerate}
In particular, $\dFl$ is a full and faithful functor $\ascat{WkDgr} \to \ascat{WkOSC}$, $\ascat{TcDgr} \to \ascat{TcOSC}$ and $\ascat{StDgr} \to \ascat{StOSC}$.
\end{lemma}
\begin{proof}
Suppose $f$ induces a weak digraph map $G \to H$.
Take any simplex $p = v_0 \dots v_k\in \dFl(G)$.
Then for every $i \leq j$, there is an edge $v_i \to v_j$ in $G$.
Since $f$ is a weak digraph map $f(v_i) \tooreq f(v_j)$ for every $i\leq j$.
Either these are all edges, in which case $f(p) \in \dFl(G)$, or there is an equality in which case $f(p)$ is non-simplicial.
Hence, $f$ induces a weak simplicial morphism $\dFl(G) \to \dFl(H)$ and so certainly induces a weak simplicial morphism $K_G \to \dFl(H)$ for any $\Sk_1(\dFl(G)) \subseteq K_G \subset \dFl(H)$.

Now suppose $f$ induces a weak simplicial morphism $\Sk_1(\dFl(G)) \to \dFl(H)$.
Take an edge $v \to w$ in G.
Note there is a simplex $vw \in \Sk_1(\dFl(G))$.
Therefore, either $f(vw)=f(v)f(w)\in \dFl(H)$, in which case $f(v) \to f(w)$ in $H$, or $f(vw)$ is non-simplicial, in which case $f(v) = f(w)$.
Hence, $f$ induces a weak digraph map $G \to H$.

A similar argument shows the second point and the third point follows immediately from definitions.
\end{proof}

Another way to construct a path complex from a digraph is by considering the collection of all its directed paths.
This is the object studied by path homology~\cite{Grigoryan2012}.

\begin{defin}
Given a directed graph $G=(V, E)$, the \mdf{allowed path complex}, $\mdf{\mathcal{A}(G)}$, is the path complex on $V$
such that given an arbitrary elementary path $p= v_0 \dots v_k$,
\begin{equation}
v_0 \dots v_k \in \mathcal{A}(G) \iff \text{for all }i, \; v_i \to v_{i+1}\text{ in }G.
\end{equation}
We refer to the elements of $\mathcal{A}(G)$ as the \mdf{allowed paths} (in $G$).
\end{defin}

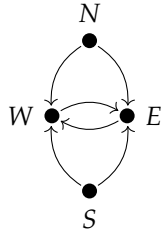
\begin{figure}[htbp]
  \centering
  \begin{tikzpicture}[
  roundnode/.style={circle, fill=black, minimum size=4pt},
	squarenode/.style={fill=black, minimum size=4pt},
	inner sep=2pt,
	outer sep=1pt
  ]

  \node (c) at (3, 1.5) [roundnode, label=left:$W$] {};
  \node (d) at (4, 1.5) [roundnode, label=right:$E$] {};
  \node (e) at (3.5, 2.5) [roundnode, label=above:$N$] {};
  \node (f) at (3.5, 0.5) [roundnode, label=below:$S$] {};

  \draw[->, bend left] (c) to (d);
  \draw[->, bend left] (d) to (c);
  \draw[->, bend right] (e) to (c);
  \draw[->, bend left] (e) to (d);
  \draw[->, bend left] (f) to (c);
  \draw[->, bend right] (f) to (d);

\end{tikzpicture}

   \caption{A simple example digraph $G$ for which $\dFl(G)\neq \mathcal{A}(G)$}\label{fig:dfl_neq_a_example}
\end{figure}

\begin{example}\label{ex:dfl_neq_a}
Figure~\ref{fig:dfl_neq_a_example} shows an example digraph, $G$, for which $\mathcal{A}(G) \neq \dFl(G)$.
First note that $EWE \in \mathcal{A}(G)$, but this path is not simplicial and hence does not belong to $\dFl(G)$.
Also note that $\dFl(G)$ contains two simplices on the nodes $\{N, E, W\}$, namely $NEW$ and $N WE$.
We will see later, in Example~\ref{ex:dfl_neq_a_revisit}, that $\dFl(G)$ has the homology of the $2$-sphere whilst $\mathcal{A}(G)$ is contractible.
\end{example}

Note that $\mathcal{A}(G)$ is always a regular path complex but is not, in general, a simplicial complex.
For example, consider any non-oriented digraph or, more generally, a digraph containing some path $a \to b \to c$ for which $a \not\to c$.
The functoriality of this construction is simple: by a proof similar to that of Lemma~\ref{lem:dfl_functoriality} we obtain the following.

\begin{lemma}\label{lem:dig_map_iff_path_morph}
Given two digraphs $G, H$, and a vertex map $f: V(G) \to V(H)$,
\begin{enumerate}
     \item $f$ is a weak digraph map $G \to H$ if and only if it is a weak path morphism $P_G \to \mathcal{A}(H)$ for some path complex $\Sk_1(\mathcal{A}(G))\subseteq P_G \subseteq \mathcal{A}(G)$;
     \item $f$ is a strong digraph map $G \to H$ if and only if it is a strong path morphism $P_G \to \mathcal{A}(H)$ for some path complex $\Sk_1(\mathcal{A}(G))\subseteq P_G \subseteq \mathcal{A}(G)$.
\end{enumerate}
In particular, $\mathcal{A}$ is a full and faithful functor $\ascat{WkDgr}\to\ascat{WkRPC}$ and $\ascat{StDgr}\to\ascat{StRPC}$.
\end{lemma}

Finally, we note the following relation between these two constructions.

\begin{lemma}\label{lem:subfunctor}
When viewed as functors $\ascat{TcDgr} \to \ascat{WkRPC}$, $\dFl$ is a subfunctor of $\mathcal{A}$.
\end{lemma}
\begin{proof}
Given a digraph $G$, if $p=v_0 \dots v_k$ is a directed $k$-clique then $p$ is certainly an allowed path.
Therefore, $\dFl(G) \subset \mathcal{A}(G)$.
Moreover, given a triangle-collapsing digraph map $f: G \to H$, there are induced maps $\dFl(f): \dFl(G) \to \dFl(H)$ and $\mathcal{A}(f): \mathcal{A}(G) \to \mathcal{A}(H)$ which trivially commute with the inclusions.
\end{proof}

We summarise this section with the diagram of functors shown in Figure~\ref{fig:functor_diagram}.
Each black arrow is an inclusion of subcategories.
The vertical, dashed, black arrows are inclusions of wide subcategories. The horizontal, solid, black arrows are inclusions of full subcategories. Each orange arrow is the directed flag complex functor, $\dFl$, and each blue arrow is the allowed path complex functor, $\mathcal{A}$.
By Lemmas~\ref{lem:dfl_functoriality} and~\ref{lem:dig_map_iff_path_morph}, the coloured arrows are full and faithful functors.
The sub-diagram of black and orange arrows commutes, as does the sub-diagram of black and blue arrows.
Given two paths through the diagram between the same categories, one containing an orange edge and one containing a blue edge, the path containing an orange edge yields a subfunctor of the other.

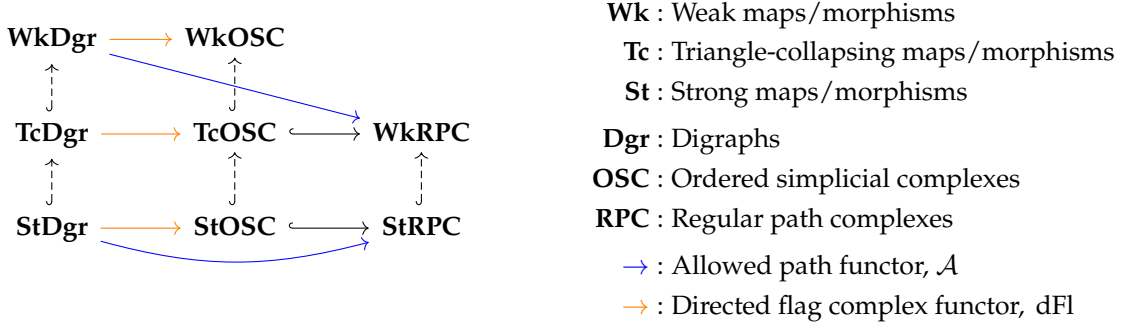
\begin{figure}[hbtp]
\centering
\begin{subfigure}[c]{0.49\textwidth}
\centering
\begin{tikzcd}
    \ascat{WkDgr} \arrow[r, orange] \arrow[rrd, blue] & \ascat{WkOSC}                       & \\
    \ascat{TcDgr} \arrow[r, orange] \arrow[u, dashed, hook] & \ascat{TcOSC} \arrow[u, dashed, hook] \arrow[r, hook] & \ascat{WkRPC}    \\
    \ascat{StDgr} \arrow[r, orange] \arrow[u, dashed, hook] \arrow[rr, blue, bend right=15] & \ascat{StOSC} \arrow[u, dashed, hook] \arrow[r, hook] & \ascat{StRPC} \arrow[u, dashed, hook] 
\end{tikzcd}
\end{subfigure}
\begin{subfigure}[c]{0.49\textwidth}
\begin{align*}
\ascat{Wk}&: \text{Weak maps/morphisms} \\
\ascat{Tc}&: \text{Triangle-collapsing maps/morphisms} \\
\ascat{St}&: \text{Strong maps/morphisms} \\[0.3em]
\ascat{Dgr}&: \text{Digraphs} \\
\ascat{OSC}&: \text{Ordered simplicial complexes} \\
\ascat{RPC}&: \text{Regular path complexes} \\[0.3em]
{\color{blue} \to } &: \text{Allowed path functor, }\mathcal{A}\\
{\color{orange} \to } &: \text{Directed flag complex functor, }\dFl
\end{align*}
\end{subfigure}
\caption{The categories defined in this section, and the functors between them.
The black arrows denote inclusions of subcategories; the dashed arrows are wide inclusions whilst the solid arrows are full inclusions.
The orange arrows denote $\dFl$
and the blue arrows denote $\mathcal{A}$.
}\label{fig:functor_diagram}
\end{figure}

\subsection{Chain complexes from regular path complexes}\label{sec:chain_complexes}
We now construct a chain complex which algebraically represents the combinatorial structure in a regular path complex.
This construction is due to \citeauthor{Grigoryan2012}~\cite{Grigoryan2012}.
It is possible to repeat this construction for any path complex $P$, by first applying the regularisation functor, $\toreg: \ascat{WkPathC} \to \ascat{WkRPC}$.

To begin, we choose an arbitrary field $\Field$ and
let $\mdf{\ascat{Ch}}$ denote the category of non-negatively graded chain complexes of vector spaces over $\Field$.
Given a regular path complex $P$, generate the following free $\Field$-vector spaces:
\begin{align}
\Lambda_k(P) &\defeq \Field \left\langle\{ p=v_0 \dots v_k \rmv p \text{ is an elementary }k\text{-path on }V(P) \}\right\rangle, \\
\mathcal{R}_k(P) &\defeq \Field \left\langle\{ p=v_0 \dots v_k \rmv p\text{ is a regular }k\text{-path on }V(P) \}\right\rangle,\\
\mathcal{I}_k(P) &\defeq \Field \left\langle\{ p=v_0 \dots v_k \rmv p \text{ is an irregular }k\text{-path on }V(P) \}\right\rangle, \\
C_k(P) &\defeq \Field \left\langle\{ p=v_0 \dots v_k \rmv p\in P \}\right\rangle.
\end{align}
We will not distinguish between paths and their corresponding generators in these free groups.
We drop $P$ from notation when it is clear from context.

\begin{rem}
\begin{enumerate}
\item The vector spaces $\Lambda_k(P)$, $\mathcal{R}_k(P)$ and $\mathcal{I}_k(P)$ each depend only on $V(P)$.
\item There is always a direct sum decomposition $\Lambda_k = \mathcal{R}_k \oplus \mathcal{I}_k$.
\item Since we only consider regular path complexes,  there is a subspace relation, $C_k(P) \subseteq \mathcal{R}_k(P)$.
\end{enumerate}
\end{rem}

We can define a \mdf{non-regular boundary map} amongst the $\Lambda_\bullet$ in the usual fashion:
given a standard basis element $v_0 \dots v_k \in \Lambda_k$, define
\begin{equation}
\bd[nreg]_k(v_0 \dots v_k) \defeq \sum_{i=0}^k (-1)^i v_0 \dots \hat{v_i} \dots v_k\label{eq:nonreg_bd}
\end{equation}
and then extend linearly to a map $\bd[nreg]_k : \Lambda_k \to \Lambda_{k-1}$.
As usual, $\bd[nreg]_{k-1} \circ \bd[nreg]_k = 0$~\cite[Lemma~2.3]{Grigoryan2012}.

Next, once can verify that $\bd_k(\mathcal{I}_k) \subseteq \mathcal{I}_{k-1}$~\cite[Lemma~2.9]{Grigoryan2012} and hence $\bd[nreg]_k$ passes to the quotient to define the \mdf{regular boundary map} $\bd_k : \mathcal{R}_k \to \mathcal{R}_{k-1}$.
Intuitively, this boundary map can be computed by following the usual formula (\ref{eq:nonreg_bd}) and then dropping any summands corresponding to irregular paths.
The chain complex equation passes to the quotient to give $\bd_{k-1} \circ \bd_k = 0$.

Unfortunately, this boundary map does not restrict to a map $C_k \to C_{k-1}$ because, a priori, $P$ is not a simplicial complex.
Instead, we must pass to a further subspace
\begin{equation}
    \Omega_k(P) \defeq \left\{ p \in C_k(P) \rmv \bd_k(p) \in C_{k-1}(P) \right\}.
\end{equation}
Notice that the chain complex equation implies that $\bd_k$ restricts to a boundary map $\bd_k : \Omega_k \to \Omega_{k-1}$.
Therefore, ${\{ \Omega_k, \bd_k \}}_{k\geq 0}$ forms a chain complex, which we call the \mdf{regular chain complex of $P$}.

In general, elements of $\Omega_k(P)$ are \textit{sums} of paths in $P$.
Indeed, there may not even be a basis of $\Omega_k(P)$ for which each basis element is supported on a single path.
Computing a basis for $\Omega_k(\mathcal{A}(G))$ is often non-trivial~\cite{Dey2020a}, especially for $k > 2$.

\begin{theorem}[{\cite[Theorem~2.10]{Grigoryan2014}}]
The regular chain complex of a regular path complex can be made into a functor
$\Omega: \ascat{WkRPC} \to \ascat{Ch}$.
\end{theorem}
\begin{proof}
Let $f: P_1 \to P_2$ be a weak path morphism of regular path complexes.
We will define $\inducedch{f}: \Omega(P_1) \to \Omega(P_2)$ on each component.
First we define maps $\mathcal{R}_k(f): \mathcal{R}_k(P_1) \to \mathcal{R}_k(P_2)$ on the standard basis.
Given an arbitrary path $p\in P_1$, let
\begin{equation}
\mathcal{R}_k(f)(p) \defeq
\begin{cases}
f(p) & \text{if }f(p)\text{ is regular,} \\
0    & \text{otherwise.}
\end{cases}
\end{equation}
One can verify that this yields a chain map and moreover, since $f$ is a weak path morphism, $\mathcal{R}_k(f)(\Omega_k(P_1)) \subseteq \Omega_k(P_2)$.
We let $\inducedch{f}$ denote the restriction of $\mathcal{R}_k(f)$ to $\Omega_k(P_1) \to \Omega_k(P_2)$.
The construction of $\mathcal{R}_k(f)$ is clearly functorial and hence so is the restriction.
\end{proof}

\begin{rem}
There is a corresponding non-regular chain complex, which repeats the construction above but with the non-regular boundary map $\bd[nreg]$.
This can distinguish between non-regular path complexes that share the same regularisation.
However, this construction can only be made into a functor $\ascat{StPathC} \to \ascat{Ch}$.
We refer the interested reader to~\cite{GRIGORYAN2019106877}.
\end{rem}

Given a weak path morphism of regular path complexes $f:P_1 \to P_2$, we let $\mdf{\inducedch{f}} \defeq \Omega(f)$ denote the chain map induced by $\Omega$.
We can compose $\Omega$ with the homology functor (in some degree $k$) to obtain a functor
$H_k \circ \Omega : \ascat{WkRPC} \to \ascat{Vec}$.
For brevity's sake, we denote $\mdf{H_k(P)}\defeq H_k(\Omega(P))$ for a regular path complex $P$ and $\mdf{\inducedhom{f}}\defeq H_k(\Omega(f))$ for a weak path morphism $f$.
We also use similar notational shortcuts when composing with the functor which takes homology in every degree simultaneously, viewed as a graded vector space, $H: \ascat{Ch} \to \ascat{grVec}$.
Finally, given an object or morphism in $\ascat{WkPathC}$, we use the same notation to denote these induced objects/morphisms, after first applying the $\toreg$ functor.

In the case where $P$ is in fact a simplicial complex, this construction simplifies substantially.
\begin{lemma}\label{lem:gens_of_osc}
If $P$ is an ordered simplicial complex then $\Omega_k(P) = C_k(P)$ for every $k$ and moreover
\begin{equation}
\bd_k(v_0 \dots v_k) \defeq \sum_{i=0}^k (-1)^i v_0 \dots \hat{v_i} \dots v_k\label{eq:reg_bd_osc}
\end{equation}
for every $v_0 \dots v_k \in P$.
\end{lemma}
\begin{proof}
Given a path $p=v_0 \dots v_k \in P$, all the vertices are distinct because $p$ is simplicial.
Therefore, all the sub-paths $v_0 \dots \hat{v_i} \dots v_k$ are regular and thus equation (\ref{eq:reg_bd_osc}) correctly computes the \textit{regular} boundary map.
Next, all the sub-paths belong to $P$ because $P$ is a simplicial complex and thus $\bd_k(p) \in C_{k-1}(P)$ and moreover $p \in \Omega_k(P)$.
Since $p$ was arbitrary, this shows that $C_k(P) \subseteq \Omega_k(P)$.
\end{proof}

If $P=\mathcal{A}(G)$ or $P=\dFl(G)$ for some digraph $G$, then $H_k(P)$ can be used as a functorial, algebraic invariant of $G$.
For example, $H_0$ encodes the weakly connected components of $G$.

\begin{prop}[{\cite[Proposition~3.24]{Grigoryan2012}}]
Given any digraph $G$,
$H_0(\dFl(G)) = H_0(\mathcal{A}(G))$
and moreover $\dim H_0(\dFl(G))$ is the number of weakly connected components of $G$.
\end{prop}
 \section{Systems of  one-step homotopies}\label{sec:abstract_1sh}

All the homotopies presented herein are generated by a class of directional, one-step homotopies.
The one-step homotopies determine a binary relation on the morphisms, which is subsequently completed to an equivalence relation by identifying `zig-zags' of one-step homotopies.
In this section, we present a small framework for developing these systems of one-step homotopies and relating them via functors.

\begin{defin}
Fix a locally small category $\C$.
\begin{enumerate}
\item A \mdf{system of one-step homotopies} $\Sys$ for $\C$ is a choice of digraph $\Sys(C_1, C_2)$ on the vertex set $\MorXY{\C}{C_1}{C_2}$, for every pair of objects $C_1, C_2\in C$.
\item We call an edge $f \to g$ in $\Sys(C_1, C_2)$ a \mdf{one-step $\Sys$-homotopy from $f$ to $g$}.
A priori, this may not necessarily correspond to a specific morphism in $\C$, but this will often be the vase.
\item If there is an edge $f \to g$ or $g \to f$ in $\Sys(C_1, C_2)$ then we say $f$ and $g$ are \mdf{one-step $\Sys$-homotopic} and write \mdf{$f \simeq_{\Sys, 1} g$}.
\item We say two morphisms $f, g$ are \mdf{$\Sys$-homotopic}, and write \mdf{$f\simeq_{\Sys} g$}, if they belong to the same weak path component of $\Sys(C_1, C_2)$, i.e.\ if there is a sequence of morphisms
\begin{equation}
f = f_0, f_1, \dots, f_{m-1}, f_m=g
\end{equation}
such that $f_{i-1} \simeq_{\Sys, 1} f_i$ for every $i$.
We refer to such a sequence of morphisms as a \mdf{multi-step $\Sys$-homotopy from $f$ to $g$}.
\end{enumerate}
For brevity, we drop the subscript $\Sys$ from the notation when it is clear from context.
\end{defin}

\begin{rem}
By construction, $\simeq_{\Sys}$ is an equivalence relation on $\MorXY{\C}{C_1}{C_2}$.
\end{rem}

A system of one-step homotopies can be pulled back through a functor.
Under sufficient conditions this pull-back can induce `the same' equivalence relation on the morphisms.

\begin{defin}
Given a functor $\mu:\ascat{C}\to \ascat{D}$ and a system of one-step homotopies $\Sys$ for $\D$, the
\mdf{pull-back} is the system of one-step homotopies \mdf{$\mu^\ast \Sys$} for $\C$ given by
\begin{equation}
f \to g\text{ in }\mu^\ast \Sys(C_1, C_2) \iff \mu(f) \tooreq \mu(g) \text{ in }\Sys(\mu(C_1), \mu(C_2)).
\end{equation}
\end{defin}

\begin{rem}
Given a pair of functors $\mu: \C_0 \to \C_1$, $\nu: \C_1 \to \C_2$ and a system of one-step homotopies $\Sys$ for $\C_2$, one can check that $(\nu \circ \mu)^\ast \Sys = \mu^\ast \nu^\ast \Sys$.
\end{rem}

For every pair of objects, $\mu$ defines a weak digraph map $\mu^\ast \Sys(C_1, C_2) \to \Sys(\mu(C_1), \mu(C_2))$.
Indeed, an alternative definition of $\mu^\ast \Sys$ is that $\mu^\ast \Sys(C_1, C_2)$ is the largest digraph on $\MorXY{\C}{C_1}{C_2}$ such that $\mu$ induces such a weak digraph map.
If $\mu$ is faithful then each of these vertex maps is injective and hence forms a strong digraph map.
If $\mu$ is full and faithful then each map is bijective and hence forms an isomorphism of digraphs.

\begin{lemma}\label{lem:system_pullback}
Given a functor $\mu:\ascat{C}\to \ascat{D}$ and a system of one-step homotopies $\Sys$ for $\D$ then
\begin{equation}
f \simeq_{\mu^\ast \Sys} g \implies \mu(f) \simeq_{\Sys} \mu(g).
\end{equation}
If $\mu$ is a full functor then
\begin{equation}
f \simeq_{\mu^\ast \Sys} g \iff \mu(f) \simeq_{\Sys} \mu(g).
\end{equation}
\end{lemma}
\begin{proof}
If $f\simeq_{\mu^\ast \Sys} g$ then there is a sequence of morphisms
\begin{equation}
f = f_0, f_1, \dots, f_{m-1}, f_m=g
\end{equation}
such that $f_{i-1} \to f_i$ in $\mu^\ast \Sys(C_1, C_2)$ for every $i$.
Hence,
\begin{equation}\label{eq:seq_of_mu_morphisms}
\mu(f) = \mu(f_0), \mu(f_1), \dots, \mu(f_{m-1}), \mu(f_m)=\mu(g)
\end{equation}
is a sequence of morphisms, such that for every $i$ either $\mu(f_{i-1}) = \mu(f_i)$ or $\mu(f_{i-1}) \to \mu(f_i)$ in $\Sys(\mu(C_1), \mu(C_2))$.
After removing duplicate morphisms, we see $\mu(f)\simeq_{\Sys} \mu(g)$.

If $\mu$ is full then, by definition, $\mu$ is surjective as a vertex map $\MorXY{\C}{C_1}{C_2} \to \MorXY{\D}{\mu(C_1)}{\mu(C_2)}$.
Hence, if $\mu(f)\simeq_{\Sys} \mu(g)$ then there is a sequence of morphisms as in equation (\ref{eq:seq_of_mu_morphisms}) 
such that $\mu(f_{i-1}) \to \mu(f_i)$ in ${\Sys}(\mu(C_1), \mu(C_2))$ for every $i$.
By the construction of the pull-back
\begin{equation}
f = f_0, f_1, \dots, f_{m-1}, f_m=g
\end{equation}
is a sequence of morphisms showing that $f\simeq_{\mu^\ast\Sys} g$.
\end{proof}

This equivalence relation for morphisms can induce an equivalence relation on the objects of $\C$, in the standard way.
However, we must first check that the relations between the morphisms compose transitively.

\begin{defin}
Let ${\Sys}$ be a system of one-step homotopies for a category $\C$.
We say two objects $C_1, C_2\in\C$ are \mdf{${\Sys}$-homotopy equivalent}, and write \mdf{$C_1 \simeq_{\Sys} C_2$}, if there are a pair of morphisms $f: C_1 \to C_2$ and $g: C_2 \to C_1$ such that
\begin{equation}\label{eq:obj_t_equiv}
g \circ f \simeq_{\Sys} \id_{C_1}
\quad \text{and} \quad
f \circ g \simeq_{\Sys} \id_{C_2}.
\end{equation}
For brevity, we drop the subscript ${\Sys}$ from the notation $\simeq_{\Sys}$ when it is clear from context.
\end{defin}

\begin{defin}
Let ${\Sys}$ be a system of one-step homotopies for a category $\C$.
We say ${\Sys}$ is \mdf{transitive}
if given morphisms $g_1, g_2 \in\MorXY{\C}{C_1}{C_2}$, $f\in \MorXY{\C}{C_0}{C_1}$ and $h\in\MorXY{\C}{C_2}{C_3}$, such that
$g_1 \simeq_{\Sys} g_2$,
then
$g_1 \circ f \simeq_{\Sys} g_2 \circ f$
and
$h \circ g_1 \simeq_{\Sys} h \circ g_2$.
\begin{equation}
\begin{tikzcd}
C_0 \arrow[r, "f"] & C_1 \arrow[r, bend left, "g_1", ""{name=U, below}] \arrow[r, bend right, "g_2"', ""{name=D}] & C_2 \arrow[r, "h"] & C_3 \arrow[phantom, "{\simeq_{\Sys}}"{sloped}, from=U, to=D]
\end{tikzcd}
\end{equation}
\end{defin}

\begin{lemma}\label{lem:system_transitive}
Let ${\Sys}$ be a system of one-step homotopies for a category $\C$.
If ${\Sys}$ is transitive then $\simeq_{\Sys}$ is an equivalence relation on the objects of $\C$.
\end{lemma}
\begin{proof}
It is clear that $\simeq_{\Sys}$ is symmetric and reflexive, so
it remains to show that $\simeq_{\Sys}$ is transitive.
Suppose $C_0\simeq_{\Sys} C_1$ and $C_1\simeq_{\Sys} C_2$, then we have morphisms in the following arrangement
\begin{equation}
\begin{tikzcd}
C_0
\arrow[r, "f_1", bend left, ""{name=U1, below}] &
C_1
\arrow[l, bend left, "g_1", ""{name=D1, below}]
\arrow[r, bend left, "f_2", ""{name=U2, below}] &
C_2
\arrow[l, bend left, "g_2", ""{name=D2, below}]
\end{tikzcd}
\end{equation}
such that each $f_i, g_i$ pair satisfy the equations~(\ref{eq:obj_t_equiv}).
Then
\begin{equation}
(g_1\circ g_2) \circ (f_2 \circ f_1)
= g_1 \circ (g_2 \circ f_2) \circ f_1
\simeq_{\Sys} g_1 \circ \id_{C_1} \circ f_1
= g_1 \circ f_1
\simeq_{\Sys} \id_{C_0},
\end{equation}
where in the first $\simeq_{\Sys}$ step we use the transitivity of ${\Sys}$.
A similar proof shows that the opposite composition is ${\Sys}$-homotopic to $\id_{C_2}$.
Hence, we see $C_0 \simeq C_2$.
\end{proof}

A standard application of Lemmas~\ref{lem:system_pullback} and~\ref{lem:system_transitive} yields the following corollaries.

\begin{cor}\label{cor:pullback_transitivity}
Given a full functor $\mu:\ascat{C} \to \ascat{D}$ and a system of one-step homotopies ${\Sys}$ for $\D$, if ${\Sys}$ is transitive then the pull-back $\mu^\ast {\Sys}$ is transitive.
\end{cor}

\begin{cor}
Given a functor $\mu:\ascat{C} \to \ascat{D}$ and a system of one-step homotopies ${\Sys}$ for $\D$,
\begin{equation}
C_1 \simeq_{\mu^\ast {\Sys}} C_2 \implies \mu(C_1) \simeq_{\Sys} \mu(C_2).
\end{equation}
If $\mu$ is a full functor then
\begin{equation}
C_1 \simeq_{\mu^\ast {\Sys}} C_2 \iff \mu(C_1) \simeq_{\Sys} \mu(C_2).
\end{equation}
\end{cor}

One way to construct a system of one-step homotopies is via a choice of cylinder functor~ (see \cite{kamps1997abstract} for an introduction).

\begin{defin}
A \mdf{cylinder functor} for a category $\C$ is a functor
\begin{equation}
\Cyl: \C \to \C
\end{equation}
equipped with three natural transformations
\begin{equation}
\iota_0, \iota_1 : \id_{\C} \Rightarrow \Cyl,
\quad\text{and}\quad
\rho: \Cyl \Rightarrow \id_{\C},
\end{equation}
satisfying $\rho \iota_i = \id_c$ for each $i=0, 1$.
\end{defin}

With such a cylinder functor in hand, one can construct a system of one-step homotopies $\Sys$ for $\C$ as follows.
Given $f, g: C_1 \to C_2$, a \mdf{one-step ${\Sys}$-homotopy from $f$ to $g$} is a morphism $F: \Cyl(C_1) \to C_2$ such that
\begin{equation}
F \circ \iota_0 = f
\quad\text{and}\quad
F \circ \iota_1 = g.
\end{equation}
This determines the system $\Sys$ by including an arrow $f \to g$ in $\Sys(C_1, C_2)$ whenever such a one-step $\Sys$-homotopy exists.
In a slight abuse of terminology, we use the term `one-step $\Sys$-homotopy' to refer interchangeably between the edge in $\Sys(C_1, C_2)$ and a morphism that implies its existence.
We refer to $\Sys$ as the system of one-step homotopies associated to $\Cyl$.

\begin{lemma}[{\cite[Lemma~2.3]{kamps1997abstract}}]
The system of one-step homotopies associated to a cylinder functor is transitive.
\end{lemma}

In this work, we opt for slightly more generality than cylinder functors in order to emphasise the important properties of the resultant systems.
Moreover, in Section~\ref{sec:dflag_homotopy}, we will see that it is not clear how to construct an appropriate cylinder functor in $\ascat{TcDgr}$ without artificially inflating the category.

 \section{Homotopies for ordered simplicial complexes}\label{sec:big_osc_homotopy_sec}
In this section, we develop a system of one-step homotopies for ordered simplicial complexes, so that two homotopic simplicial morphisms induce identical maps on homology, through the functor $H \circ \Omega: \ascat{TcOSC} \to \ascat{grVec}$.
In Section~\ref{sec:homotopy_theory_pathc}, we recall one such system for $\ascat{WkRPC}$,
which was first defined in~\cite{GRIGORYAN2019106877} in terms of a cylinder functor $\Cyl$ for $\ascat{WkRPC}$.

In Section~\ref{sec:homotopy_theory_osc}, we pull back along the full inclusion $\ascat{TcOSC} \hookrightarrow \ascat{WkRPC}$, to obtain a system for ordered simplicial complexes.
We then describe this pull-back in terms of a cylinder functor $\overline{\Cyl}$ for $\ascat{TcOSC}$.
More precisely, in Theorem~\ref{thm:characterise_osc_homot}, we show that a vertex map $F$ induces a one-step homotopy with respect to $\overline{\Cyl}$ if and only if it induces one with respect to $\Cyl$.

For completeness, in Section~\ref{sec:sset} we instead consider embedding $\ascat{TcOSC}$ in the category of simplicial sets.
We repeat the process above for this inclusion, and show that the system of homotopies we obtain is identical to the one obtained in Section~\ref{sec:homotopy_theory_osc}.

\subsection{Homotopies for path complexes}\label{sec:homotopy_theory_pathc}

\begin{defin}
Given a path complex $P$, the \mdf{cylinder over $P$}, \mdf{$\Cyl(P)$}, is a new path complex on $V(P) \times \{0, 1\}$ containing paths of the following forms
\begin{enumerate}
    \item $(v_0, 0) \dots (v_k, 0)$ such that $v_0 \dots v_k \in P$;
    \item $(v_0, 1) \dots (v_k, 1)$ such that $v_0 \dots v_k \in P$;
    \item $(v_0, 0) \dots (v_i, 0) (v_i, 1) \dots (v_k, 1)$ such that $v_0 \dots v_k \in P$ and $0\leq i \leq k$.
\end{enumerate}
\end{defin}
This can be made into an endofunctor on $\ascat{WkPathC}$ by sending a weak path morphism $f: P_1 \to P_2$ to the weak path morphism $\Cyl(f):\Cyl(P_1) \to \Cyl(P_2)$, given by $(v, i) \mapsto(f(v), i)$.
Moreover, there are weak path morphisms 
\begin{align}
\iota_i&: P \to \Cyl(P), \quad v \mapsto (v, i), \\
\rho   &: \Cyl(P) \to P, \quad (v, i)\mapsto v
\end{align}
for $i=0, 1$.
These formulae determine natural transformations which,
together with $\Cyl$, constitute a cylinder functor for $\ascat{WkPathC}$.
Since the cylinder of a regular path complex is again a regular path complex, this restricts to a cylinder functor on $\ascat{WkRPC}$.
We let $\Sys[\ascat{WkPathC}]$ and $\Sys[\ascat{WkRPC}]$ denote the associated systems of one-step homotopies.
Clearly, $\Sys[\ascat{WkRPC}] = \iota^\ast \Sys[\ascat{WkPathC}]$ where $\iota:\ascat{WkRPC}\hookrightarrow\ascat{WkPathC}$ is the inclusion of full subcategories.

As one might expect, the homology of a path complex is indeed invariant with respect to homotopy equivalence.

\begin{theorem}[{\cite[Theorem~3.8]{GRIGORYAN2019106877}}]\label{thm:homotopy_yields_chain_homotopy}\label{thm:path_complex_homot_invariance}
Suppose $f, g: P_1 \to P_2$ are $\Sys[\ascat{WkRPC}]$-homotopic weak path morphisms of regular path complexes,
 then there is a chain homotopy between $\inducedch{f}$ and $\inducedch{g}$.
Therefore, $f$ and $g$ induce identical maps on homology, $\inducedhom{f}$ and $\inducedhom{g}$.
In particular, if $P_1 \simeq P_2$ then the homology groups $H(P_1)$ and $H(P_2)$ are isomorphic.
\end{theorem}

We automatically inherit a similar result for any pull-back of $\Sys[\ascat{WkRPC}]$.

\begin{cor}\label{cor:pullback_homotopy_induces_identical}
Given a functor $\mu: \C \to \ascat{WkPathC}$, denote the pull-back $T\defeq \mu^\ast \Sys[\ascat{WkPathC}]$.
If $f\simeq_T g$, then $(H\circ \mu)(f)$ and $(H\circ \mu)(g)$ are identical maps on homology.
In particular, if $C_1, C_2\in\C$ are two objects such that $C_1 \simeq_T C_2$ then the homology groups $(H\circ \mu)(C_1)$ and $(H\circ \mu)(C_2)$ are isomorphic.
\end{cor}

\begin{figure}
  \centering
  \begin{tikzpicture}[
  roundnode/.style={circle, fill=black, minimum size=4pt},
	squarenode/.style={fill=black, minimum size=4pt},
	inner sep=2pt,
	outer sep=1pt
  ]

\draw[green!80!black, line width=1mm, opacity=0.25] (-0.5,0.2)--(0,0.2)--(0,1.2)--(2,1.2)--(4,1.2)--(6,1.2)--(8,1.2)--(8.5,1.2);
\draw[red!80!black  , line width=1mm, opacity=0.25] (-0.5,0.1)--(0,0.1)--(2,0.1)--(2,1.1)--(4,1.1)--(6,1.1)--(8,1.1)--(8.5,1.1);
\draw[green!80!black, line width=1mm, opacity=0.25] (-0.5,0)--(0,0)--(2,0)--(4,0)--(4,1)--(6,1)--(8,1)--(8.5,1);
\draw[red!80!black  , line width=1mm, opacity=0.25] (-0.5,-0.1)--(0,-0.1)--(2,-0.1)--(4,-0.1)--(6,-0.1)--(6,0.9)--(8,0.9)--(8.5,0.9);
\draw[green!80!black, line width=1mm, opacity=0.25] (-0.5,-0.2)--(0,-0.2)--(2,-0.2)--(4,-0.2)--(6,-0.2)--(8,-0.2)--(8,0.8)--(8.5,0.8);

\node (a) at (0, 1) {$(v_0, 1)$};
\node (b) at (2, 1) {$(v_1, 1)$};
\node (c) at (4, 1) {$(v_2, 1)$};
\node (d) at (6, 1) {$(v_3, 1)$};
\node (e) at (8, 1) {$(v_4, 1)$};

\node (a2) at (0, 0) {$(v_0, 0)$};
\node (b2) at (2, 0) {$(v_1, 0)$};
\node (c2) at (4, 0) {$(v_2, 0)$};
\node (d2) at (6, 0) {$(v_3, 0)$};
\node (e2) at (8, 0) {$(v_4, 0)$};

\draw[->] (a)--(b);
\draw[->] (b)--(c);
\draw[->] (c)--(d);
\draw[->] (d)--(e);

\draw[->] (a2)--(b2);
\draw[->] (b2)--(c2);
\draw[->] (c2)--(d2);
\draw[->] (d2)--(e2);

\draw[->] (a2)--(a);
\draw[->] (b2)--(b);
\draw[->] (c2)--(c);
\draw[->] (d2)--(d);
\draw[->] (e2)--(e);

\end{tikzpicture}
   \caption{
  Visualisation of the lift $\mathfrak{L}(v_0 v_1 v_2 v_3 v_4)$.
  The lift is a sum of $5$ different $5$-paths, the green paths are added with coefficient $(+1)$ and the red paths are added with coefficient $(-1)$.
  }\label{fig:lifting}
\end{figure}
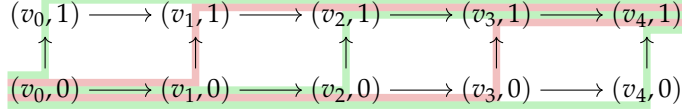

Theorem~\ref{thm:homotopy_yields_chain_homotopy} was shown in~\cite{GRIGORYAN2019106877}, we reproduce the proof here for clarity's seek and to make the chain homotopy explicit in the multi-step case.
Given a regular path complex $P$, the \mdf{lifting map} is a linear map $\mathfrak{L}: \mathcal{R}_k(P) \to \mathcal{R}_{k+1}(\Cyl (P))$, which is given on the standard basis by
\begin{equation}
\mathfrak{L}( v_0 \dots v_k ) \defeq \sum_{i=0}^k 
(-1)^i \cdot (v_0, 0) \dots (v_i, 0) (v_i, 1) \dots (v_k, 1),
\end{equation}
as visualised in Figure~\ref{fig:lifting}.
Note that this clearly restricts to a linear map $\mathfrak{L}: C_k(P) \to C_{k+1}(\Cyl(P))$.
The following result characterises how the lifting map interacts with the regular boundary map.
For a proof, we refer the reader to~\cite{GRIGORYAN2019106877}.

\begin{lemma}[{Product Rule~\cite[Lemma~3.2, Theorem~3.8]{GRIGORYAN2019106877}}]
Given a regular path complex $P$ and $c \in \Omega_k(P)$,
\begin{equation}
\bd \mathfrak{L}(v)  + \mathfrak{L}(\bd c) = \inducedch{(\iota_1)}(v) - \inducedch{(\iota_0)}(c)
\end{equation}
and hence the lifting map restricts to a linear map 
$\mathfrak{L}: \Omega_k(P) \to \Omega_{k+1}(\Cyl (P))$.
\end{lemma}

\begin{proof}[{Proof of Theorem~\ref{thm:homotopy_yields_chain_homotopy}}]
To begin, we deal with the case where $f\simeq_1 g$.
Suppose $F: \Cyl(P_1) \to P_2$ is a one-step homotopy from $f$ to $g$.
We define the chain homotopy $L_k: \Omega_k(P_1) \to \Omega_{k+1}(P_2)$ by $L_k(c) \defeq \inducedch{F}(\mathfrak{L}(c))$.
Using the fact that $\inducedch{F}$ is a chain map, along with the product rule, given any $c \in \Omega_k(P_1)$,
\begin{align}
(\bd L_k + L_{k-1}\bd)(c)
& = \bd \inducedch{F} \mathfrak{L} (c) + \inducedch{F} \mathfrak{L} ( \bd c ) \\
& = \inducedch{F} \bd \mathfrak{L} (c) + \inducedch{F} \mathfrak{L} ( \bd c ) \\ & = \inducedch{F} ( \bd \mathfrak{L} (c) + \mathfrak{L} ( \bd c ) ) \\ & = \inducedch{F} ( \inducedch{(\iota_1)}(c) - \inducedch{(\iota_0)}(c) ) \\ & = \inducedch{g}(c) - \inducedch{f}(c).
\end{align}
Hence, ${\{L_k\}}_{k\geq 0}$ constitutes a chain homotopy between $\inducedch{f}$ and $\inducedch{g}$, as required.

In the general case,
suppose there is a finite sequence of weak path morphisms
\begin{equation}
f = f_0 , f_1, \dots, f_{m-1}, f_m = g
\end{equation}
such that $f_{i-1} \simeq_1 f_{i}$ for $i= 1, \dots, m$.
For each $i$, there is a one-step homotopy $F_i: \Cyl(P_1) \to P_2$ between $f_i$ and $f_{i+1}$.
If $F_i$ is from $f_i$ to $f_{i+1}$, let $\alpha_i =1$.
If $F_i$ is in the opposite direction, let $\alpha_i = -1$.
By the argument above, $F_i$ induces a chain homotopy
${\{L^{(i)}_k: \Omega_k (P_1) \to \Omega_{k+1}(P_2)\}}_{k\geq 0}$
that satisfies
\begin{equation}
\alpha_i \cdot (\bd L^{(i)}_k + L^{(i)}_{k-1}\bd)(c) = \inducedch{(f_{i+1})}(c) - \inducedch{(f_i)}(c)
\end{equation}
for every $c \in \Omega_k(P_1)$.
Therefore, if we define $L_k \defeq \sum_{i=1}^m  \alpha_i \cdot L^{(i)}_k $, then ${\{L_k\}}_{k\geq 0}$ is a chain homotopy between $\inducedch{f}$ and $\inducedch{g}$, as required.
\end{proof}

\subsection{Ordered simplicial complexes as path complexes}\label{sec:homotopy_theory_osc}

Due to Lemma~\ref{lem:osc_weak_path_iff_tc}, there is an inclusion of full subcategories $\ascat{TcOSC}$, $\iota:\ascat{TcOSC}\hookrightarrow\ascat{WkRPC}$.
Therefore, the system of one-step homotopies described in Section~\ref{sec:homotopy_theory_pathc} restricts to a system for ordered simplicial complexes, by taking the pull-back $\iota^\ast \Sys[\ascat{WkRPC}]$.
In this section, we seek to characterise this system intrinsically within $\ascat{TcOSC}$, i.e.\ without reference to path complexes or morphisms. 
First, we require the notion of the simplicial closure of a path complex.

\begin{defin}
A path complex, $P$, is called \mdf{pre-simplicial} if all of its constituent paths are simplicial.
Given a pre-simplicial path complex $P$, its \mdf{simplicial closure}, $\mdf{\overline{P}}$ is the smallest simplicial complex containing $K$.
\end{defin}

\begin{rem}
\begin{enumerate}
\item A pre-simplicial path complex is necessarily regular.
\item The simplicial closure of a pre-simplicial path complex is well-defined because the set of all simplicial complexes is closed under arbitrary intersections.
\end{enumerate}
\end{rem}

It will be important to understand the structure of $\overline{\Cyl(K)}$, for a simplicial complex $K$, because a homotopy of path complexes has $\Cyl(K)$ as its domain.
First we list the additional simplices in the closure of a cylinder.

\begin{lemma}\label{lem:closure_paths}
Given any simplicial complex $K$, $\Cyl(K)$ is pre-simplicial.
Moreover,
\begin{equation}
\overline{\Cyl(K)}\setminus \Cyl(K) 
= \left\{
(v_0, 0) \dots (v_i, 0)(v_{i+1}, 1)(v_k, 1)
\rmv
v_0 \dots v_k \in K,
0 \leq i < k
\right\}.
\end{equation}
\end{lemma}
\begin{proof}
It is clear to see that all paths in the cylinder of $K$ are simplicial.
Then, for every $v_0 \dots v_k \in K$ and $0\leq i < k$, $\Cyl(K)$ contains the path
\begin{equation}
(v_0, 0) \dots (v_i, 0)(v_{i+1}, 0)(v_{i+1}, 1) \dots (v_k, 1).
\end{equation}
Therefore, in order to be closed under taking ordered subsets, $\overline{\Cyl(K)}$ must contain the path
\begin{equation}
(v_0, 0) \dots (v_i, 0)(v_{i+1}, 1) \dots (v_k, 1).
\end{equation}
Finally, note that adding these paths to $\Cyl(K)$ yields a simplicial complex.
\end{proof}

With these, we are able to characterise the homotopies between morphisms in $\ascat{TcOSC}$ without reference to path complexes or morphisms.

\begin{theorem}\label{thm:characterise_osc_homot}
Given two triangle-collapsing simplicial morphisms $f,g: K_1 \to K_2$,
a vertex map $F: V(K_1) \times \{0, 1\} \to V(K_2)$ induces a one-step $\Sys[\ascat{WkRPC}]$-homotopy from $f$ to $g$
if and only if $F$ induces a triangle-collapsing simplicial morphism $\overline{\Cyl(K_1)} \to K_2$ such that
\begin{equation}
F \circ \iota_0 = f
\quad\text{and}\quad
F\circ \iota_1 = g.
\end{equation}
\end{theorem}
\begin{proof}
First, note that it suffices to show $F$ induces a weak path morphism $\Cyl(K_1) \to K_2$ if and only if $F$ induces a triangle-collapsing simplicial morphism $\overline{\Cyl(K_1)} \to K_2$, under the assumption that $F\circ \iota_0 = f$ and $F \circ \iota_1 = g$.
Also note that $f$ and $g$ are weak path morphisms $K_1 \to K_2$, by Lemma~\ref{lem:osc_weak_path_iff_tc}.

One direction is straight-forward: if $F$ induces a triangle-collapsing simplicial morphism $\overline{\Cyl(K_1)} \to K_2$ then, by Lemma~\ref{lem:osc_weak_path_iff_tc}, $F$ induces a weak path morphism $\overline{\Cyl(K_1)} \to K_2$.
Then, because $\Cyl(K_1) \subseteq \overline{\Cyl(K_1)}$, $F$ must also induce a weak path morphism $\Cyl(K_1) \to K_2$.

For the opposite direction, suppose $F$ induces a weak path morphism $\Cyl(K_1) \to K_2$.
It remains to show that for $p\in\overline{\Cyl(K)}\setminus \Cyl(K)$, either $F(p) \in K_2$ or $F(p)$ is irregular.
By Lemma~\ref{lem:closure_paths}, we can assume that $p$ takes the form
\begin{equation}
(v_0, 0) \dots (v_i, 0) (v_{i+1}, 1) \dots (v_k, 1)
\end{equation}
where $v_0 \dots v_k \in K_1$ and $0\leq i < k$.
Consider the cylinder path
\begin{equation}
p_0 \defeq (v_0, 0) \dots (v_i, 0) (v_{i+1}, 0) (v_{i+1}, 1) \dots (v_k, 1) \in \Cyl(K_1).
\end{equation}
Since $F$ induces a weak path morphism $\Cyl(K_1) \to K_2$, either $F(p_0)\in K_2$ or $F(p_0)$ is irregular.
If $F(p_0)\in K_2$ then we are done because $p$ is a face of $p_0$ and $K_2$ is a simplicial complex.
We deal with the remaining case, in which $F(p_0)$ is irregular.
Consider the possible pairs of adjacent vertices in $p_0$ that have the same image under $F$.
In most cases, this irregularity implies that $F(p)$ is also irregular.
The remaining cases are $F(v_i, 0) = F(v_{i+1}, 0)$ and $F(v_{i+1}, 0) = F(v_{i+1}, 1)$.

\textbf{Case 1:} Suppose $F(v_i, 0) = F(v_{i+1}, 0)$.
Then
\begin{align}
F(p)
&= F\big(
(v_0, 0) \dots (v_{i-1}, 0)(v_{i}, 0)(v_{i+1}, 1)(v_{i+2}, 1) \dots (v_k, 1)
\big) \\
 &= F\big(
(v_0, 0) \dots (v_{i-1}, 0)(v_{i+1}, 0)(v_{i+1}, 1)(v_{i+2}, 1) \dots (v_k, 1)
\big)
\in F(\Cyl(K))
\end{align}
and hence either $F(p)\in K_2$ or $F(p)$ is irregular.

\textbf{Case 2:} Suppose $F(v_{i+1}, 0) = F(v_{i+1}, 1)$.
Then
\begin{align}
F(p)
&= F\big(
(v_0, 0) \dots (v_{i-1}, 0)(v_{i}, 0)(v_{i+1}, 1)(v_{i+2}, 1) \dots (v_k, 1)
\big) \\
 &= F\big(
(v_0, 0) \dots (v_{i-1}, 0)(v_{i}, 0)(v_{i+1}, 0)(v_{i+2}, 1) \dots (v_k, 1)
\big).
\end{align}
Iterating this argument, we either fall into a Case 1 (in which case $F(p)\in K_2$ or $F(p)$ is irregular),
or we arrive at Case 2 with $i=k-1$.
This yields $F(v_k, 0) = F(v_k, 1)$ and
\begin{align}
F(p)
&= F\big(
(v_0, 0) \dots (v_{i-1}, 0)(v_{i}, 0)(v_{i+1}, 1)(v_{i+2}, 1) \dots (v_k, 1)
\big) \\
 &= F\big(
(v_0, 0) \dots  (v_{k-1}, 0) (v_k, 1)
\big).
\end{align}
From this we can conclude
\begin{equation}
F(p) = F\big(
  (v_0, 0) \dots  (v_{k-1}, 0) (v_k, 0)
\big)
 = (F \circ \iota_0)(v_0 \dots v_k).
\end{equation}
Then, since $F \circ \iota_0$ is a weak path morphism $K_1 \to K_2$, we obtain that either $F(p)\in K_2$ 
or $F(p)$ is irregular.
\end{proof}

The map $K \mapsto \overline{\Cyl(K)}$ is a functor $\overline{\Cyl}:\ascat{TcOSC}\to\ascat{TcOSC}$.
Together with the usual natural transformations $\iota_i$ and $\rho$, this determines a cylinder functor for $\ascat{TcOSC}$.
We denote the resulting system of one-step homotopies
  by \mdf{$\Sys[\ascat{TcOSC}]$}.
To be explicit, in this system,
a one-step $\Sys[\ascat{TcOSC}]$-homotopy from $f$ to $g$ is a triangle-collapsing simplicial morphism
$F:\overline{\Cyl(K_1)} \to K_2$ such that
\begin{equation}
F \circ \iota_0 = f
\quad\text{and}\quad
F \circ \iota_0 = g.
\end{equation}
Theorem~\ref{thm:characterise_osc_homot} tells us that this is an intrinsic description of the pull-back system .

\begin{cor}\label{cor:pullback_TcOSC_WkRPC}
The pull-back system along the inclusion $\iota: \ascat{TcOSC} \hookrightarrow \ascat{WkRPC}$ is
$\iota^\ast \Sys[\ascat{WkRPC}] = \Sys[\ascat{TcOSC}]$.
\end{cor}

We conclude this subsection with some simple technical results, which control the paths in $\overline{\Cyl(K)}$.
These allow us to obtain an equivalent condition to $F$ inducing a triangle-collapsing simplicial morphism $\overline{\Cyl(K_1)} \to K_2$, which depends only on the 1-skelton, $\Sk_1(K_1)$.
This will be useful in the later sections.

\begin{lemma}\label{lem:closure_results}
Let $K$ be a simplicial complex.
\begin{enumerate}
\item If $(v_0, j) \dots (v_k, j) \in \overline{\Cyl(K)}$ for some fixed $j$ then $v_0 \dots v_k \in K$.\label{itm:closure_result_1}
\item If $(v_0, j_0) \dots (v_k, j_k) \in \overline{\Cyl(K)}$ and $v_0 \dots v_k$ is simplicial then $v_0 \dots v_k \in K$.\label{itm:closure_result_2}
\item If $(v_0, j_0) \dots (v_k, j_k)\in\overline{\Cyl(K)}$ and $v_i = v_j$ for $j> i$ then $j= i+1$.\label{itm:closure_result_3}
\item If $(v_0, j_0)(v_1, j_1)(v_2, j_2)\in\overline{\Cyl(K)}$ then $v_0 \neq v_2$.\label{itm:closure_result_4}
\end{enumerate}
\end{lemma}
\begin{proof}
For the first point, suppose $p_1\defeq (v_0, j) \dots (v_k, j) \in \overline{\Cyl(K)}$ for some fixed $j$.
Note that $p_1$ is not one of the simplices listed in Lemma~\ref{lem:closure_paths}, because the second coordinate is constant.
Therefore, we have $p_1\in\Cyl(K)$.
Considering the three types of paths in $\Cyl(K)$ we see that $v_0 \dots v_k \in K$.

For the second point, suppose
$p_2 \defeq (v_0, j_0) \dots (v_k, j_k) \in \overline{\Cyl(K)}$
and $v_0 \dots v_k$ is simplicial.
If the $j_i$ are constant at $0$ or $1$ then $p_2$ is of the form above and hence $v_0 \dots v_k \in K$.
Otherwise, note that we must have $p_2\in\overline{\Cyl(K)}\setminus\Cyl(K)$ because the remaining paths in $\Cyl(K)$ repeat vertices in the first coordinate.
Hence, by Lemma~\ref{lem:closure_paths}, we see that $v_0 \dots v_k \in K$.

For the third point, suppose
$p_3 \defeq (v_0, j_0) \dots (v_k, j_k)\in \overline{\Cyl(K)}$.
If $v_i = v_j$ for some $i < j$ then $v_0 \dots v_k$ is not simplicial.
Considering the four types of paths in $\overline{\Cyl(K)}$, we see that $p_3$ must take the form
\begin{equation}
p_3 = (w_0, 0) \dots (w_i, 0)(w_i, 1)(w_{i+1}, 1) \dots (w_{k-1}, 1)
\end{equation}
for some $w_0 \dots w_{k-1} \in K$.
Since $w_0 \dots w_{k-1}$ is simplicial, the only repeated vertex is $w_i$, and hence we must have $j=i+1$.
The final point follows immediately from the third.
\end{proof}

\begin{lemma}\label{lem:osc_tpc_equiv_condition}
Given two simplicial complexes $K_1$, $K_2$ and a weak simplicial morphism $F:\overline{\Cyl(K_1)} \to K_2$,
denote $f\defeq F \circ \iota_0$ and $g \defeq F \circ \iota_1$.
If both $f$ and $g$ are both triangle-collapsing simplicial morphisms $K_1 \to K_2$,
then $F$ is triangle collapsing 
if and only if
\begin{equation}\label{eq:osc_tpc_equiv_condition}
xy \in K_1 \text{ and } f(x) = g(y)
\quad\implies\quad
f(x) = f(y) = g(x) = g(y).
\end{equation}
\end{lemma}
\begin{proof}
First, suppose $F$ is triangle-collapsing.
Then take $xy \in K_1$ such that $f(x) = g(y)$.
This implies that $F(x, 0) = F(y, 1)$.
Now note $(x, 0)(y, 0)(y, 1)\in\overline{\Cyl(K_1)}$
so $F(y, 0)=F(x, 0)$ and hence $f(y) = f(x)$.
Likewise, $(x, 0)(x, 1)(y, 1)\in\overline{\Cyl(K_1)}$
so $F(x, 1)=F(x, 0)$ and hence $g(x) = f(x)$.

Now suppose $F$ satisfies condition (\ref{eq:osc_tpc_equiv_condition}).
Take a simplex $p=(v_0, j_0)(v_1, j_1)(v_2, j_2)\in\overline{\Cyl(K_1)}$ such that $F(v_0, j_0) = F(v_2, j_2)$.
By {Lemma~\ref{lem:closure_results}.\ref{itm:closure_result_4}}, we may assume $v_0 \neq v_2$.
If $p = \iota_j(p')$ for some $j \in \{0, 1\}$ and $p' \in K_1$ then $F(v_1, j_1) = F(v_0, j_0)$ because $f$ and $g$ are both triangle-collapsing.
It remains to consider the case where
$p = (v_0, 0) (v_1, j_1) (v_2, 1)$.
Let us assume that $j_1 = 0$; the $j_1=1$ case admits a similar proof.
First, note the condition $F(v_0, j_0) = F(v_2, j_2)$ becomes $f(v_0) = g(v_2)$.
Since $j_1 = 0$ and $p$ is simplicial, we must have $v_1 \neq v_0$.
We split into two cases.

\textbf{Case 1:} Suppose $v_1 = v_2$.
Then $(v_0, 0)(v_1, j_1) = (v_0, 0)(v_2, 0)$ is a face of $p\in\overline{\Cyl(K_1)}$, so
we must have $(v_0, 0)(v_2, 0)\in\overline{\Cyl(K_1)}$.
Therefore, by {Lemma~\ref{lem:closure_results}.\ref{itm:closure_result_1}}, we see $v_0 v_2\in K_1$
and then
condition (\ref{eq:osc_tpc_equiv_condition}) implies that
\begin{equation}
F(v_1, j_1) = f(v_1) = f(v_2) = f(v_0) = F(v_0, j_0).
\end{equation}

\textbf{Case 2:} Suppose $v_1 \neq v_2$.
In this case note that $v_0 v_1 v_2$ is simplicial and hence, by {Lemma~\ref{lem:closure_results}.\ref{itm:closure_result_2}}, $v_0 v_1 v_2 \in K_1$.
Then, since $K_1$ is a simplicial complex, $v_0 v_2 \in K_1$ and hence condition (\ref{eq:osc_tpc_equiv_condition}) implies $g(v_2) = f(v_2)$.
Now $f$ is a triangle-collapsing simplicial morphism and $f(v_0) = g(v_2) = f(v_2)$ so $f(v_1) = f(v_0)$ which implies $F(v_1, j_1) = F(v_0, 0)$.
\end{proof}

\subsection{Ordered simplicial complexes as simplicial sets}\label{sec:sset}

Those unfamiliar with path complexes may prefer to view an ordered simplicial complex as a simplicial set.
For completeness, we consider this alternative perspective here.
The remainder of this work is not dependent on this section.

First, we describe a full and faithful embedding $\sigma:\ascat{TcOSC} \to \ascat{sSet}$, into the category of simplicial sets.
To each simplicial set, one can assign a chain complex, called its normalised Moore complex, via a functor $N:\ascat{sSet} \to \ascat{Ch}$ (see~\cite[\S~2]{Goerss2009Chapter3}).
From a homological point of view, these two perspectives are identical, in the sense that $\Omega$ and $N \circ \sigma$ are naturally isomorphic as functors $\ascat{TcOSC} \to \ascat{Ch}$.
The category of simplicial sets is equipped with a system of one-step homotopies, via its `left homotopies', which we denote $\Sys[\ascat{sSet}]$.
In Corollary~\ref{cor:sset_homotopies_same}, we show that $\sigma^\ast \Sys[\ascat{sSet}] = \Sys[\ascat{TcOSC}]$ and hence these two perspectives yield identical systems of one-step homotopies.
We begin with definitions of the simplex category and of simplicial sets, which facilitate the construction of $\sigma$.

\begin{defin}
Given an integer $n\geq 0$, the \mdf{$n$\textsuperscript{th} ordinal} is the ordered simplicial complex, \mdf{$[n]$}, on the vertex set $\{0, \dots, n\}$ consisting of \emph{all} strictly increasing sequences.
Put another way, $[n]$ is the minimal simplicial complex containing the simplex $0 \dots n$.
The \mdf{simplex category}, \mdf{$\SmplxCat$}, is the full subcategory of $\ascat{TcOSC}$, with objects restricted to the ordinals.
\end{defin}

\begin{rem}
Let $G_n$ be the digraph on $\{0, \dots, n\}$ with an edge $i \to j$ if and only if $i < j$.
Then notice $\mathcal{A}(G_n) = \dFl(G_n) = [n]$.
By Lemmas~\ref{lem:osc_weak_path_iff_tc} and~\ref{lem:dig_map_iff_path_morph}, a morphism $[m] \to [n]$ in $\SmplxCat$ is just a weak digraph map $G_m \to G_n$.
This is equivalent to a (not necessarily strictly) increasing function $\{0, \dots, m\} \to \{0, \dots, n\}$.

\end{rem}

\begin{defin}
A \mdf{simplicial set} is a contravariant functor from the simplex category into the category of sets, i.e.\ a functor $\opcat{\SmplxCat} \to \ascat{Set}$.
A \mdf{simplicial map} is a morphism in the functor category $\Funct{\opcat{\SmplxCat}}{\ascat{Set}}$, i.e.\ a natural transformation.
This defines the \mdf{category of simplicial sets}, which we denote \mdf{$\ascat{sSet}$}.
\end{defin}

\begin{lemma}\label{lem:sset_embed}
There is a full and faithful embedding $\sigma: \ascat{TcOSC} \to \ascat{sSet}$.
\end{lemma}
\begin{proof}
There is an inclusion of subcategories $\iota_{\SmplxCat}:\SmplxCat \to \ascat{TcOSC}$.
We define $\sigma$ to be the restricted Yoneda embedding corresponding to this inclusion, i.e.\
$ \sigma \defeq \iota_{\SmplxCat}^{\ast} \circ \yo$,
where $\iota_{\SmplxCat}^{\ast}$ is the restriction functor
$\iota_{\SmplxCat}^{\ast}: \Funct{\opcat{\ascat{TcOSC}}}{\ascat{Set}} \to \Funct{\opcat{\SmplxCat}}{\ascat{Set}}$
and $\yo: \ascat{TcOSC} \to \Funct{\opcat{\ascat{TcOSC}}}{\ascat{Set}}$ is the Yoneda embedding.
More concretely, given an ordinal $[n] \in \SmplxCat$,
\begin{equation}
\sigma(K)([n]) = \MorXY{\ascat{TcOSC}}{[n]}{K}.
\end{equation}

The proof that $\sigma$ is full and faithful follows the usual proof that $\SmplxCat$ is dense in the category of small categories.
Essentially, this amounts to showing that,
given simplicial complexes $K_1, K_2$ and a natural transformation $\eta: \sigma(K_1) \Rightarrow \sigma(K_2)$, there is a unique vertex map $f: V(K_1) \to V(K_2)$ such that $f$ induces a triangle-collapsing simplicial morphism $K_1 \to K_2$ and $\sigma(f) = \eta$.
This vertex map is fully determined by the function
\begin{equation}
\eta_{[0]}: \MorXY{\ascat{TcOSC}}{[0]}{K_1} \to \MorXY{\ascat{TcOSC}}{[0]}{K_2},
\end{equation}
which induces to a triangle-collapsing simplicial morphism precisely because $\eta$ is natural.
\end{proof}

We now recall the homotopies in $\ascat{sSet}$ and explore how they relate to the homotopies in $\ascat{TcOSC}$.
First, we need the notion of product for a general category.

\begin{defin}
Given two objects $G, H\in\C$ in some category $\C$, the \mdf{categorical product} is an object \mdf{$G\times_{\C}H$} together with morphisms $\pi_G: G\times_{\C}H \to G$ and $\pi_H : G\times_{\C}H \to H$ satisfying the following universal property:

\textbf{(Universal Property of Products)} Given $\phi_G: Y \to G$ and $\phi_H: Y \to H$, there is a unique morphism $f: Y \to G \times_{\C} H$ such that the following diagram commutes
\begin{equation}
\begin{tikzcd}
 & Y \arrow[dl, "\phi_G"'] \arrow[dr, "\phi_H"] \arrow[d, dashed, "f"] & \\
G & G\times_{\C}H \arrow[l, "\pi_G"] \arrow[r, "\pi_H"'] & H
\end{tikzcd}
\end{equation}
\end{defin}

\begin{rem}
The categorical product may not exist but, if it does, it is unique up to unique isomorphism.
\end{rem}

\begin{defin}
Given a simplicial set $S\in\ascat{sSet}$, the \mdf{cylinder over $S$} is the simplicial set
\begin{equation}
\mdf{\Cyl(S)} \defeq S \times_{\ascat{sSet}} \sigma([1]).
\end{equation}
\end{defin}

This construction can clearly be made into a functor $\Cyl: \ascat{sSet} \to \ascat{sSet}$.
Since $\ascat{sSet}$ is a category of functors into $\ascat{Set}$ which is complete, this product exists and is computed `pointwise'.
That is,
given an ordinal $[n]\in\SmplxCat$,
\begin{equation}
\Cyl(S)[n] = S[n] \times \sigma([1]) = S[n] \times \MorXY{\ascat{\SmplxCat}}{[n]}{[1]},
\end{equation}
where $\times$ here denotes the usual Cartesian product of sets.
This formula also extends to describe the structure morphisms, $\Cyl(S)(\theta)$, for each $\theta:[m] \to [n]$ in $\SmplxCat$.
There are natural transformations
$e_0, e_1: \id_{\ascat{sSet}} \Rightarrow \Cyl$
and
$\rho: \Cyl \Rightarrow \id_{\ascat{sSet}}$.
On each component $S\in\ascat{sSet}$, these are given on the ordinals $[n]\in\SmplxCat$ by
\begin{equation}
{(e_i)}_{S}[n] : S[n] \to \Cyl(S)[n], \quad \tau \mapsto (\tau, c_i)
\end{equation}
and
\begin{equation}
(\rho)_{S}[n] : \Cyl(S)[n] \to S[n], \quad (\tau, i) \mapsto \tau,
\end{equation}
where
$c_i\in\MorXY{\ascat{TcOSC}}{[n]}{[1]}$ denotes the constant morphism which sends all vertices to $i$.
These natural transformations compose as follows $\rho \circ e_o = \rho \circ e_1 = \id_{\ascat{sSet}}$ which
gives $\Cyl$ the structure of a cylinder functor.
We denote the associated system of one-step homotopies by $\Sys[\ascat{sSet}]$.
Again, to be explicit, given simplicial maps $f, g: S_1 \to S_2$ a one-step $\Sys[\ascat{sSet}]$-homotopy from $f$ to $g$ is a simplicial map
$F: \Cyl(S_1) \to S_2$ such that
\begin{equation}
F \circ e_0 = f
\quad\text{and}\quad
F \circ e_1 = g.
\end{equation}
We refer the interested reader to~\cite[\S~1.6]{Goerss2009Chapter1} for a more thorough treatment of simplicial homotopies.
As with $\Sys[\ascat{WkRPC}]$, homotopic simplicial maps induce identical maps on homology.

\begin{theorem}[{\cite[Theorem~2.4]{Goerss2009Chapter3}}]
Suppose $f, g: S_1 \to S_2$ are $\Sys[\ascat{sSet}]$-homotopic simplicial maps
 then there is a chain homotopy between $N(f)$ and $N(g)$.
Therefore, $f$ and $g$ induce identical maps on homology.
In particular, if $S_1 \simeq S_2$ then the homology groups $(H \circ N)(S_1)$ and $(H \circ N)(S_2)$ are isomorphic.
\end{theorem}

\begin{lemma}\label{lem:sset_cyls}
For any simplicial complex $K$, $\Cyl(\sigma(K)) = \sigma(\overline{\Cyl(K)})$ (up to a natural isomorphism).
Moreover, as natural transformations
$
\sigma \iota_i = e_i \sigma
$.
\end{lemma}
\begin{proof}
Given an ordinal $[n]$,
\begin{equation}
\Cyl(\sigma(K))[n] = \MorXY{\ascat{TcOSC}}{[n]}{K} \times \MorXY{\ascat{TcOSC}}{[n]}{[1]}
\end{equation}
and
\begin{equation}
\sigma(\overline{\Cyl(K)})[n] = \MorXY{\ascat{TcOSC}}{[n]}{\overline{\Cyl(K)}}.
\end{equation}
There are triangle-collapsing simplicial morphisms $\pi_K : \overline{\Cyl(K)} \to K$ and $\pi_{[1]}: \overline{\Cyl(K)} \to [1]$, given by projection to each of the coordinates of the vertex set.
Composition with these projections induces a simplicial map $\phi: \sigma(\overline{\Cyl(K)}) \to \Cyl(\sigma(K))$.
At each ordinal $[n]$, this map is given by
\begin{align}
\phi_{[n]}: \MorXY{\ascat{TcOSC}}{[n]}{\overline{\Cyl(K)}} & \to
\MorXY{\ascat{TcOSC}}{[n]}{K} \times \MorXY{\ascat{TcOSC}}{[n]}{[1]} \\
p &\mapsto (\pi_K\circ p, \pi_{[1]} \circ p).
\end{align}
By Lemma~\ref{lem:closure_paths}, each $\phi_{[n]}$ is a bijection and
hence $\phi$ is an isomorphism of simplicial sets.
Moreover, the formula is natural in $K$ and hence this is a natural isomorphism between the two functors $\ascat{TcOSC} \to \ascat{sSet}$.
Finally, the equation $\sigma \iota_i = e_i \sigma$ 
follows automatically from unwinding the definitions, modulo the natural isomorphism $\phi$. 
\end{proof}

\begin{cor}\label{cor:sset_homotopies_same}
Given two triangle-collapsing simplicial morphisms $f, g: K_1 \to K_2$,
\begin{enumerate}
\item if $F: \overline{\Cyl(K_1)} \to K_2$ is a one-step homotopy from $f$ to $g$ then
$\sigma(F)$ is a one-step homotopy from $\sigma(f)$ to $\sigma(g)$.
\item if $F: \Cyl(\sigma(K_1)) \to \sigma(K_2)$ is a one-step homotopy from $\sigma(f)$ to $\sigma(g)$ then $\sigma^{-1}(F)$ is a one-step homotopy from $f$ to $g$.
\end{enumerate}
In particular, the pull-back along $\sigma$ is $\sigma^\ast \Sys[\ascat{sSet}] = \Sys[\ascat{TcOSC}]$.
\end{cor}
\begin{proof}
The first point follows immediately from Lemma~\ref{lem:sset_cyls}.
For the second point, up to a unique isomorphism $F$ is a morphism $\sigma(\overline{\Cyl(K_1)}) \to \sigma(K_2)$.
By Lemma~\ref{lem:sset_embed}, $\sigma$ is full and faithful, so there is a unique morphism $G: \overline{\Cyl(K_1)} \to \Cyl(K_2)$ in $\ascat{TcOSC}$ such that $\sigma(G) = F$.
Since $\sigma$ is faithful, the equations
\begin{equation}
F \circ e_0(\sigma(K_1)) = \sigma(G) \circ \sigma(\iota_0) = \sigma(f)
\quad\text{and}\quad
F \circ e_1(\sigma(K_2)) = \sigma(G) \circ \sigma(\iota_1) = \sigma(g)
\end{equation}
imply
$G \circ \iota_0 = f$
and
$G \circ \iota_1 = g$.
\end{proof}

For the interested reader, these two perspectives can be unified by further embedding $\ascat{sSet}$ and $\ascat{WkRPC}$ in the category of path pairs of sets, $\ascat{pSetPair}$, which was introduced in~\cite{PathSetPaper}.

\begin{defin}
A \mdf{path set} is a functor $\opcat{\Pi} \to \ascat{Set}$ where $\Pi$ is the full subcategory of $\ascat{WkRPC}$, restricted to the objects $\{P_n\}_{n\geq 0}$ in which $P_n$ is the minimal path complex containing the path $0 \dots n$.
Any simplicial set restricts to a path set by composition with the inclusions $P_n \hookrightarrow [n]$.
A \mdf{path pair of sets} is a pair $(S, P)$ in which $S$ is a simplicial set and $P$ is a path subset of $S$.
A morphism $(S_1, P_1)\to (S_2, P_2)$ in $\ascat{pSetPair}$ is simply a simplicial map $f: S_1 \to S_2$ such that $f(P_1) \subseteq P_2$.
\end{defin}

By~\cite[Proposition~9.6]{PathSetPaper}, there is a natural isomorphism $\phi$ such that the functors $N$ and $\phi \circ \Omega$ factor through $\ascat{pSetPair}$, making the diagram of functors~(\ref{eq:pSet_Ch_diagram}) commute.
To be explicit, the functor $\iota: \ascat{sSet} \to \ascat{pSetPair}$ is given by the assignment $S \mapsto (S, S)$ whilst
the functor $\sigma: \ascat{WkRPC} \to \ascat{pSetPair}$ can be defined as
\begin{equation}
\sigma(P)\defeq
\big(
 \MorXY{\ascat{WkRPC}}{-}{\mathcal{R}(P)}
 ,\;
 \MorXY{\ascat{WkRPC}}{-}{P}
\big) \in \Funct{\opcat{\SmplxCat}}{\ascat{Set}} \times \Funct{\opcat{\Pi}}{\ascat{Set}},
\end{equation}
where $\mathcal{R}(P)$ here denotes the path complex of all regular paths on $V(P)$.
Note that $\iota \sigma \neq \sigma \iota$ because the simplicial set component of $\sigma\iota(K)$ is, in general, significantly larger than $\sigma(K)$.

\begin{equation}\label{eq:pSet_Ch_diagram}
\begin{tikzcd}
& \ascat{WkRPC} \arrow[r, hook, "\sigma"] \arrow[rrd, "\phi \circ \Omega"'] & \ascat{pSetPair} \arrow[rd, "\omega"] & \\
\ascat{TcOSC} \arrow[ru, hook, "\iota"] \arrow[rd, hook, "\sigma"'] & & & \ascat{Ch} \\
& \ascat{sSet} \arrow[r, hook, "\iota"'] \arrow[rru, "N"] & \ascat{pSetPair} \arrow[ru, "\omega"'] &
\end{tikzcd}
\end{equation}

The category $\ascat{pSetPair}$ comes equipped with a cylinder functor~\cite[\S~9.1]{PathSetPaper} which induces an associated system of one-step homotopies, \mdf{$\Sys[\ascat{pSetPair}]$}.
This cylinder functor is derived from a box product for $\ascat{pSetPair}$.
The cylinder functor for $\ascat{WkRPC}$ can also be defined in terms of a box product for $\ascat{pSetPair}$ and
these two box products commute with $\sigma$ by~\cite[Proposition~9.7]{PathSetPaper}.
Hence,
$\sigma^\ast \Sys[\ascat{pSetPair}] = \Sys[\ascat{WkRPC}]$.
Moreover, for a path pair of sets $(S, P)$, the simplicial set component of the cylinder over $(S, P)$ is $\Cyl(S)$ and hence $\Sys[\ascat{pSetPair}] = \Sys[\ascat{sSet}]$.

\subsubsection{Other viewpoints}

Since an ordered simplicial complex contains no `degenerate' simplices, one may prefer to view an ordered simplicial complex as a semi-simplicial set (sometimes called a $\Delta$-set).
These objects are the combinatorial data underpinning what~\citeauthor{HatcherAllen2002At} calls a $\Delta$-complex structure~\cite{HatcherAllen2002At}.
One can factor the functor $\sigma:\ascat{TcOSC} \to \ascat{sSet}$ through the category of semi-simplicial sets, by composing with the functor which freely adds in all degeneracies to a semi-simplicial set (the functor $G$ in~\cite[Lemma~1.1]{Rouke1971}).
We choose to focus here on simplicial sets since they are better-studied in the literature and the two perspectives are compatible (see~\cite{Rouke1971}).

Finally, we note that there is a notion of homotopy theory for directed topological spaces, from the field of directed algebraic topology~\cite[\S~4.2]{fajstrup2016directed}.
To the authors' knowledge, this theory has not yet been connected to the various homotopy and homology theories of directed graphs.
One could choose instead to view the directed flag complex as a directed topological space; this perspective requires further study.

 \section{Homotopies for digraphs}\label{sec:homotopy_theory}
The over-arching goal for this section is to develop a system of one-step homotopies for $\ascat{TcDgr}$, to which the homology of the directed flag complex is invariant.
In Section~\ref{sec:homotopy_theory_allowed_paths}, we pull back along the allowed path complex functor $\mathcal{A}:\ascat{WkDgr} \to \ascat{WkRPC}$.
We find that the resulting system, $\mathcal{A}^\ast \Sys[\ascat{WkRPC}]$ coincides with the `path homotopy' theory for digraphs, first introduced in~\cite{Grigoryan2014}.
In particular, the one-step homotopies can be characterised in terms of a cylinder functor for $\ascat{WkDgr}$.
This coincidence is unsurprising since the path homology of digraphs was first defined as a functor which factors through $\ascat{WkRPC}$~\cite{Grigoryan2012}.
However, this connection motivates a similar study of the directed flag complex.

In Section~\ref{sec:dflag_homotopy}, we repeat this process for the directed flag complex functor, $\dFl: \ascat{TcDgr} \to \ascat{TcOSC}$;
this yields the desired system for $\ascat{TcDgr}$.
Unlike in Section~\ref{sec:homotopy_theory_allowed_paths}, there exist one-step $\Sys[\ascat{TcOSC}]$-homotopies which cannot be realised as the image of some morphism in $\ascat{TcDgr}$, through $\dFl$.
However, in Corollary~\ref{cor:dfl_f_to_g}, we obtain simple, edge-based conditions for the existence of these one-step $\Sys[\ascat{TcOSC}]$-homotopies and hence characterise $\dFl^\ast \Sys[\ascat{TcOSC}]$.

\subsection{Pull-back along the allowed path functor}\label{sec:homotopy_theory_allowed_paths}

Consider the allowed path functor, $\mathcal{A}:\ascat{WkDgr} \to \ascat{WkRPC}$.
We can characterise the pull-back system $\mathcal{A}^\ast \Sys[\ascat{WkRPC}]$ intrinsically within the category $\ascat{WkDgr}$.
First, in order to construct an appropriate cylinder functor, we must define sensible notions of intervals and products in $\ascat{WkDgr}$.

\begin{defin}
The \mdf{standard unit interval} is the digraph ${\unitint}$ on nodes $\{0, 1\}$ whose only edge is $0 \to 1$.
\end{defin}

\begin{defin}
Given two digraphs $G, H$, the \mdf{box product}, \mdf{$G\boxdot H$} is a digraph on $V(G) \times V(H)$ such that
\begin{equation}
(x, y) \to (x', y') \iff (x = x \text{ and } y \to y') \text{ or } (x \to x' \text{ and } y = y' ).
\end{equation}
\end{defin}

Taking the box product with the standard unit interval, gives us a `cylinder' digraph.
As one would hope, it is straightforward to show that $\mathcal{A}$ commutes with the relevant cylinder operations.

\begin{lemma}\label{lem:a_cyl_commutes}
For any digraph $G$, $\mathcal{A}(G \boxdot {\unitint}) = \Cyl(\mathcal{A}(G))$.
\end{lemma}

Moreover, there are weak digraph maps $\iota_i: G\boxdot {\unitint} \hookrightarrow G$ given by $\iota_i(v) \defeq (v, i)$.
Clearly $\mathcal{A}(\iota_i)$ is the component $(\iota_i)_G$ of the natural transformation $\iota_i$, associated to the cylinder functor for $\ascat{WkRPC}$.
This allows us to easily characterise when a vertex map constitutes a one-step homotopy.

\begin{cor}\label{cor:characterise_ahomot}
Given two weak digraph maps $f, g: G \to H$,
a vertex map $F: V(G) \times \{0, 1\} \to V(H)$ is a one-step homotopy from $\mathcal{A}(f)$ to $\mathcal{A}(g)$
if and only if it induces a weak digraph map $F: G \boxdot {\unitint} \to H$ such that
\begin{equation}
F \circ \iota_0 = f
\quad\text{and}\quad
F \circ \iota_1 = g.
\end{equation}
\end{cor}
\begin{proof}
The vertex map $F$ is a weak digraph map $G \boxdot {\unitint} \to H$ if and only if it is a weak path morphism $\mathcal{A}(G \boxdot {\unitint}) \to \mathcal{A}(H)$, by Lemma~\ref{lem:dig_map_iff_path_morph}.
By Lemma~\ref{lem:a_cyl_commutes}, this is equivalent to $F$ forming a weak path morphism $\Cyl(\mathcal{A}(G)) \to \mathcal{A}(H)$.
\end{proof}

We are led to the definition of a functor $\mdf{\Cyl}:\ascat{WkDgr} \to \ascat{WkDgr}$, given by
$\mdf{\Cyl(G)}\defeq G \boxdot {\unitint}$.
The inclusions $\iota_i$ are in fact natural transformation $\id_{\ascat{WkDgr}} \Rightarrow \Cyl$ and the obvious projection map $\Cyl(G) \to G$ induces a natural transformation $\rho:\Cyl \Rightarrow \id_{\ascat{WkDgr}}$.
Together, this gives $\Cyl$ the structure of a cylinder functor and determines an associated system of one-step homotopies for $\ascat{WkDgr}$, which we denote \mdf{$\Sys[\mathcal{A}]$}.
For brevity, we use \mdf{$\simeq_{\mathcal{A}}$}, in lieu of $\simeq_{\Sys[\mathcal{A}]}$, to denote the resulting equivalence relations.
In this system, given two weak digraph maps $f, g: G \to H$, a \mdf{one-step $\Sys[\mathcal{A}]$-homotopy from $f$ to $g$} is a weak digraph map $F: G\boxdot {\unitint} \to H$ such that
\begin{equation}
F \circ \iota_0 = f
\quad\text{and}\quad
F \circ \iota_1 = g.
\end{equation}

\begin{cor}
The pull-back system along $\mathcal{A}:\ascat{WkDgr} \to \ascat{WkRPC}$ is
$\mathcal{A}^\ast \Sys[\ascat{WkRPC}] = \Sys[\mathcal{A}]$.
\end{cor}

A one-step $\Sys[\mathcal{A}]$-homotopy is precisely the definition of a one-step digraph homotopy, in the sense introduced in~\cite{Grigoryan2014}.
Hence, $\mathcal{A}^{\ast} \Sys[\ascat{WkRPC}]$ coincides with the system introduced in~\cite{Grigoryan2014},
and enjoys a particularly simple characterisation of its one-step homotopies.

\begin{cor}[{\cite[\S~3]{Grigoryan2014}}]\label{cor:exists_a_eq_f_to_g}
Given two weak digraph maps $f, g: G \to H$, there is a one-step $\Sys[\mathcal{A}]$-homotopy from $f$ to $g$ if and only if for every $x\in V(G)$, $f(x) \tooreq g(x)$.
\end{cor}

\begin{rem}
The category $\ascat{WkDgr}$ admits a cofibration category structure in which the weak equivalences are those digraph maps which induce isomorphisms on path homology, $(H \circ \mathcal{A})(G)$ (see~\cite{Carranza2024}).
This structure can be refined so that the weak equivalences are those maps which induce an isomorphism on the entire second page of the magnitude path spectral sequence~\cite{hepworth2024bigraded}.
\end{rem}

\subsection{Pull-back along the directed flag complex functor}\label{sec:dflag_homotopy}

We now consider the directed flag complex functor $\dFl: \ascat{TcDgr} \to \ascat{TcOSC}$; we seek an intrinsic characterisation of $\dFl^\ast \Sys[\ascat{TcOSC}]$.
As before, we require a product structure for $\ascat{TcDgr}$.
Note that, if $v_0 \to v_1$ is an edge in $G$ then $(v_0, 0)(v_1, 1)\not\in\dFl(G\boxdot {\unitint})$ and so $\overline{\Cyl(\dFl(G))}\not\subseteq \dFl(G\boxdot {\unitint})$.
Hence, we must consider a larger product graph, to ensure that all the necessary directed cliques are present.

\begin{defin}
Given two digraphs $G, H$, the \mdf{cross product}, \mdf{$G\times H$} is a digraph on $V(G) \times V(H)$ such that
\begin{equation}
  (x, y) \to (x', y') \iff x \tooreq x' \text{ and } y \tooreq y'\text{ but not }(x=x'\text{ and }y=y').
\end{equation}
\end{defin}

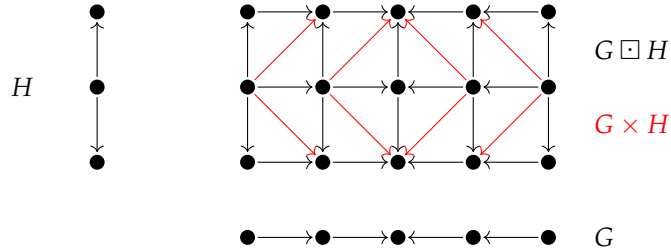
\begin{figure}[hptb]
  \centering
  \begin{tikzpicture}[
  roundnode/.style={circle, fill=black, minimum size=4pt},
	squarenode/.style={fill=black, minimum size=4pt},
	inner sep=2pt,
	outer sep=1pt
  ]

\node (a) at (0, 0) [roundnode]  {};
\node (b) at (1, 0) [roundnode]  {};
\node (c) at (2, 0) [roundnode]  {};
\node (d) at (3, 0) [roundnode]  {};
\node (e) at (4, 0) [roundnode]  {};

\node (a2) at (0, 1) [roundnode] {};
\node (b2) at (1, 1) [roundnode] {};
\node (c2) at (2, 1) [roundnode] {};
\node (d2) at (3, 1) [roundnode] {};
\node (e2) at (4, 1) [roundnode] {};

\node (a3) at (0, 2) [roundnode] {};
\node (b3) at (1, 2) [roundnode] {};
\node (c3) at (2, 2) [roundnode] {};
\node (d3) at (3, 2) [roundnode] {};
\node (e3) at (4, 2) [roundnode] {};

\draw[->] (a)--(b);
\draw[->] (b)--(c);
\draw[->] (d)--(c);
\draw[->] (e)--(d);

\draw[->] (a2)--(b2);
\draw[->] (b2)--(c2);
\draw[->] (d2)--(c2);
\draw[->] (e2)--(d2);

\draw[->] (a3)--(b3);
\draw[->] (b3)--(c3);
\draw[->] (d3)--(c3);
\draw[->] (e3)--(d3);

\draw[->] (a2)--(a);
\draw[->] (b2)--(b);
\draw[->] (c2)--(c);
\draw[->] (d2)--(d);
\draw[->] (e2)--(e);

\draw[->] (a2)--(a3);
\draw[->] (b2)--(b3);
\draw[->] (c2)--(c3);
\draw[->] (d2)--(d3);
\draw[->] (e2)--(e3);

\draw[->, red] (a2)--(b);
\draw[->, red] (b2)--(c);
\draw[->, red] (d2)--(c);
\draw[->, red] (e2)--(d);

\draw[->, red] (a2)--(b3);
\draw[->, red] (b2)--(c3);
\draw[->, red] (d2)--(c3);
\draw[->, red] (e2)--(d3);

\node[anchor=west] at (4.5, 1.5) {$G \boxdot H$};
\node[anchor=west] at (4.5, 0.5) {\color{red}$G \times H$};

\node (a_base) at (0, -1) [roundnode] {};
\node (b_base) at (1, -1) [roundnode] {};
\node (c_base) at (2, -1) [roundnode] {};
\node (d_base) at (3, -1) [roundnode] {};
\node (e_base) at (4, -1) [roundnode] {};

\draw[->] (a_base)--(b_base);
\draw[->] (b_base)--(c_base);
\draw[->] (d_base)--(c_base);
\draw[->] (e_base)--(d_base);

\node[anchor=west] at (4.5, -1) {$G$};

\node (i0) at (-2, 0) [roundnode] {};
\node (i1) at (-2, 1) [roundnode] {};
\node (i2) at (-2, 2) [roundnode] {};

\draw[->] (i1)--(i0);
\draw[->] (i1)--(i2);

\node at (-3, 1) {$H$};

\end{tikzpicture}
   \caption{
  Illustration of the difference between $G\boxdot H$ and $G\times H$.
  The box product $G\boxdot H$ consists of just the black edges.
  The cross product $G\times H$ additionally contains the red edges.
  }\label{fig:digraph_products}
\end{figure}

For any digraphs $G, H$ there is an inclusion $G \boxdot H \subseteq G \times H$, as illustrated in Figure~\ref{fig:digraph_products}.
Unfortunately, $\dFl$ does not enjoy as simple a characterisation of its cylinders (c.f.\ Lemma~\ref{lem:a_cyl_commutes} for $\mathcal{A}$).
However, we can show that $G\times {\unitint}$ is the minimal digraph whose directed flag complex contains $\overline{\Cyl(\dFl(G))}$.
This is in analogy with $\overline{\Cyl(K)}$ being the minimal simplicial complex containing $\Cyl(K)$.

\begin{lemma}\label{lem:smallest_containing_cyl}
For any digraph $G$, $\Cyl(\dFl(G))\subseteq\overline{\Cyl(\dFl(G))} \subseteq \dFl(G \times {\unitint})$.
Moreover, $G\times {\unitint}$ is the smallest such digraph, i.e.
\begin{equation}
\Cyl(\dFl(G)) \subseteq \dFl(H) \implies G \times {\unitint} \subset H.
\end{equation}
\end{lemma}
\begin{proof}
To begin, note that 
\begin{equation}
\overline{\Cyl(\dFl(G))} \subset \dFl(H) \iff \Cyl(\dFl(G)) \subseteq \dFl(H)
\end{equation}
because $\dFl(H)$ is a simplicial complex.
Therefore, it suffices to prove that $\Cyl(\dFl(G)) \subseteq \dFl(G\times{\unitint})$ and that $G\times{\unitint}$ is the smallest such digraph.

Take a simplex $v_0 \dots v_k \in \dFl(G)$.
Now for any $i < j$, $v_i \to v_j$ in $G$ and certainly $0 = 0$, $1=1$ and $0 \to 1$ in ${\unitint}$.
Therefore, the paths
$(v_0, 0) \dots (v_k, 0)$,
$(v_0, 1) \dots (v_k, 1)$,
and 
$(v_0, 0) \dots (v_i, 0)(v_i, 1) \dots (v_k, 1)$
are all cliques in $G\times{\unitint}$.
Hence, $\Cyl(\dFl(G)) \subseteq \dFl(G \times {\unitint})$.

Now suppose $\Cyl(\dFl(G))\subseteq \dFl(H)$.
Take a node $v \in V(G)$, then $(v, 0)(v, 1) \in \Cyl(\dFl(G))$,
so there is an edge $(v, 0) \to (v, 1)$ in $H$.
Take an edge $v \to w$ in $G$, then $(v, 0)(w, 0)(w, 1)\in\Cyl(\dFl(G))$
so there are edges $(v, 0) \to (w, 0)$ and $(v, 0) \to (w, 1)$ in $H$.
Likewise, $(v, 0) (w, 1) (w, 1) \in \Cyl(\dFl(G))$ and hence there is an edge $(v, 1) \to (w, 1)$ in $H$.
These are all the edges in $G \times {\unitint}$ and hence $G\times {\unitint} \subseteq H$.
\end{proof}

Despite this, it is typically a much stronger condition for a vertex map to induce a triangle-collapsing digraph map $G\times{\unitint} \to H$ than to induce a triangle-collapsing simplicial morphism $\overline{\Cyl(\dFl(G))} \to \dFl(H)$.
The main exception is when $G$ is oriented.

\begin{lemma}\label{lem:cyl_differs_iff_recip}
Given a digraph $G$, the directed $3$-cliques in $\dFl(G\times{\unitint}) \setminus \overline{\Cyl(\dFl(G))}$ are
\begin{equation}
\mathfrak{R}\defeq \left\{
(v_0, 0)(v_1, \alpha)(v_0, 1)
\rmv
v_0 \recip v_1
,\,
\alpha=0\text{ or }1
\right\}.
\end{equation}
Moreover, all simplices in 
$\dFl(G\times{\unitint}) \setminus \overline{\Cyl(\dFl(G))}$
have at least one simplex from $\mathfrak{R}$ as a face.
In particular, $\overline{\Cyl(\dFl(G))} = \dFl(G\times{\unitint})$ if and only if $G$ is oriented.
\end{lemma}
\begin{proof}
Suppose there is a clique $p\in\dFl(G\times{\unitint}) \setminus \overline{\Cyl(\dFl(G))}$.
We can write
\begin{equation}
p = (v_0, 0)\dots (v_{k}, 0)(w_0, 1)\dots (w_m, 1)
\end{equation}
where $v_0 \dots v_k$ and $w_0 \dots w_m$ are both simplices in $\dFl(G)$ and moreover $v_i \tooreq w_j$ for every $0\leq i \leq k$ and $0\leq j \leq m$. 
Note that we must have $v_{i^{\ast}} = w_{j^\ast}$ for some $(i^\ast, j^\ast)$ because otherwise $v_0 \dots v_k w_0 \dots w_m\in \dFl(G)$ and so $p\in\overline{\Cyl(\dFl(G))}$, by Lemma~\ref{lem:closure_paths}.
Moreover, we can assume that $(i^{\ast}, j^{\ast}) \neq (k, 0)$ because if this is the only time that $v_i = w_j$ then $p\in \Cyl(\dFl(G))$.
We split into cases depending on which coordinate differs.

\textbf{Case 1:} If $i^\ast \neq k$ then a face of $p$ is
\begin{equation}
p'\defeq (v_{i^\ast}, 0)(v_k, 0)(w_{j^\ast}, 1)
=
(v_{i^\ast}, 0)(v_k, 0)(v_{i^\ast}, 1).
\end{equation}
This clique must be in $\dFl(G\times{\unitint})$ which implies that there is a reciprocal edge $v_{i^\ast} \recip v_k$.

\textbf{Case 2:} If $j^\ast \neq 0$ then a face of $p$ is
\begin{equation}
p''\defeq (v_{i^\ast}, 0)(w_0, 0)(w_{j^\ast}, 1)
=
(v_{i^\ast}, 0)(w_0, 0)(v_{i^\ast}, 1).
\end{equation}
This clique must be in $\dFl(G\times{\unitint})$ which implies that there is a reciprocal edge $v_{i^\ast} \recip w_0$.
\end{proof}

For general $G$, $\dFl(G\times{\unitint})$ may significantly differ from $\overline{\Cyl(\dFl(G))}$.
Indeed, the two may not even be $\Sys[\ascat{TcOSC}]$-homotopy equivalent.

\begin{figure}[hbtp]
  \centering
  \begin{tikzpicture}[
  roundnode/.style={circle, fill=black, minimum size=4pt},
	squarenode/.style={fill=black, minimum size=4pt},
	inner sep=2pt,
	outer sep=1pt
  ]

  \node (a) at (0, 0) [roundnode, label=left:$a$] {};
  \node (b) at (1, 0) [roundnode, label=right:$b$] {};
  \node (ap) at (0, 1) [roundnode, label=left:$a'$] {};
  \node (bp) at (1, 1) [roundnode, label=right:$b'$] {};

  \draw[->, red] (a) to (b);
  \draw[->, bend left] (b) to (a);
  \draw[->, bend left, red] (ap) to (bp);
  \draw[->] (bp) to (ap);

  \draw[->, red] (a) -- (ap);
  \draw[->] (b) -- (bp);

  \draw[->] (a) -- (bp);
  \draw[->] (b) -- (ap);

\end{tikzpicture}
   \caption{
  The digraph $G\times {\unitint}$, as relabelled in the proof of Proposition~\ref{prop:dfl_times_not_simeq_cyl}.
  The tree, $T$, used to produce the basis for $\ker\bd_1$ is shown in red.
  }\label{fig:cylinder_of_recip}
\end{figure}
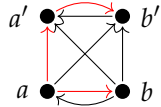

\begin{prop}\label{prop:dfl_times_not_simeq_cyl}
There exists a digraph $G$ for which, as path complexes,
$\Cyl(\dFl(G)) \not\simeq \dFl(G \times {\unitint})$
and
$\overline{\Cyl(\dFl(G))} \not\simeq \dFl(G \times {\unitint})$.
\end{prop}
\begin{proof}
We first note that for any simplicial complex $K$,
$
\Cyl(K) \simeq K \simeq \overline{\Cyl(K)}
$.
The proof of this is relatively technical, so we delay these results to
Propositions~\ref{prop:homotopy_for_path_complex_cylinders} and~\ref{prop:homotopy_for_osc_cylinders}.
Taking $K = \dFl(G)$, we see that it suffices to find $G$ for which
$\dFl(G)\not\simeq\dFl(G\times{\unitint})$.

Let $G$ be the complete digraph on two nodes $a, b$, i.e.\ there is a reciprocal edge $a\recip b$ (see Figure~\ref{fig:cylinder_of_recip}).
Since $G$ contains no $k$-cliques for $k\geq 2$, a simple Euler characteristic argument shows that $\dim H_1(\dFl(G)) = 1$.
Therefore, by Theorem~\ref{thm:path_complex_homot_invariance}, it suffices to show that $\dim H_1(\dFl(G \times {\unitint})) \neq 1$.
In the sequel we show that in fact $\dim H_1(\dFl(G\times {\unitint})) = 0$.
To reduce notation we change the names of the vertices in $G \times {\unitint}$ as follows
\begin{equation}
(a, 0) \mapsto a,
\quad
(b, 0) \mapsto b,
\quad
(a, 1) \mapsto a',
\quad
(b, 1) \mapsto b'.
\end{equation}

In the chain complex
\begin{equation}
\begin{tikzcd}
\Omega_2(\dFl(G \times {\unitint})) \arrow[r, "\bd_2"]
& \Omega_1(\dFl(G \times {\unitint})) \arrow[r, "\bd_1"]
& \Omega_0(\dFl(G \times {\unitint}))
\end{tikzcd}
\end{equation}
we can find a basis for $\ker \bd_1$ by first picking an oriented tree $T\subset G$ and then adding one cycle to the basis for each edge in $E(G) \setminus E(T)$.
One such basis (corresponding to the tree shown in Figure~\ref{fig:cylinder_of_recip}) is as follows
\begin{alignat}{3}
&c_1 &&\defeq ab + ba &&= \bd_2( aba' + baa' ), \\
&c_2 &&\defeq a'b' + b'a' &&= \bd_2( ab'a' + aa'b' ), \\
&c_3 &&\defeq ab + ba' - aa' &&= \bd_2( aba' ), \\
&c_4 &&\defeq aa' + a'b' - ab' &&= \bd_2( aa'b' ), \\
&c_5 &&\defeq [aa' + a'b'] - [ab + bb'] &&= \bd_2( aa'b' - abb').
\end{alignat}
Note that each of the $3$-cliques on the right-hand side belongs to $\Omega_2(\dFl(G\times {\unitint}))$.
Therefore, every element of this basis for $\ker\bd_1$ is null-homologous.
\end{proof}

In light of this, we choose to characterise when there is a one-step $\Sys[\ascat{TcOSC}]$-homotopy from $\dFl(f)$ to $\dFl(g)$ in terms of edge-based conditions, akin to Corollary~\ref{cor:exists_a_eq_f_to_g}.

\begin{cor}\label{cor:dfl_f_to_g}\label{cor:dfl_homot_characterise}
Given two triangle-collapsing digraph maps $f, g: G \to H$,
a vertex map $F: V(G) \times \{0, 1\} \to V(H)$
is a one-step $\Sys[\ascat{TcOSC}]$-homotopy from $\dFl(f)$ to $\dFl(g)$ if and only if the following conditions are all satisfied:
\begin{enumerate}
\item $x \tooreq y$ $\implies$ $f(x) \tooreq g(y)$; \item $x \to y$ in $G$ and  $f(x) = g(y)$  $\implies$ $f(x) = f(y) = g(x) = g(y)$;
\item 
$F \circ \iota_0 = f$
and
$F \circ \iota_1 = g$.
\end{enumerate}
\end{cor}
\begin{proof}
First note that the third condition uniquely determines $F$ and, furthermore, is necessary for $F$ to induce a one-step $\Sys[\ascat{TcOSC}]$-homotopy from $\dFl(f)$ to $\dFl(g)$.
Given this, the first condition becomes equivalent to $F$ inducing a weak digraph map $G\times{\unitint} \to H$.
By Lemma~\ref{lem:dfl_functoriality}, this is then equivalent to $F$ inducing a weak simplicial morphism $\overline{\Cyl(\dFl(G))}\to \dFl(H)$.
Then, Lemma~\ref{lem:closure_paths} and Lemma~\ref{lem:smallest_containing_cyl}, show us that
\begin{equation}
\Sk_1(\dFl(G\times{\unitint}))
\subseteq
\overline{\Cyl(\dFl(G))}
\subseteq
\dFl(G\times{\unitint}).
\end{equation}
Finally, Lemma~\ref{lem:osc_tpc_equiv_condition} shows us that the second condition is equivalent to $F$ inducing a triangle-collapsing simplicial morphism $\overline{\Cyl(\dFl(G))} \to \dFl(H)$.
\end{proof}

This begets the definition of a system of one-step homotopies for $\ascat{TcDgr}$.

\begin{defin}\label{def:t_dfl}
For arbitrary digraphs $G, H$, we construct a digraph $\mdf{\Sys[\dFl](G, H)}$ on the vertex set $\MorXY{\ascat{TcDgr}}{G}{H}$ as follows.
Given triangle-collapsing digraph maps $f, g: G\to H$, include an edge $f \to g$ if and only if
\begin{enumerate}
\item $x\tooreq y \implies f(x) \tooreq g(y)$; and
\item $x\to y$ and $f(x) = g(y)$ $\implies$ $f(x) = f(y)= g(x) = g(y)$.
\end{enumerate}
Repeating this construction over all $G, H$ yields a system of one-step homotopies, $\mdf{\Sys[\dFl]}$, for $\ascat{TcDgr}$.
For brevity, we use \mdf{$\simeq_{\dFl}$}, in lieu of $\simeq_{\Sys[\dFl]}$, to denote the resulting equivalence relations.
\end{defin}

\begin{cor}\label{cor:dfl_pullback}
The pull-back system along the directed flag complex functor $\dFl: \ascat{TcDgr}\to \ascat{TcOSC}$ is $\dFl^\ast \Sys[\ascat{TcOSC}] = \Sys[\dFl]$.
\end{cor}

Comparing the conditions in Definition~\ref{def:t_dfl} to Corollary~\ref{cor:characterise_ahomot}, we see that $\Sys[\dFl]$ is a strictly smaller system.

\begin{cor}
For triangle-collapsing digraph maps $f, g: G \to H$,
$f\simeq_{\dFl} g \implies f \simeq_{\mathcal{A}} g$
\end{cor}

This system is of a markedly different quality: we do not define the system in terms of a cylinder functor. However, when $G$ is oriented, Lemma~\ref{lem:cyl_differs_iff_recip} makes it possible to find morphisms in $\ascat{TcDgr}$ which induce one-step $\Sys[\ascat{TcOSC}]$-homotopies.

\begin{cor}
Let $G$ be an oriented digraph and let $f, g: G \to H$ be triangle-collapsing digraph maps.
A vertex map $F: V(G) \times \{0, 1\} \to V(H)$ induces a one-step $\Sys[\ascat{TcOSC}]$-homotopy from $\dFl(f)$ to $\dFl(g)$ if and only if it induces a triangle-collapsing digraph map $F: G\times {\unitint} \to H$, and
\begin{equation}
F \circ \iota_0 = f
\quad\text{and}\quad
F \circ \iota_1 = g.
\end{equation}
\end{cor}

\subsection{Comparing the two pull-backs}\label{sec:examples}

To illustrate the difference between the two systems, we now discuss a number of standard examples of $\simeq_{\mathcal{A}}$ equivalence that require extra conditions in order to pass over to the new system, $\Sys[\dFl]$.
These examples may help the practitioner to decide which homology theory is appropriate for a given application.
We will primarily be concerned with when a digraph retracts onto one of its induced subgraphs.

\begin{defin}
Given a digraph $G=(V, E)$ an \mdf{induced subgraph} is a digraph consisting of some of the vertices in $V$ but all the edges between those vertices, i.e.\ a digraph of the form $A=(V_A, E_A)$ where $V_A \subseteq V$ and $E_A = E \cap (V_A \times V_A)$.
\end{defin}

\begin{defin}
Let $\{\ast\}$ denote the digraph on a single vertex.
We say that a digraph $G$ is \mdf{$\mathcal{A}$-contractible} (resp.\ \mdf{$\dFl$-contractible}) if $G\simeq_{\mathcal{A}} \{\ast\}$ (resp.\ $G\simeq_{\dFl}\{\ast\})$.
\end{defin}

\begin{rem}
Any $\dFl$-contractible digraph is $\mathcal{A}$-contractible.
\end{rem}

\begin{defin}
Given a digraph $G$ and an induced subgraph $A \subseteq G$, we say a weak digraph map $r:G \to A$ is a \mdf{retraction} if $\restr{r}{A} = \id_A$.
Let $\iota: A \to G$ denote the inclusion digraph map.
If $\iota \circ r\simeq_{\mathcal{A}} \id_G$ then we say $r$ is an \mdf{$\mathcal{A}$-deformation retraction}.
If $r$ is triangle-collapsing and $\iota \circ r \simeq_{\dFl} \id_G$ then we say $r$ is a \mdf{$\dFl$-deformation retraction}.
\end{defin}

\citeauthor{Grigoryan2014} obtained a particularly simple necessary condition for $r$ to be an $\mathcal{A}$-deformation retract.

\begin{prop}[{\cite[Corollary~3.7]{Grigoryan2014}}]
Given a digraph $G$ and a retraction $r:G \to A$,
suppose further that
\begin{equation}
\forall x \in V(G) \quad x \tooreq r(x)
\quad\text{or}\quad
\forall x \in V(G) \quad r(x) \tooreq x,
\end{equation}
where edges are required in $G$,
then $r$ is an $\mathcal{A}$-deformation retraction.
\end{prop}

A similar result can be obtained for $\dFl$-deformation retractions.

\begin{prop}\label{prop:dfl_def_retract}
Given a digraph $G$, and a triangle-collapsing retraction $r: G\to A$,
suppose further that
\begin{equation}
\forall x \in V(G) \quad x \tooreq r(x)
\quad\text{and}\quad
\forall (x, y) \in E(G) \quad x \to r(y)
\end{equation}
or
\begin{equation}
\forall x \in V(G) \quad r(x) \tooreq x
\quad\text{and}\quad
\forall (x, y) \in E(G) \quad r(x) \to y,
\end{equation}
where edges are required in $G$,
then $r$ is a $\dFl$-deformation retraction.
\end{prop}
\begin{proof}
Let $\iota: A \to G$ denote the inclusion digraph map.
Then $r \circ \iota = \id_A$ and $\iota \circ r$ is a digraph map $G \to G$, given by the same vertex map as $r$.
Let us assume the first pair of conditions holds; the second pair admits a similar proof.
We show that there is a one-step $\Sys[\dFl]$-homotopy from $\id_G$ to $\iota \circ r$.

First note that both $\id_G$ and $\iota \circ r$ are triangle-collapsing digraph maps.
Then, our assumptions on $r$ imply that
\begin{equation}
x \tooreq y \implies x \tooreq r(y),
\end{equation}
and hence the first condition of Definition~\ref{def:t_dfl} holds.
For the second condition, take an edge $x \to y$ in $G$ such that $\id_G(x) = (\iota \circ r)(y)$.
Then our assumption on $r$ implies that $x \to r(y) = x$.
Therefore, the edge is a self-loop, which cannot exist.
Hence, the second condition is vacuously satisfied.
\end{proof}

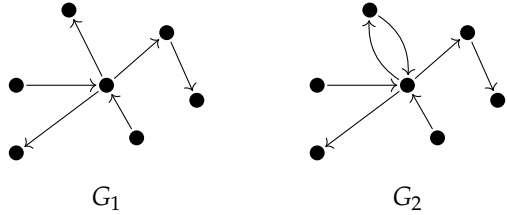
\begin{figure}[hbtp]
  \centering
  \begin{tikzpicture}[
  roundnode/.style={circle, fill=black, minimum size=4pt},
	squarenode/.style={fill=black, minimum size=4pt},
	inner sep=2pt,
	outer sep=1pt
  ]

  \node (a) at (0, 0) [roundnode] {};
  \node (b) at (-0.5, 1) [roundnode] {};
  \node (c) at (0.8, 0.7) [roundnode] {};
  \node (d) at (1.2, -0.2) [roundnode] {};
  \node (e) at (-1.2, -0.9) [roundnode] {};
  \node (f) at (0.4, -0.7) [roundnode] {};
  \node (g) at (-1.2, 0) [roundnode] {};

  \draw[->] (a) -- (b);
  \draw[->] (a) -- (c);
  \draw[->] (c) -- (d);
  \draw[->] (a) -- (e);
  \draw[->] (f) -- (a);
  \draw[->] (g) -- (a);

  \node (a2) at (4, 0) [roundnode] {};
  \node (b2) at (3.5, 1) [roundnode] {};
  \node (c2) at (4.8, 0.7) [roundnode] {};
  \node (d2) at (5.2, -0.2) [roundnode] {};
  \node (e2) at (2.8, -0.9) [roundnode] {};
  \node (f2) at (4.4, -0.7) [roundnode] {};
  \node (g2) at (2.8, 0) [roundnode] {};

  \draw[->, bend left] (a2) to (b2);
  \draw[->, bend left] (b2) to (a2);
  \draw[->] (a2) -- (c2);
  \draw[->] (c2) -- (d2);
  \draw[->] (a2) -- (e2);
  \draw[->] (f2) -- (a2);
  \draw[->] (g2) -- (a2);

  \node at (0, -1.5) {$G_1$};
  \node at (4, -1.5) {$G_2$};
\end{tikzpicture}
   \caption{Examples of two pseudo-trees.
  The first example, $G_1$, is oriented and hence is both $\dFl$-contractible and $\mathcal{A}$-contractible.
  The second example, $G_2$, is not oriented and hence is only $\mathcal{A}$-contractible.
  }\label{fig:example_tree}
\end{figure}

\begin{defin}
We say a digraph $G$ is a \mdf{pseudo-tree} if its underlying undirected graph (collapsing any reciprocal edges to a single edge) is a tree.
\end{defin}

\begin{prop}
If $G$ is a pseudo-tree then $G$ is $\mathcal{A}$-contractible.
Furthermore, a pseudo-tree is $\dFl$-contractible if and only if it is oriented.
\end{prop}
\begin{proof}
For the first statement, we refer the reader to Example~3.10 of~\cite{Grigoryan2014}.
For the second, suppose $G$ is not oriented.
Then, since $G$ is connected and not oriented, a standard induction argument shows
$\# E(G) \geq \# V(G)$.
Note that a pseudo-tree has no $k$-cliques for $k\geq 3$.
Therefore, $\Omega(\dFl(G))$ has homology only in degrees $0$ and $1$.
By considering the Euler characteristic, we obtain
\begin{equation}
\dim H_1(\dFl(G)) - \dim H_0(\dFl(G)) = \# E(G) - \#V(G) \geq 0
\end{equation}
Moreover, note that $\dim H_0(\dFl(G))=1$ since this counts the number of weakly connected components of $G$.
Therefore, $H_1(\dFl(G))$ must be non-trivial and hence $G$ is not $\dFl$-contractible.

It remains to show that if $G$ is oriented then $G$ is $\dFl$-contractible.
We proceed by induction on the number of edges in $G$.
If $\# E = 0$ then $G=\{*\}$, which is clearly $\dFl$-contractible.
Else, pick a leaf node $v$ and let $w$ be its unique neighbour.
Let $r$ be given by $r(v) \defeq w$ and the identity elsewhere.
Then we see that $G\setminus\{v\}$ is a $\dFl$-deformation retract of $G$ by Proposition~\ref{prop:dfl_def_retract}.
\end{proof}

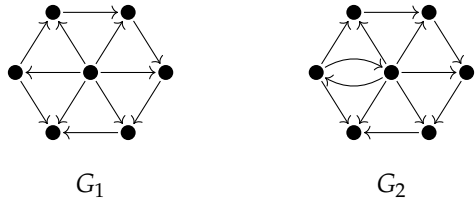
\begin{figure}[hbtp]
  \centering
  \begin{tikzpicture}[
  roundnode/.style={circle, fill=black, minimum size=4pt},
	squarenode/.style={fill=black, minimum size=4pt},
	inner sep=2pt,
	outer sep=1pt
  ]

  \node (a) at (4, 0) [roundnode] {};
  \node (b) at (3, 0) [roundnode] {};
  \node (c) at (5, 0) [roundnode] {};
  \node (d) at (3.5, 0.8) [roundnode] {};
  \node (e) at (4.5, 0.8) [roundnode] {};
  \node (f) at (3.5, -0.8) [roundnode] {};
  \node (g) at (4.5,  -0.8) [roundnode] {};

  \draw[->, bend left] (a) to (b);
  \draw[->, bend left] (b) to (a);

  \draw[->] (a) -- (c);
  \draw[->] (a) -- (d);
  \draw[->] (a) -- (e);
  \draw[->] (a) -- (f);
  \draw[->] (a) -- (g);
  \draw[->] (d) -- (e);
  \draw[->] (e) -- (c);
  \draw[->] (c) -- (g);
  \draw[->] (g) -- (f);
  \draw[->] (b) -- (f);
  \draw[->] (b) -- (d);

  \node (a2) at (0, 0) [roundnode] {};
  \node (b2) at (-1, 0) [roundnode] {};
  \node (c2) at (1, 0) [roundnode] {};
  \node (d2) at (-0.5, 0.8) [roundnode] {};
  \node (e2) at ( 0.5, 0.8) [roundnode] {};
  \node (f2) at (-0.5, -0.8) [roundnode] {};
  \node (g2) at ( 0.5,  -0.8) [roundnode] {};

  \draw[->] (a2) -- (b2);
  \draw[->] (a2) -- (c2);
  \draw[->] (a2) -- (d2);
  \draw[->] (a2) -- (e2);
  \draw[->] (a2) -- (f2);
  \draw[->] (a2) -- (g2);
  \draw[->] (d2) -- (e2);
  \draw[->] (e2) -- (c2);
  \draw[->] (c2) -- (g2);
  \draw[->] (g2) -- (f2);
  \draw[->] (b2) -- (f2);
  \draw[->] (b2) -- (d2);

  \node at (0, -1.5) {$G_1$};
  \node at (4, -1.5) {$G_2$};
\end{tikzpicture}
   \caption{
  Two examples of star-like graphs.
  In both examples the centre node is the star centre.
  In the first example, there is no reciprocal edge and hence $G_1$ is both $\dFl$-contractible and $\mathcal{A}$-contractible.
  In the second example, there is a reciprocal edge involving the star centre and hence $G_2$ is only $\mathcal{A}$-contractible.
  }\label{fig:example_star_like}
\end{figure}

\begin{defin}\label{def:star_like}
We say a digraph $G$ is \mdf{star-like} if there exists $v_0 \in V(G)$ such that for any other $v\in V(G)$, $v_0 \to v$.
We say $G$ is \mdf{inverse star-like} if there exists $v_0\in V(G)$ such that for any other $v\in V(G)$, $v \to v_0$.
We call $v_0$ a \mdf{(inverted) star centre}.
\end{defin}

\begin{cor}\label{cor:star_like}
If $G$ is star-like or inverse star-like then $G$ is $\mathcal{A}$-contractible.
If $G$ has a (inverted) star centre $v_0$ and there are no reciprocal edges involving $v_0$, then $G$ is $\dFl$-contractible.
\end{cor}
\begin{proof}
For the first statement, we refer the reader to Example~3.11 of~\cite{Grigoryan2014}.
For the second statement, we deal with the case where $G$ is star-like with star centre $v_0$, the inverse case admits a similar proof.
Let $G_0$ be the induced subgraph on $\{v_0\}$ and let $r: G \to G_0$ be the unique digraph map.
We show that $r$ satisfies the conditions of Proposition~\ref{prop:dfl_def_retract}.

Certainly $r$ is triangle-collapsing.
First, since $v_0$ is a star centre, for any $x \in V(G)$ we have $ r(x) = v_0 \tooreq x$.
Second, given an edge $x \to y$,
we cannot have $y=v_0$
because otherwise there would be a double edge $v_0 \recip x$.
Since $y \neq v_0$, there must be an edge $v_0 = r(x) \to y$.
\end{proof}

Using these examples, we can show that the directed flag complex cannot be made into a functor from $\ascat{WkDgr}$.
Indeed, even the subsequent chain complex cannot be made into such a functor.

\begin{figure}[htbp]
  \centering
  \begin{tikzpicture}[
  roundnode/.style={circle, fill=black, minimum size=4pt},
	squarenode/.style={fill=black, minimum size=4pt},
	inner sep=2pt,
	outer sep=1pt
  ]

  \node (a) at (0, 0) [roundnode, label=left:$v_1$] {};
  \node (b) at (1, 0) [roundnode, label=right:$v_2$] {};
  \node at (0.5, -0.5) {$G$};

  \node (c) at (3, 1.5) [roundnode, label=left:$v_1$] {};
  \node (d) at (4, 1.5) [roundnode, label=right:$v_2$] {};
  \node (e) at (3.5, 2.5) [roundnode, label=above:$v_0$] {};
  \node at (3.5, 1) {$H$};

  \node (f) at (6, 0) [roundnode, label=left:$v_1$] {};
  \node (g) at (7, 0) [roundnode, label=right:$v_2$] {};
  \node at (6.5, -0.5) {$G$};

  \draw[->, bend left] (a) to (b);
  \draw[->, bend left] (b) to (a);

  \draw[->, bend left] (c) to (d);
  \draw[->, bend left] (d) to (c);
  \draw[->, bend right] (e) to (c);
  \draw[->, bend left] (e) to (d);

  \draw[->, bend left] (f) to (g);
  \draw[->, bend left] (g) to (f);

  \node (g1) at (10, 1) {$\Field$};
  \node (g2) at (11, 2) {$0$};
  \node (g3) at (12, 1) {$\Field$};

  \draw[->] (g1) to node [midway, above, sloped] {$\inducedhom{f}$} (g2);
  \draw[->] (g2) to node [midway, above, sloped] {$\inducedhom{g}$} (g3);
  \draw[->] (g1) to node [midway, below, sloped] {$\id$} (g3);

  \draw[->, bend left] (0.5, 0.5) to node [midway, above, sloped] {$f$} (2, 1.75);
  \draw[->, bend left] (5, 1.75) to node [midway, above, sloped] {$g$} (6.5, 0.5);
  \draw[->] (2, 0) to node [midway, below, sloped] {$\id$} (5, 0);

  \draw[->, blue,
    line join=round,
    decorate, decoration={
        zigzag,
        segment length=7,
        amplitude=1.3,post=lineto,
        post length=2pt
    }] (7, 1.5) to node [midway, above] {$H_1 \circ \Omega \circ \dFl$} (9.5, 1.5);

\end{tikzpicture}
   \caption{A commuting diagram in $\Dgr$ which illustrates why $\dFl$ cannot be made into a functor $\Dgr\to\Ch$.
  For definitions of the digraph maps, see the proof of Proposition~\ref{prop:dflag_func_failure}}\label{fig:dflag_functoriality}
\end{figure}
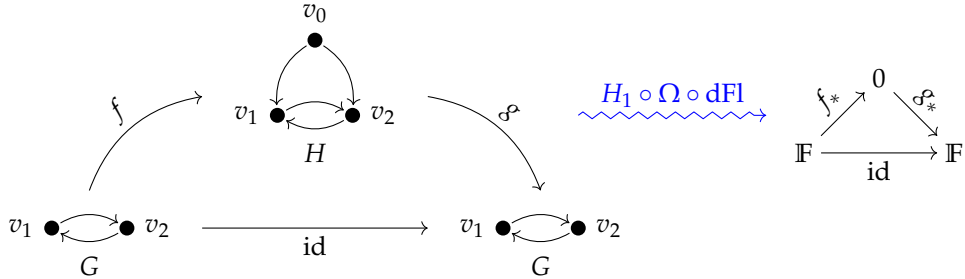

\begin{prop}\label{prop:dflag_func_failure}
The object map,
$\Omega \circ \dFl: \Obj(\ascat{WkDgr}) \to \Obj(\ascat{Ch})$,
which assigns to each digraph the chain complex associated to its directed flag complex,
cannot be made into a functor $\ascat{WkDgr} \to \ascat{Ch}$.
\end{prop}
\begin{proof}
Suppose, for contradiction, that such a functor exists.
Consider the three digraphs illustrated in Figure~\ref{fig:dflag_functoriality}.
The weak digraph map $f$ is given by the obvious inclusion, whilst $g$ maps $v_1$ and $v_2$ to themselves and $v_0\mapsto v_1$.
Note that $g\circ f$ composes to the identity and hence the triangle of digraph maps commutes.
Applying $H_1 \circ \Omega \circ \dFl$ to this diagram we find that $\inducedhom{g}\circ\inducedhom{f}$ is the identity on $H_1(\dFl(G))$.
Using Lemma~\ref{lem:gens_of_osc}, we see that $\Omega(\dFl(G))$ is
\begin{equation}
\begin{tikzcd}
0 \arrow[r] & \Field\langle v_1 v_2, v_2 v_1 \rangle \arrow[r, "\bd_1"] & \Field\langle v_1,  v_2\rangle
\end{tikzcd}
\end{equation}
and moreover $\bd_1(v_1 v_2) = - \bd_1(v_2 v_1) \neq 0$.
Therefore, $H_1(\dFl(G))\cong \Field$.
However, $H_1(\dFl(H))$ is the trivial vector space because $H$ is star-like and hence $\dFl$-contradctible.
The identity on $H_1(\dFl(G))$ cannot factor through $0$, so we obtain a contradiction.
\end{proof}

\begin{defin}
If there is a reciprocal edge $v_0 \recip v_1$ then we say that $w$ \mdf{cones} the reciprocal edge if
\begin{equation}
(w \to v_0\text{ and }w \to v_1)
\quad\text{or}\quad
(v_0 \to w\text{ and }v_1\to w).
\end{equation}
\end{defin}

\begin{prop}\label{prop:retract_a_to_b0}
Let $G$ be a digraph, let $a \in V(G)$ and let the neighbouring vertices of $a$ be denoted $b_0, \dots, b_n$.
Then let $H$ denote the subgraph of $G$ on $V(G) \setminus \{a\}$.
Suppose $b_0$ is such that, for all $i\in \{ 1, \dots, n \}$,
\begin{align}
a \to b_i & \implies b_0 \to b_i; \label{eq:a_to_bi}\\
a \leftarrow b_i & \implies b_0 \leftarrow b_i. \label{eq:bi_to_a}
\end{align}
Then there is a weak digraph map $r: G \to H$ given by $\restr{r}{H} = \id_H$ and $r(a) = b_0$.
Moreover, $r$ is an $\mathcal{A}$-deformation retract.
Furthermore, if 
there is not a reciprocal edge $a \recip b_0$ and
there does not exist $i\in \{1, \dots, n\}$ such that there is a reciprocal edge $b_0 \recip b_i$ which is coned by $a$
then $r$ is a $\dFl$-deformation retract.
\end{prop}
\begin{proof}
For the proof that $r$ is a weak digraph map and $\mathcal{A}$-deformation retract, see Example~3.15 of~\cite{Grigoryan2014}.
We prove the final statement, under the assumption that $a \to b_0$ (the $b_0 \to a$ case admits a similar proof).
Note that the lack of a reciprocal edge means $b_0 \not \to a$.

We proceed by checking the conditions of Proposition~\ref{prop:dfl_def_retract}.
To begin, take arbitrary $x \in V(G)$ then either $x \neq a$, in which case $x = r(x)$, or $x = a$, in which case $x \to b_0 = r(x)$.
Therefore, we see $x \tooreq r(x)$.
Next, take an edge $x \to y$.
Either $y \neq a$, in which case $x \to y = r(y)$, or $y=a$ in which case $x=b_i$ for some $i$.
We have seen that $b_0 \not \to a$, so we must have $i > 0$ and hence condition~(\ref{eq:bi_to_a}) implies that $x=b_i \to b_0 = r(y)$.

It remains to show that $r$ is triangle-collapsing, so take a simplex $v_0 v_1 v_2 \in \dFl(G)$ and supose that $r(v_0) = r(v_2)$
Since $r$ is the identity on $H$ we must have $\{v_0, v_2\} = \{a, b_0\}$ and $v_1=b_i$ for some $i>0$.
Moreover, we know $b_0 \not\to a$ so $v_0 v_1 v_2$ must have the form $a b_i b_0$ for some $i> 0$.
But then condition (\ref{eq:a_to_bi}) implies $b_0 \to b_i$ and hence there is a reciprocal edge $b_i \recip b_0$ which is coned by $a$, contradicting our assumption.
\end{proof}

\begin{prop}
Let $G_n$ be a complete digraph on $n$ nodes.
Then $G_n$ is $\mathcal{A}$-contractible for every $n\geq 1$
but $G_n$ is $\dFl$-contractible if and only if $n=1$.
\end{prop}
\begin{proof}
The first statement follows immediately from Corollary~\ref{cor:star_like} because $G_n$ is star-like.
For the second, it is clear that $G_1$ is contractible.
Given $n>1$, note that $\dim\Omega_n(\dFl(G_n))=0$ because there can be no $n$-simplices on a vertex set of size $n$.
Therefore, it suffices to find a cycle, i.e.\ an element $c_{n-1}\in \Omega_{n-1}(\dFl(G_n))$ such that $\bd c_{n-1} = 0$.
We construct these inductively via a suspension-style argument, starting with the base case $n=2$.
We denote the nodes of $G_n$ by $x_1, \dots, x_n$.
When $n=2$, we can take $c_1\defeq x_0 x_1 + x_1 x_0$.

Now suppose $c_{n-1}\in \Omega_{n-1}(\dFl(G_{n}))$ is a cycle, such that $\bd c_{n-1} = 0$.
Using Lemma~\ref{lem:gens_of_osc}, we can write $c_{n-1}$ as a sum of $(n-1)$-simplices,
\begin{equation}
c_{n-1} = \sum_{v_0 \dots v_{n-1} \in \dFl(G_n)} \alpha^{v_0 \dots v_{n-1}} v_0 \dots v_{n-1}
\end{equation}
for some coefficients $\alpha^{v_0 \dots v_{n-1}}\in\Field$.
Note that every $(n-1)$-simplex in $\dFl(G_n)$ is also an $(n-1)$-simplex of $\dFl(G_{n+1})$.
Hence, suppressing notation for the inclusion $G_n \hookrightarrow G_{n+1}$, we can view $c_{n-1}$ as an element of $\Omega_{n-1}(\dFl(G_{n+1}))$.
Moreover,
for every $(n-1)$-simplex $v_0 \dots v_{n-1}\in\dFl(G_n)$,
there are $n$-simplices $v_0 \dots v_{n-1} x_{n+1}, x_{n+1} v_0 \dots v_{n-1}\in\dFl(G_{n+1})$.
We form two sums
\begin{align}
c_{n, s} &\defeq 
\sum_{v_0 \dots v_{n-1} \in \dFl(G_n)} \alpha^{v_0 \dots v_{n-1}} (x_{n+1} v_0 \dots v_{n-1}),  \\
c_{n, t} &\defeq 
\sum_{v_0 \dots v_{n-1} \in \dFl(G_n)} \alpha^{v_0 \dots v_{n-1}} (v_0 \dots v_{n-1} x_{n+1}),
\end{align}
such that $c_{n,s}, c_{n,t} \in \Omega_n(\dFl(G_{n+1}))$.
A standard calculation, using the fact that $\bd c_{n-1} = 0$, shows that $\bd(c_{n, s}) = c_{n-1}$ and $\bd(c_{n, t}) = (-1)^n c_{n-1}$.
Therefore, we can set $c_n \defeq c_{n, s} + (-1)^{n+1} c_{n, t}$ to complete the inductive step.
\end{proof}

Despite not being contractible, we anticipate that most of the homology groups of $\dFl(G_n)$ are trivial.
The following conjecture has been verified up to $n=8$, via Flagser~\cite{Luetgehetmann2020}.
\begin{conjecture}
Let $G_n$ denote a complete digraph on $n$ nodes.
If $n> 1$ then
\begin{equation}
\dim H_k(\dFl(G_n)) = 
\begin{cases}
1 &\text{if }k=0,\\
!n &\text{if }k=n-1,\\
0 &\text{otherwise},
\end{cases}
\end{equation}
where $!n$ is the number of derangements of $\{1, \dots, n\}$.
\end{conjecture}

\begin{example}\label{ex:dfl_neq_a_revisit}
Finally, we revisit the digraph, $G$, in Figure~\ref{fig:dfl_neq_a_example} from Example~\ref{ex:dfl_neq_a}.
First, we can show $G$ is $\mathcal{A}$-contractible.
Using Proposition~\ref{prop:retract_a_to_b0}, we can successively retract nodes $N$ and $S$ onto the subgraph spanned by $\{x_0, x_1\}$.
This is a complete digraph; hence $G$ is $\mathcal{A}$-contractible.
For the directed flag complex, the maximum non-trivial homology is in dimension $2$ because there are no directed $4$-cliques.
One can then verify computationally (e.g.\ via Flagser~\cite{Luetgehetmann2020}) that $H(\dFl(G))$ is the homology of the $2$-sphere.
Alternatively, viewing $\dFl(G)$ as a semi-simplicial set, one can recognise that $\dFl(G)$ coincides with the semi-simplicial set underlying a standard $\Delta$-complex structure on $S^2$ (see~\cite[\S~2.1]{HatcherAllen2002At} for a definition).
\end{example}
 \section{Stability of persistent directed flag complex homology}\label{sec:stability}
In this section, we study the stability of persistent homology pipelines which use the directed flag complex.
We do not study the stability of path homology, since it has already been well-treated~\cite{ Chaplin2024,ChowdhurySIAM,zhang2024stability}.

Closely related to the directed flag complex is the ordered tuples complex, as defined by~\citeauthor{turner2019rips}~\cite{turner2019rips}.
This complex can also be viewed as a path complex and can be made into a functor on $\ascat{WkDgr}$.
This allows\ \citeauthor{turner2019rips} to obtain stability theorems for persistent ordered tuples homology, with respect to the correspondence distance~\cite[Theorem~21]{turner2019rips}.
In contrast, \citeauthor{turner2019rips}~showed that ordered set homology (i.e.\ directed flag complex homology) does not enjoy such a general stability result~\cite[\S~5.2]{turner2019rips}.

\subsection{The interleaving distance}
Let \mdf{$\Rposet$} denote $\R$ equipped with the $\leq$ poset structure, viewed as a category.
Given a category $\C$, a \mdf{persistent $\C$-object} is a functor $M:\Rposet \to \C$; the morphisms $M(s\leq t)$ are called the \mdf{structure maps} of $M$.
Given a persistent $\ascat{TcDgr}$-object, $M \in \Funct{\Rposet}{\ascat{TcDgr}}$, its \mdf{persistent (directed flag complex) homology} is $(H \circ \dFl)(M)$.
The resulting object is a persistent finite-dimensional graded vector space, commonly known as a \mdf{persistence module}.
Any persistence module can be described by a complete, discrete invariant, called its \mdf{barcode} (for more information, see~\cite{chazal2016structure, CrawleyBoevey2012}).
The barcode can be used as a statistical descriptor of the initial persistent object $M$.
For more background on persistent homology and its statistical analysis, we refer the reader to~\cite{Chazal2021}.

In order to be useful for applications, it is important that this construction is stable.
That is, when $M, N\in\Funct{\Rposet}{\ascat{TcDgr}}$ differ by some small perturbation to the input data, it is desirable to obtain bounds on
\begin{equation}
d ( (H\circ \dFl)(M), (H \circ \dFl)(N) ),
\end{equation}
for some appropriate choice of metric, $d$.
We can use the structure of $\Rposet$ to put an extended pseudometric on $\Funct{\Rposet}{\C}$, the category of persistent $\C$-objects.

\begin{defin}
Fix a category $\C$ and $\delta \geq 0$ and two persistent objects $M, N: \Rposet \to \C$.
\begin{enumerate}
\item The \mdf{$\delta$-shift of $M$}, $\mdf{\shifted{M}{\delta}}$, is the persistent $\C$-object given by
\begin{equation}
\shifted{M}{\delta}(t) \defeq M(t+\delta)
\quad \text{and} \quad
\shifted{M}{\delta}(s \leq t) \defeq M(s + \delta \leq t + \delta).
\end{equation}
\item The \mdf{$\delta$-transition morphism}, \mdf{$\transmorph{M}{\delta}$}, is a natural transformation $M \Rightarrow
 \shifted{M}{\delta}$, given by $\transmorph{M}{\delta}(t) \defeq M(t \leq t + \delta)$.
\item A \mdf{$\delta$-interleaving} between $M$ and $N$ is a pair of natural transformations
\begin{equation}
f: M \Rightarrow \shifted{N}{\delta}
\quad \text{and} \quad
g: N \Rightarrow \shifted{M}{\delta},
\end{equation}
such that $\shifted{g}{\delta} \circ f = \transmorph{M}{2\delta}$ and $\shifted{f}{\delta} \circ g = \transmorph{N}{2\delta}$.
If such an interleaving exists, we say $M$ and $N$ are \mdf{$\delta$-interleaved (in $\C$)}.
\item Finally, the \mdf{interleaving distance} between $M$ and $N$ is
\begin{equation}
\mdf{d_I(M , N)} \defeq \left\{ \delta \geq 0 \rmv M\text{ and }N \text{ are }\delta\text{-interleaved} \right\} \in [0, \infty].
\end{equation}
\end{enumerate}
\end{defin}

The interleaving distance is an extended pseudo-metric on persistent $\C$-objects~\cite{Bubenik2012}.
Moreover, the interleaving distance behaves well with respect to functors.

\begin{lemma}\label{lem:functors_are_interleaving_lipschitz}
Any functor $F: \C \to \D$ induces a functor $F: \Funct{\Rposet}{\C} \to \Funct{\Rposet}{\D}$.
Endowing these functor categories with the interleaving distance, $F$ is a $1$-Lipschitz map, i.e.\
\begin{equation}
d_I( F(M), F(N) ) \leq d_I (M , N)
\end{equation}
for all $M,N\in\Funct{\Rposet}{\C}$.
\end{lemma}
\begin{proof}
The functor is defined in the obvious way:
given $M\in\Funct{\Rposet}{\C}$,
we define $F(M)(t) \defeq F(M(t))$ and $F(M)(s\leq t)\defeq F(M(s\leq t))$.
If $f, g$ constitute a $\delta$-interleaving between $M, N \in \Funct{\Rposet}{\C}$ then $F(f), F(g)$ constitute a $\delta$-interleaving between $F(M)$ and $F(N)$.
\end{proof}

\subsection{Stability theorems}
Since $\Sys[\dFl]$ is a pull-back of $\Sys[\ascat{WkPathC}]$, it is a transitive system.
Consequently, the composition operator for morphisms in $\ascat{TcDgr}$ passes to the quotient to define a composition operator between $\simeq_{\dFl}$ equivalence classes of triangle-collapsing digraph maps.
This allows us to define a new category.

\begin{defin}
The \mdf{naive $\dFl$-homotopy category}, \mdf{$\ascat{h}_{\dFl}\ascat{TcDgr}$} is the category whose objects are simple digraphs and morphisms are $\simeq_{\dFl}$ equivalence classes of triangle-collapsing digraph maps.
We use the notation \mdf{$[\bullet]_{\dFl}$} to denote objects and morphisms in this category.
\end{defin}

Note that the objects of $\ascat{h}_{\dFl}\ascat{TcDgr}$ are the same as $\ascat{TcDgr}$, but the morphisms are collapsed, according to the $\simeq_{\dFl}$ equivalence relation.
Corollary~\ref{cor:pullback_homotopy_induces_identical} automatically yields the following factorisation and hence a stability theorem.

\begin{cor}\label{cor:functor_factorisations}\label{cor:generic_stability}
The directed flag complex homology functor $H \circ \dFl: \ascat{TcDgr} \to \ascat{grVec}$ factors through $\ascat{h}_{\dFl}\ascat{TcDgr}$.
Therefore,
given $M, N \in \Funct{\Rposet}{\ascat{TcDgr}}$,
\begin{equation}
d_I((H\circ \dFl)(M), (H\circ \dFl)(N)) \leq d_I({[M]}_{\dFl}, {[N]}_{\dFl}) \leq d_I(M, N).
\end{equation}
More explicitly, if there are natural transformations
\begin{equation}
f: M \Rightarrow \shifted{N}{\delta}
\quad \text{and} \quad
g: N \Rightarrow \shifted{M}{\delta},
\end{equation}
such that $\shifted{g}{\delta} \circ f \simeq_{\dFl} \transmorph{M}{2\delta}(t)$
and $\shifted{f}{\delta} \circ g \simeq_{\dFl} \transmorph{N}{2\delta}(t)$
for every $t\in\Rposet$
then,
\begin{equation}
d_I\big( (H\circ \dFl)(M),  (H \circ\dFl)(N) \big) \leq \delta.
\end{equation}
\end{cor}
Essentially, this tells us that it suffices to $\delta$-interleave $M$ and $N$, up to $\simeq_{\dFl}$ equivalence.
In general, this is a significantly easier task, especially when $M$ and $N$ are on different vertex sets.
Note that we only require a homotopy at each $t\in\Rposet$; the homotopies need not be natural, i.e.\ they need not commute with the structure maps of $M$ and $N$.
In the sequel, we construct explicit examples of such interleavings.

A \mdf{filtration of digraphs on $V$} is a family of digraphs ${\{ G_t \}}_{t \in \R}$ such that $V(G_t)=V$ is constant and if $s \leq t$ then $E(G_s) \subseteq E(G_t)$.
Connecting these graphs via the natural inclusions, we can view a filtration as a functor $G_\bullet: \Rposet \to \ascat{StDgr}$.
Equivalently, we can view such a filtration as a function $\mathrm{ET}(G_\bullet): V \times V \to \R \cup \{\pm \infty\}$, which records the entrance time of each edge:
\begin{equation}
\mathrm{ET}(G_\bullet)(v, w) \defeq \inf \left\{ t\in \R \rmv (v, w) \in E(G_t) \right\}.
\end{equation}

\begin{prop}
Let $G_\bullet, G'_\bullet$ be two filtrations of digraphs on a common vertex set $V$.
Then
\begin{equation}
d_I\big(
    (H \circ \dFl) ( G_\bullet ), (H \circ \dFl) ({G'}_\bullet)
\big)
\leq
\norm{\mathrm{ET}(G_\bullet) - \mathrm{ET}({G'}_\bullet)}_\infty.
\end{equation}
\end{prop}
\begin{proof}
Take arbitrary $\epsilon > 0$ and denote $\delta \defeq
\norm{\mathrm{ET}(G_\bullet) - \mathrm{ET}(G'_\bullet)}_\infty$.
Using the identity vertex map in both directions, it is straightforward to show that $G_\bullet$ and $G'_\bullet$ are $(\delta+\epsilon)$-interleaved.
Lemma~\ref{lem:functors_are_interleaving_lipschitz} applied to the functor $H \circ \dFl$ then implies
$
d_I\big(
    (H \circ \dFl) ( {G}_\bullet ), (H \circ \dFl) ({G'}_\bullet)
\big) \leq \delta + \epsilon
$.
Since $\epsilon$ was arbitrary, we can take $\epsilon \to 0$.
\end{proof}

Note that, in this proof, we construct an interleaving directly in $\ascat{TcDgr}$, rather than in the homotopy category $\ascat{h}_{\dFl}\ascat{TcDgr}$.
In the following, we will see an example where we must use $\Sys[\dFl]$.

\begin{defin}
A \mdf{weighted digraph} is a triple $G=(V, E, w)$ where $(V, E)$ is a digraph and $w: E \to (0, \infty)$ is a positively-valued function, which we call the \mdf{weighting}.
We use $\mdf{w(G)}\defeq w$ to refer to the weighting.
\end{defin}

\begin{defin}
Fix a weighted digraph $G$, with weighting $w\defeq w(G)$.
\begin{enumerate}
\item A \mdf{path} $p$ is a sequence of edges $e_0, \dots, e_k$ such that, writing $e_i=(s_i, t_i)$, then $s_{i+1} = t_i$ for every $i$.
We say $p$ is a path \mdf{from $s_0$ to $t_k$} and write \mdf{$p: s_0 \leadsto t_k$}.
\item The \mdf{length} of such a path is the sum of its weights, $\mdf{\ell(p)}\defeq\sum_{i=0}^k w(e_i)$.
\item
The \mdf{shortest-path filtration} of $G$ is a filtration of digraphs on $V(G)$,
$\mdf{\mathrm{SP}(G)}\defeq {\{ {\mathrm{SP}(G)}_{t} \}}_{t\in \R}$,
where
\begin{equation}
(i, j) \in E({\mathrm{SP}(G)}_t)
\iff
\text{there is a path }p:i\leadsto j\text{ such that }\ell(p)\leq t.
\end{equation}
\item
We say $G$ is \mdf{directed acyclic}, or a \mdf{DAG},
if whenever there is a (non-trivial) path $p:v_0 \leadsto v_1$ then there is not a path $v_1 \leadsto v_0$.
\end{enumerate}
\end{defin}

\begin{rem}
If $G$ is a DAG then $\mathrm{SP}(G)_t$ is oriented for every $t\in\R$.
\end{rem}

Given a weighted digraph, one can obtain a new weighted digraph by subdividing each of its edges.
\begin{defin}
Fix a weighted digraph $G$, with weighting $w\defeq w(G)$.
\begin{enumerate}
\item A \mdf{subdivision} is a map from a subset of edges, $F \subseteq E(G)$, 
into the formal disjoint union of the interiors of the standard simplices,
$
S: F \to \sqcup_{d\in \N} \mathrm{int}(\Delta^d)
$.
\item Given such a subdivision and an edge $e\in F$, let \mdf{$d(e)$} denote the dimension such that $S(e) \in \mathrm{int}(\Delta^{d(e)})$.
Further, denote the components of $S(e)$ by ${S(e)}_0, \dots, {S(e)}_{d(e)}$.
\item The \mdf{subdivision of $G$ by $S$} is a new weighted digraph $\mdf{\wdop{s}{S}{G}}\defeq(V_S, E_S, w_S)$ given by
\begin{align}
V_S &\defeq V(G) \sqcup \bigsqcup_{e \in F} \{v_{e,1}, \dots, v_{e, d(e)} \}, \\
E_S &\defeq \big( E(G)\setminus F \big) \sqcup \bigsqcup_{e \in F} \{\tau_{e,0}, \dots, \tau_{e, d(e)} \}, \\w_S(\tau) & \defeq
\begin{cases}
w(\tau) & \text{if }\tau \in E(G), \\
S(e)_i \cdot w(e) &\text{if }\tau=\tau_{e, i},
\end{cases}
\end{align}
where $\tau_{e, i} = (v_{e, i}, v_{e, i+1})$, and we denote $v_{e, 0} \defeq \st(e)$ and $v_{e, d(e)+1}\defeq \fn(e)$.
\end{enumerate}
\end{defin}

\begin{figure}[phbt]
  \centering
  \begin{tikzpicture}[
  roundnode/.style={circle, fill=black, minimum size=4pt},
	squarenode/.style={fill=black, minimum size=4pt},
	inner sep=2pt,
	outer sep=1pt
  ]
  \node (a) at (0, 0) [roundnode, label=left:$v_0$] {};
  \node (b) at (1, 1) [roundnode, label=above:$v_1$] {};
  \node (c) at (2, 0) [roundnode, label=right:$v_2$] {};

  \node (a2) at (4, 0) [roundnode, label=left:$v_0$] {};
  \node (b2) at (5, 1) [roundnode, label=above:$v_1$] {};
  \node (c2) at (6, 0) [roundnode, label=right:$v_2$] {};

  \node (a2b2_1) at (4.5, 0.5) [roundnode] {};
  \node (b2c2_1) at (5.5, 0.5) [roundnode] {};
  \node (a2c2_1) at (4.666, 0) [roundnode] {};
  \node (a2c2_2) at (5.333, 0) [roundnode] {};

  \draw[->] (a)--(b) node[midway, above, sloped] {\tiny $2$};
  \draw[->] (b)--(c) node[midway, above, sloped] {\tiny $2$};
  \draw[->] (a)--(c) node[midway, below, sloped] {\tiny $3$};

  \draw[->] (a2)--(a2b2_1) node[midway, above, sloped] {\tiny $1$};
  \draw[->] (a2b2_1)--(b2) node[midway, above, sloped] {\tiny $1$};
  \draw[->] (b2)--(b2c2_1) node[midway, above, sloped] {\tiny $1$};
  \draw[->] (b2c2_1)--(c2) node[midway, above, sloped] {\tiny $1$};
  \draw[->] (a2)--(a2c2_1) node[midway, below, sloped] {\tiny $1$};
  \draw[->] (a2c2_1)--(a2c2_2) node[midway, below, sloped] {\tiny $1$};
  \draw[->] (a2c2_2)--(c2) node[midway, below, sloped] {\tiny $1$};

  \node[] at (1, -1) {$G$};
  \node[] at (5, -1) {$\wdop{s}{S}{G}$};
\end{tikzpicture}
   \caption{
  Two weighted digraphs, where weights are denoted by edge annotations.
  The second digraph, $\wdop{s}{S}{G}$, is the result of applying the subdivision $S$ to the first digraph, $G$,
  where $S(v_0, v_2) = (1/3, 1/3, 1/3)$ and $S(v_0, v_1) = S(v_1, v_2) = (1/2, 1/2)$.
  }\label{fig:subdivision}
\end{figure}
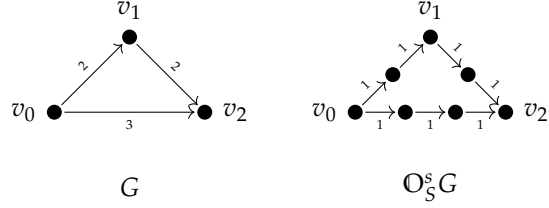

\begin{prop}\label{prop:subdiv_dag_stability}
Let $G$ be a weighted DAG and $S: F \to \sqcup_{d\in \N}\mathrm{int}(\Delta^d)$ a subdivision.
Then
\begin{equation}
d_I\big(
    (H\circ \dFl)(\mathrm{SP}(G))
    ,
    (H\circ \dFl)(\mathrm{SP}(\wdop{s}{S}{G}))
\big)
\leq
\max_{e \in F} w(e).
\end{equation}
\end{prop}
\begin{proof}
Denote $\delta \defeq \max_{e\in F} w(e)$.
By Corollary~\ref{cor:generic_stability}, it suffices to show that $[\mathrm{SP}(G)]_{\dFl}$ and $[\mathrm{SP}(\wdop{s}{S}{G})]_{\dFl}$ are $\delta$-interleaved in $\ascat{h}_{\dFl}\ascat{TcDgr}$.
We construct this interleaving via two vertex maps, as was done in~\cite[Theorem~5.16]{Chaplin2024}.
Let $f:V(G) \to V(\wdop{s}{S}{G})$ be given by the obvious inclusion.
Let $g:V(\wdop{s}{S}{G}) \to V(G)$ be given by the formula
\begin{equation}
g(v)\defeq
\begin{cases}
v & \text{if }v\in V(G), \\
\st(e) & \text{if }v=v_{e, i}\text{ and }\sum_{j < i}S(e)_j < 1/2,\\
\fn(e) & \text{if }v=v_{e, i}\text{ and }\sum_{j < i}S(e)_j \geq 1/2.\\
\end{cases}
\end{equation}
It is easy to see that $f$ does not increase the lengths of paths and hence induces a digraph map $f:{\mathrm{SP}(G)}_{t} \to {\mathrm{SP}(\wdop{s}{S}{G})}_{t}$ at every $t\in \R$.
That is, $f$ induces a natural transformation $\mathrm{SP}(G) \Rightarrow \mathrm{SP}(\wdop{s}{S}{G})$.
Since $\mathrm{SP}(\wdop{s}{S}{G})$ is a filtration, $f$ also induces a natural transformation $\mathrm{SP}(G) \Rightarrow \shifted{\mathrm{SP}(\wdop{s}{S}{G})}{\delta}$.
Note that since $f$ is injective, it is certainly triangle-collapsing.

In the opposite direction, $g$ can increase the length of paths.
In particular, the start of the path may increase by $\delta / 2$ and so too may the end of the path.
Therefore, $g$ induces a digraph map $g:{\mathrm{SP}(\wdop{s}{S}{G})}_{t} \to {\mathrm{SP}(G)}_{t+\delta}$ at every $t\in \R$.
That is, $g$ induces a natural transformation $\mathrm{SP}(\wdop{s}{S}{G}) \Rightarrow \shifted{\mathrm{SP}(G)}{\delta}$.
For more details, see the proof of~\cite[Theorem~5.16]{Chaplin2024}.
Next, the only time we observe $g(x) = g(y)$ with $x \neq y$ is if $x=v_{e,i}$ and $y=v_{e, k}$ for some $e \in F$ and $i < k$.
Moreover, $x$ and $y$ must appear in the same half of a subdivided edge.
If there is a triangle $xvy \in \dFl({\mathrm{SP}(\wdop{s}{S}{G})}_{t})$ then $v$ must appear between $x$ and $y$ in the subdivision of $e$ (since $G$ is directed acyclic).
That is, we must have $v=v_{e, j}$ for some $i < j < k$, in which case $g(v)=g(x)$.
Hence, $g$ is triangle-collapsing.

It remains to show that these transformations constitute a $\delta$-interleaving in $\ascat{h}_{\dFl}\ascat{TcDgr}$.
Notice that, as vertex maps $g \circ f = \id_{V(G)}$ and hence $g \circ f=\transmorph{{\mathrm{SP}(G)}}{2\delta}$. 
In general, it is not the case that $f \circ g = \transmorph{\mathrm{SP}(\wdop{s}{S}{G})}{2\delta}$.
Instead, we will show $f \circ g \simeq_{\dFl} \transmorph{\mathrm{SP}(\wdop{s}{S}{G})}{2\delta}$ at every $t\in \R$.

To prove this $\dFl$-homotopy equivalence, we go via an interim digraph map $h: V(\wdop{s}{S}{G}) \to V(\wdop{s}{S}{G})$, given by
\begin{equation}
h(v)\defeq
\begin{cases}
v & \text{if }v\in V(G), \\
\st(e) & \text{if }v=v_{e, i}\text{ and }\sum_{j < i}S(e)_j < 1/2,\\
v & \text{if }v=v_{e, i}\text{ and }\sum_{j < i}S(e)_j \geq 1/2.\\
\end{cases}
\end{equation}
We first note that this is indeed a digraph map
$h:{\mathrm{SP}(\wdop{s}{S}{G})}_{t} \to {\mathrm{SP}(\wdop{s}{S}{G})}_{t+2\delta}$ because $h$ can only lengthen paths at their beginning, and by at most $\delta/2$.
Moreover, $h$ is triangle-collapsing at every $t\in\R$, by the same argument as $g$.

\begin{claim}
If $x\tooreq y$ in ${\mathrm{SP}(\wdop{s}{S}{G})}_{t}$ then $h(x) \tooreq y$ in ${\mathrm{SP}(\wdop{s}{S}{G})}_{t+2\delta}$.
\end{claim}
\begin{poc}
Take a vertex $x \in V(\wdop{s}{S}{G})$, then either $h(x) = x$ or $x=v_{e, i}$ for some $e \in F$ and $h(x) = \st(e)$, in which case there is a path of length at most $\delta/2$ in $\wdop{s}{S}{G}$ from $h(x)$ to $x$.
Therefore, $h(x) \to x$ in ${\mathrm{SP}(\wdop{s}{S}{G})}_{t+2\delta}$ for every $t\in \R$.
Suppose now $x \to y$ is an edge in ${\mathrm{SP}(\wdop{s}{S}{G})}_{t}$, i.e.\ there is a path $x\leadsto y$ in $\wdop{s}{S}{G}$ of length at most $t$.
Again, either $h(x) = x$ or there is a path $h(x) \leadsto x$ of length at most $\delta/2$.
Therefore, there is a path $h(x)\leadsto y$ of length at most $t + \delta/2$.
\end{poc}
\begin{claim}\label{claim:subdiv_no_directed_cycle}
There is no edge $x\to y$
in ${\mathrm{SP}(\wdop{s}{S}{G})}_{t}$ such that $h(x) = y$.
\end{claim}
\begin{poc}
Suppose for contradiction there is an edge $x\to y$ is an edge in ${\mathrm{SP}(\wdop{s}{S}{G})}_{t}$ such that $h(x) = y$.
As we have noted before, there is always a path $h(x) \leadsto x$ in $\wdop{s}{S}{G}$, hence there is a composite path $h(x) \leadsto x \leadsto y = h(x)$.
Since $x \neq y = h(x)$, this violates the assumption that $G$ is directed acyclic.
\end{poc}

Hence, there is a one-step $\Sys[\dFl]$-homotopy from $h$ to $\id_{V(\wdop{s}{S}{G})}$, viewed as digraph maps ${\mathrm{SP}(\wdop{s}{S}{G})}_{t} \to {\mathrm{SP}(\wdop{s}{S}{G})}_{t+2\delta}$.

\begin{claim}
If $x\tooreq y$ in ${\mathrm{SP}(\wdop{s}{S}{G})}_{t}$ then $h(x) \tooreq (f\circ g)(y)$ in ${\mathrm{SP}(\wdop{s}{S}{G})}_{t+2\delta}$.
\end{claim}
\begin{poc}
Take a vertex $x \in V(\wdop{s}{S}{G})$, then either $h(x)= (f\circ g)(x)$ or $x = v_{e, i}$ for some $e \in F$ and $h(x) = x$ and $(f \circ g)(x) = \fn(e)$, in which case there is a path of length at most $\delta/2$ in $\wdop{s}{S}{G}$ from $h(x)$ to $(f\circ g)(x)$.
Therefore, $h(x) \to (f\circ g)(x)$ in ${\mathrm{SP}(\wdop{s}{S}{G})}_{t+2\delta}$ for every $t \in \R$.
Suppose now $x \to y$ is an edge in ${\mathrm{SP}(\wdop{s}{S}{G})}_{t}$, i.e.\ there is a path $x\leadsto y$ in $\wdop{s}{S}{G}$ of length at most $t$.
Again, either $h(x) = (f\circ g)(x)$ or there is a path $h(x) \leadsto (f\circ g)(x)$ of length at most $\delta/2$.
Since $(f\circ g)$ increases path lengths by at most $\delta$, there is a path $(f\circ g)(x) \leadsto (f\circ g)(y)$ of length at most $t + \delta$.
Composing these paths yields a path $h(x) \leadsto (f\circ g )(y)$ of length at most $t + 3\delta/2$.
\end{poc}

\begin{claim}
If there is an edge $x \to y$
in ${\mathrm{SP}(\wdop{s}{S}{G})}_{t}$ such that $h(x) = (f\circ g)(y)$
then $h(x)=h(y) = (f\circ g)(x) = (f\circ g)(y)$.
\end{claim}
\begin{poc}
First note that for any node $v\in V(\wdop{s}{S}{G})$,  $(f\circ g)(v) = h(v)$ if and only if $v\in V(G)$ or $v$ appears in the first half of its subdivided edge (i.e.\ 
 $v = v_{e, i}$ for some $e \in F$ and $\sum_{j< i} S(e)_j \geq 1/2$)

As we saw in the previous claims, there must be a path $h(x) \leadsto x$ and hence there is a path $(f \circ g)(y) \leadsto x \leadsto y$.
Since $G$ is directed acyclic this means that there cannot be a path $y \leadsto (f\circ g)(y)$.
Hence, $y$ cannot be a new node in the second half of a subdivided edge.
Therefore, $h(y) = (f\circ g)(y)$.
Finally, note that $h(x) = (f\circ g)(y)$ is a node in $V(G)$, so $x$ also cannot be a new node in the second half of a subdivided edge since these are mapped to $V(\wdop{s}{S}{G})\setminus V(G)$ by $h$.
Again, this implies $h(x) = (f\circ g)(x)$.
\end{poc}

Hence, viewed as digraph maps ${\mathrm{SP}(\wdop{s}{S}{G})}_{t} \to {\mathrm{SP}(\wdop{s}{S}{G})}_{t+2\delta}$, there is a one-step $\Sys[\dFl]$-homotopy from $h$ to $(f \circ g)$.
\end{proof}

\subsection{Instabilities}\label{sec:instabilities}

Despite these stability guarantees, there are many seemingly innocuous edits one can make to a weighted digraph that drastically alter the persistent homology of its directed flag complex.
As a first example, as a partial converse to Proposition~\ref{prop:subdiv_dag_stability}, we exhibit a non-DAG which is not stable to subdivision.

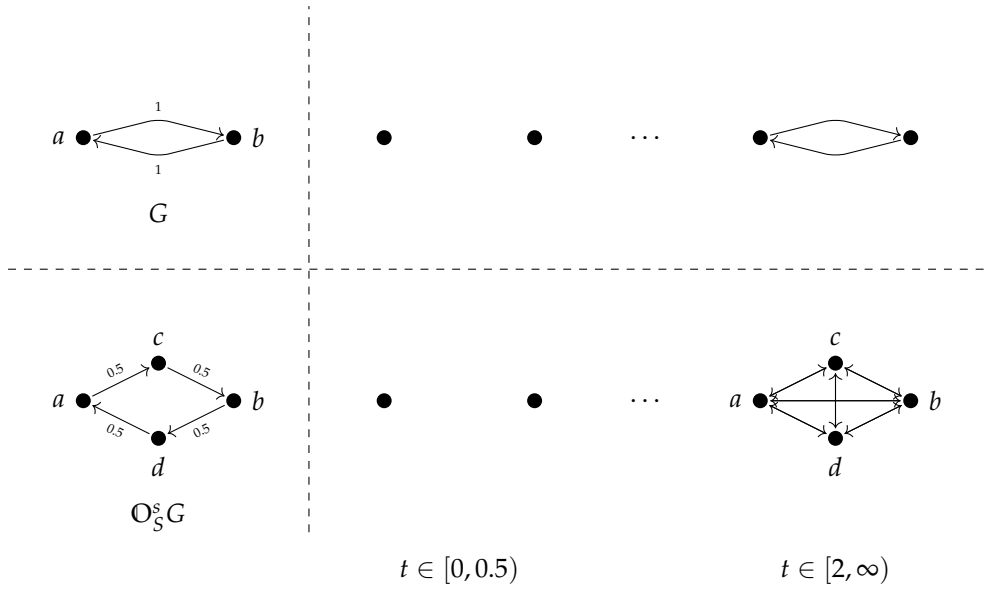
\begin{figure}[phbt]
  \centering
  \begin{tikzpicture}[
  roundnode/.style={circle, fill=black, minimum size=4pt},
	squarenode/.style={fill=black, minimum size=4pt},
	inner sep=2pt,
	outer sep=1pt
  ]

\node (a) at (0, 0.25) [roundnode, label=left:$a$] {};
  \node (b) at (2, 0.25) [roundnode, label=right:$b$] {};

  \draw[-> , rounded corners] (a) -- (1, 0.5) node[above, sloped] {\tiny $1$} -- (b);
  \draw[-> , rounded corners] (b) -- (1, 0) node[below, sloped] {\tiny $1$} -- (a);

\node (a) at (4, 0.25) [roundnode] {};
  \node (b) at (6, 0.25) [roundnode] {};

  \node at (7.5, 0.25) {$\dots$};

\node (a) at (9, 0.25) [roundnode] {};
  \node (b) at (11, 0.25) [roundnode] {};

  \draw[-> , rounded corners] (a) -- (10, 0.5) node[above, sloped] {} -- (b);
  \draw[-> , rounded corners] (b) -- (10, 0) node[below, sloped] {} -- (a);

  \node at (1, -0.75) {$G$};

\node (a) at (0, -3.25) [roundnode, label=left:$a$] {};
  \node (b) at (2, -3.25) [roundnode, label=right:$b$] {};
  \node (c) at (1, -2.75) [roundnode, label=above:$c$] {};
  \node (d) at (1, -3.75) [roundnode, label=below:$d$] {};

  \draw[-> , rounded corners] (a) -- (c) node[above, midway, sloped] {\tiny $0.5$};
  \draw[-> , rounded corners] (c) -- (b) node[above, midway, sloped] {\tiny $0.5$};
  \draw[-> , rounded corners] (b) -- (d) node[below, midway, sloped] {\tiny $0.5$};
  \draw[-> , rounded corners] (d) -- (a) node[below, midway, sloped] {\tiny $0.5$};

\node (a) at (4, -3.25) [roundnode] {};
  \node (b) at (6, -3.25) [roundnode] {};

  \node at (7.5, -3.25) {$\dots$};

\node (a) at (9, -3.25) [roundnode, label=left:$a$] {};
  \node (b) at (11, -3.25) [roundnode, label=right:$b$] {};
  \node (c) at (10, -2.75) [roundnode, label=above:$c$] {};
  \node (d) at (10, -3.75) [roundnode, label=below:$d$] {};

  \draw[->] (a) to  (b);
  \draw[->] (a) to  (c);
  \draw[->] (a) to  (d);
  \draw[->] (b) to  (c);
  \draw[->] (b) to  (d);
  \draw[->] (c) to  (d);

  \draw[->] (b) to  (a);
  \draw[->] (c) to  (a);
  \draw[->] (d) to  (a);
  \draw[->] (c) to  (b);
  \draw[->] (d) to  (b);
  \draw[->] (d) to  (c);

  \node at (1, -4.8) {$\wdop{s}{S}{G}$};

  \draw[dashed] (3, -5) -- (3, 2);
  \draw[dashed] (-1, -1.5) -- (12, -1.5);
  
  \node at (5, -5.5)  {$t\in [0, 0.5)$};
  \node at (10, -5.5) {$t\in [2, \infty)$};

\end{tikzpicture}
   \caption{
  A weighted digraph which, upon subdividing, yields a persistence module at $\infty$ interleaving distance.
  }\label{fig:instability_2}
\end{figure}

\begin{prop}
There exists a weighted digraph $G$ and a subdivision $S: E(G) \to \sqcup_{d\in \N} \Delta^{d}$ such that
\begin{equation}
d_I \big(
(H \circ \dFl)(\mathrm{SP}(G)) ,
  (H\circ \dFl)(\mathrm{SP}(\wdop{s}{S}{G}))
  \big) = \infty.
\end{equation}
\end{prop}
\begin{proof}
Consider the two weighted digraphs illustrated in Figure~\ref{fig:instability_2}.
The second weighted digraph is obtained from the first via the subdivision $S(a, b) = S(b,a) = (1/2, 1/2)$.
The first module, $(H \circ \dFl)(\mathrm{SP}(G))$, has non-trivial homology in degree 1 for $t\in [1, \infty)$.
For proof of this, recall the computations done in the proof of Proposition~\ref{prop:dflag_func_failure}.
In contrast,
$\mathrm{SP}(\wdop{s}{S}{G})_t$ is the complete digraph on $4$ nodes, for all $t\in[2, \infty)$.
Via a choice of basis for $\ker\bd_1$ or through explicit computation (i.e.\ Flagser~\cite{Luetgehetmann2020}), one can verify that this digraph has trivial homology in degree $1$.
Rank constraints imply that there is no $\epsilon$-interleaving between these modules for any $\epsilon\geq 0$.
\end{proof}

An arguably worse stability is that adding a single edge can also lead to a change in the persistent homology that is unbounded.

\begin{figure}[hbtp]
  \centering
  \begin{tikzpicture}[
  roundnode/.style={circle, fill=black, minimum size=4pt},
	squarenode/.style={fill=black, minimum size=4pt},
	inner sep=2pt,
	outer sep=1pt
  ]

\node (a) at (0, 0.25) [roundnode, label=left:$a$] {};
  \node (b) at (2, 0.25) [roundnode, label=right:$b$] {};

  \draw[-> , rounded corners] (a) -- (1, 0.5) node[above, sloped] {\tiny $1$} -- (b);
  \draw[-> , rounded corners] (b) -- (1, 0) node[below, sloped] {\tiny $1$} -- (a);

\node (a) at (4, 0.25) [roundnode] {};
  \node (b) at (6, 0.25) [roundnode] {};

\node (a) at (8, 0.25) [roundnode] {};
  \node (b) at (10, 0.25) [roundnode] {};

  \draw[-> , rounded corners] (a) -- (9, 0.5) node[above, sloped] {} -- (b);
  \draw[-> , rounded corners] (b) -- (9, 0) node[below, sloped] {} -- (a);

\node (a) at (12, 0.25) [roundnode] {};
  \node (b) at (14, 0.25) [roundnode] {};

  \draw[-> , rounded corners] (a) -- (13, 0.5) node[above, sloped] {} -- (b);
  \draw[-> , rounded corners] (b) -- (13, 0) node[below, sloped] {} -- (a);

  \node at (1, -0.75) {$G$};

\node (a) at (0, -4) [roundnode, label=left:$a$] {};
  \node (b) at (2, -4) [roundnode, label=right:$b$] {};
  \node (c) at (1, -2.5) [roundnode, label=above:$c$] {};

  \draw[-> , rounded corners] (a) -- (1, -3.75) node[above, sloped] {\tiny $1$} -- (b);
  \draw[-> , rounded corners] (b) -- (1, -4.25) node[below, sloped] {\tiny $1$} -- (a);
  \draw[->] (c) -- (a) node[midway, above, sloped] {\tiny $1$};

\node (a) at (4, -4) [roundnode] {};
  \node (b) at (6, -4) [roundnode] {};
  \node (c) at (5, -2.5) [roundnode, label] {};

\node (a) at (8, -4) [roundnode] {};
  \node (b) at (10, -4) [roundnode] {};
  \node (c) at (9, -2.5) [roundnode, label] {};

  \draw[-> , rounded corners] (a) -- (9, -3.75) node[above, sloped] {} -- (b);
  \draw[-> , rounded corners] (b) -- (9, -4.25) node[below, sloped] {} -- (a);
  \draw[->] (c) -- (a);

\node (a) at (12, -4) [roundnode] {};
  \node (b) at (14, -4) [roundnode] {};
  \node (c) at (13, -2.5) [roundnode, label] {};

  \draw[-> , rounded corners] (a) -- (13, -3.75) node[above, sloped] {} -- (b);
  \draw[-> , rounded corners] (b) -- (13, -4.25) node[below, sloped] {} -- (a);
  \draw[->] (c) -- (a);
  \draw[->] (c) -- (b);

  \node at (1, -4.8) {$G'$};

  \draw[dashed] (3, -5) -- (3, 2);
  \draw[dashed] (-1, -1.5) -- (15, -1.5);
  
  \node at (5, -5.5)  {$t\in [0, 1)$};
  \node at (9, -5.5)  {$t\in [1, 2)$};
  \node at (13, -5.5) {$t\in [2, \infty)$};

\end{tikzpicture}
   \caption{
  A weighted digraph which, upon adding a single appendage, yields a persistence module at $\infty$ interleaving distance.
  }\label{fig:instability_1}
\end{figure}
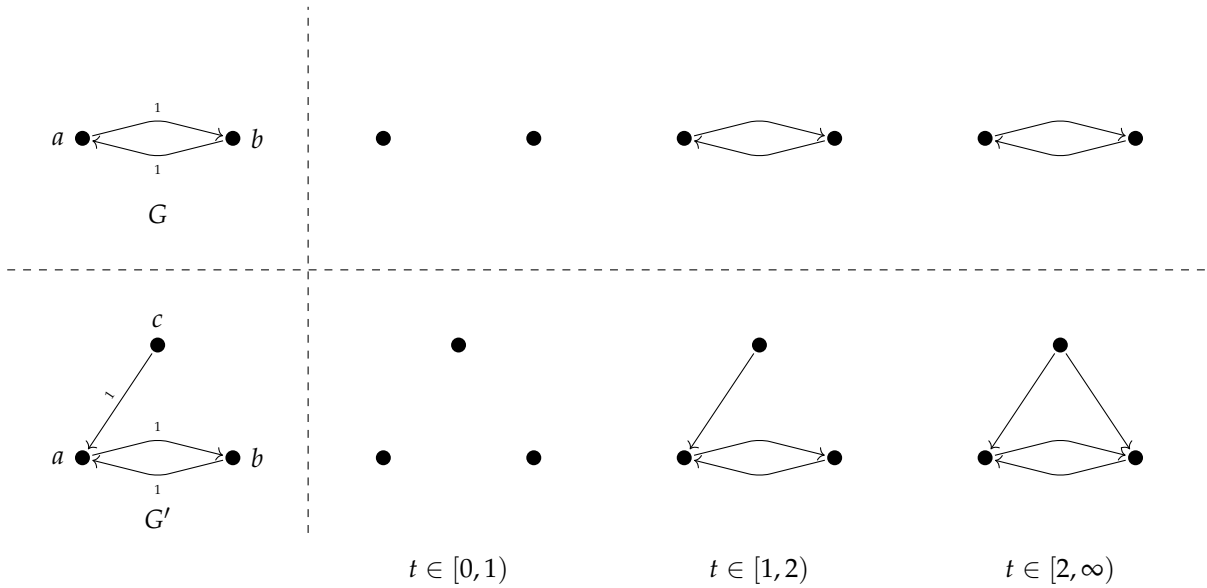

\begin{prop}
There exists a weighted digraph $G$ such that adding an appendage edge yields a weighed digraph $G'$ such that
\begin{equation}
d_I \big(
(H \circ \dFl)(\mathrm{SP}(G)) ,
  (H\circ \dFl)(\mathrm{SP}(G'))
  \big) = \infty.
\end{equation}
\end{prop}
\begin{proof}
Consider the two weighted digraphs illustrated in Figure~\ref{fig:instability_1}.
The first module, $(H \circ \dFl)(\mathrm{SP}(G))$, has non-trivial homology in degree 1 for $t\in [1, \infty)$.
In contrast, the second module,
$(H\circ \dFl)(\mathrm{SP}(G'))$,
has trivial homology in degree 1 for $t\in[2, \infty)$.
For proof of this, recall the computations done in the proof of Proposition~\ref{prop:dflag_func_failure}.
Rank constraints imply that there is no $\epsilon$-interleaving between these modules for any $\epsilon\geq 0$.
\end{proof}

In contrast, the module $(H\circ\mathcal{A})(\mathrm{SP}(G))$ is invariant to adding appendages, because a $0$-interleaving at the homotopy level can be constructed via Proposition~\ref{prop:retract_a_to_b0}.

 \appendix
\section{Mapping cylinders}

\begin{defin}
Given a weak path morphism $f: P_1 \to P_2$, the \mdf{mapping cylinder of $f$}, $\mdf{\MapCyl(f)}$, is a path complex on $(V(P_1)\times\{0\}) \cup (V(P_2)\times\{1\})$ containing paths of the following form
\begin{enumerate}
    \item $(x_0, 0) \dots (x_k, 0)$ such that $x_0 \dots x_k \in P_1$;
    \item $(y_0, 1) \dots (y_k, 1)$ such that $y_0 \dots y_k \in P_2$;
    \item $(x_0, 0) \dots (x_i, 0) (f(x_i), 1) \dots (f(x_k), 1)$ such that $x_0 \dots x_k \in P_1$ and $0\leq i \leq k$.\end{enumerate}
\end{defin}

\begin{rem}
Let $f: P_1 \to P_2$ be a weak path morphism of regular path complexes.
The mapping cylinder, $\MapCyl(f)$ is itself a regular path complex if and only if $f$ is also a strong path morphism.
\end{rem}

Note that for any path complex $P$, $\MapCyl(\id_P) = \Cyl(P)$.
We verify that mapping cylinders satisfy a universal property, analogous to the usual one in the category of topological spaces.

\begin{prop}\label{prop:mapping_cylinder_pushout}
Given a weak path morphism $f: P_1 \to P_2$, let $\iota_1: P_1 \to \Cyl(P_1)$ denote the natural inclusion $x \mapsto (x, 1)$.
Then $\MapCyl(f)$ is the pushout of $f$ and $\iota_1$ in $\ascat{WkPathC}$.
\end{prop}
\begin{proof}
First, define two vertex maps $\phi_1: P_2 \to \MapCyl(f)$ and $\phi_2: \Cyl(P_1) \to \MapCyl(f)$ by
\begin{equation}
\phi_1(y) \defeq (y, 1)
\quad\text{and}\quad
\phi_2(x, i) \defeq
\begin{cases}
(x, 0) & \text{if }i=0, \\
(f(x), 1) & \text{if }i=1.
\end{cases}
\end{equation}
These maps clearly describe weak path morphisms which make the following diagram commute
\begin{equation}
\begin{tikzcd}
P_1 \arrow[r, "f"] \arrow[d, "\iota_1"'] & P_2 \arrow[d, "\phi_1"] \\
\Cyl(P_1) \arrow[r, "\phi_2"'] & \MapCyl(f)
\end{tikzcd}
\end{equation}
Given another commuting diagram
\begin{equation}\label{dgm:pushout_psi_commutes}
\begin{tikzcd}
P_1 \arrow[r, "f"] \arrow[d, "\iota_1"'] & P_2 \arrow[d, "\psi_1"] \\
\Cyl(P_1) \arrow[r, "\psi_2"'] & Z
\end{tikzcd}
\end{equation}
in $\ascat{WkDgr}$,
define $\phi: \MapCyl(f) \to Z$ by
$\phi(x, 0)\defeq \psi_2(x, 0)$
and
$\phi(y, 1)\defeq \psi_1(y)$.
As vertex maps, $\phi$ is clearly the unique map that makes the following diagram commute
\begin{equation}
\begin{tikzcd}
P_1 \arrow[r, "f"] \arrow[d, "\iota_1"'] & P_2 \arrow[d, "\phi_1"] \arrow[ddr, bend left, "\psi_1"] & \\
\Cyl(P_1) \arrow[r, "\phi_2"'] \arrow[drr, bend right, "\psi_2"'] & \MapCyl(f) \arrow[rd, dashed, "\phi"] & \\
& & Z
\end{tikzcd}
\end{equation}
It remains to verify that $\phi$ is a weak path morphism.
Given a path $p\in\MapCyl(f)$, we split into three cases:

\textbf{Case 1:} Suppose $p = (x_0, 0) \dots (x_k, 0)$ where $x_0 \dots x_k \in P_1$.
Let $p'$ denote the same path but viewed as an element of $\Cyl(P_1)$ so that $\phi_2(p')=p$.
Then, $\phi(p) = \psi_2(p')$ must be either irregular or a path in $Z$ since $\psi_2$ is a weak path morphism.

\textbf{Case 2:} Suppose $p=(y_1, 1) \dots (y_k, 1)$ where $p'\defeq y_0 \dots y_k \in P_2$.
 Then $p = \phi_1(p')$ so $\phi(p) = \psi_1(p')$ is either irregular or a path in $Z$ since $\psi_1$ is a weak path morphism.

\textbf{Case 3:}
Suppose $p=(x_0, 0) \dots (x_i, 0) (f(x_i), 1) \dots (f(x_k), 1)$ where $x_0 \dots x_k \in P_1$ and $0 \leq i \leq n$.
Then let 
$p' \defeq (x_0, 0) \dots (x_i, 0) (x_i, 1) \dots (x_k, 1)\in \Cyl(P_1)$,
and observe
\begin{align}
\phi(p) &= \psi_2(x_0, 0) \dots \psi_2(x_i, 0) \psi_1(f(x_i)) \dots \psi_1(f(x_k)) \\
        &= \psi_2(x_0, 0) \dots \psi_2(x_i, 0) \psi_2(x_i, 1) \dots \psi_2(x_k, 1) \\
        &= \psi_2(p')
\end{align}
where the second equality holds because diagram (\ref{dgm:pushout_psi_commutes}) commutes.
Now, since $\psi_2$ is a weak path morphism, we see that $\phi(p)$ must be irregular or a path in $Z$.
\end{proof}

\begin{prop}\label{prop:homotopy_for_path_complex_cylinders}
Given a weak path morphism $f: P_1 \to P_2$, $\MapCyl(f) \simeq P_2$.
\end{prop}
\begin{proof}
Let $\gamma : P_2 \to \MapCyl(f)$ denote the natural inclusion $y\mapsto (y, 1)$ which is certainly a weak path morphism.
Let $\rho: \MapCyl(f) \to P_2$ be given by $(x, 0)\mapsto f(x)$ and $(y, 1)\mapsto y$.
Again this is a weak path morphism because paths of the form $(x_0, 0) \dots (x_i, 0) (f(x_i), 1) \dots (f(x_k), 1) \in \MapCyl(f)$ have an irregular image.
Note that $\rho \circ \gamma = \id_{P_2}$.
It remains to show that $\gamma \circ \rho \simeq \id_{\MapCyl(f)}$; we construct the homotopy explicitly.

We shall denote vertices of $\Cyl(\MapCyl(f))$ by $(v, \alpha, \beta)$ where $\alpha$ is the index for the inner mapping cylinder and $\beta$ is the index for outer cylinder.
Let $F: \Cyl(\MapCyl(f)) \to \MapCyl(f)$ be given by
\begin{equation}
F(v, \alpha, 0) \defeq (v, \alpha)
\quad\text{and}\quad
F(v, \alpha, 1) \defeq
\begin{cases}
(f(v), 1) & \text{if }\alpha=0,\\
(v, 1) & \text{if }\alpha=1.
\end{cases}
\end{equation}
There are inclusion $\iota_i: \MapCyl(f) \hookrightarrow \Cyl(\MapCyl(f))$, for $i=0, 1$, given by $\iota_i(v, \alpha) = (v, \alpha, i)$.
Note that
$F \circ \iota_0 = \id_{\MapCyl(f)}$
and 
$F \circ \iota_1  = \gamma \circ \rho$.
It remains to show that $F$ is a weak path morphism.

Take an arbitrary path $p\in\MapCyl(f)$, and denote the vertices
$
p = (x_0, \alpha_0) \dots (x_k, \alpha_k)
$.
Now consider the path
\begin{equation}
p_j \defeq (x_0, \alpha_0, 0) \dots (x_j, \alpha_j, 0) (x_j, \alpha_j, 1) \dots (x_k, \alpha_k, 1)\in \Cyl(\MapCyl(f))
\end{equation}
for some $j \in [0, k]$.
Note that the $\alpha_i$ are either constant at $0$ or $1$, or there is some index $m$ such that $\alpha_i = 0$ for $i\leq m$ and $\alpha_i = 1$ for $i > m$.
We split into cases.

\textbf{Case 1:} Suppose $\alpha_i = 0$ for every $i$.
In this case, note that $x_0 \dots x_k \in P_1$.
Then,
\begin{equation}
F(p_j) = (x_0, 0) \dots (x_j, 0) (f(x_j), 1) \dots (f(x_k), 1)
\end{equation}
which is a path in $\MapCyl(f)$.

\textbf{Case 2:} Suppose $\alpha_i = 1$ for every $i$.
In this case, $x_0 \dots x_k \in P_2$ and
\begin{equation}
F(p_j) = (x_0, 1) \dots (x_j, 1) (x_j, 1) \dots (x_k, 1) 
\end{equation}
which is clearly irregular.

\textbf{Case 3:} Suppose $\alpha_i = 0$ for $i\leq m$ and $\alpha_i=1$ for $i > m$ where $m < j$.
In this case note that $\alpha_j=1$ and hence
\begin{equation}
F(x_j, \alpha_j, 0) = (x_j, 1) = F(x_j, \alpha_j, 1).
\end{equation}
Therefore, $F(p_j)$ must be irregular.

\textbf{Case 4:} Suppose $\alpha_i = 0$ for $i\leq m$ and $\alpha_i=1$ for $i > m$ such that $m \geq j$.
In this case, note that $p_j$ contains adjacent vertices $(x_m, 0, 1)(x_{m+1}, 1, 1)$.
Moreover, since $p$ is a path in $\MapCyl(f)$ we must have $x_{m+1} = f(x_m)$.
Mapping through $F$ we see
\begin{equation}
F(x_m, 0, 1) = (f(x_m), 1) = (x_{m+1}, 1) = F(x_{m+1}, 1, 1)
\end{equation}
and thus $F(p_j)$ is irregular.

All other paths $q\in \Cyl(\MapCyl(f))$ are of the form $\iota_i(q')$ for some $i\in \{0, 1\}$ and $q'\in \MapCyl(f)$.
If $i=0$ then $F(q) = (F \circ \iota_0)(q') = q'$;
if $i=1$ then $F(q) = (F\circ \iota_1)(q') = (\gamma \circ \rho)(q')$.
Since $\gamma$ and $\rho$ are weak path morphisms, we see either $F(q) \in \MapCyl(f)$ or $F(q)$ is irregular.
\end{proof}

A similar result holds for triangle-collapsing simplicial morphisms.
It is possible to show this either by adapting the proof above, or by appealing to the connection with simplicial sets.

\begin{prop}\label{prop:homotopy_for_osc_cylinders}
Given a triangle-collapsing simplicial morphism $f: K_1 \to K_2$, $\overline{\MapCyl(f)} \simeq K_2$.
\end{prop}
\begin{proof}
By Corollary~\ref{cor:sset_homotopies_same}, the system of one-step homotopies for $\ascat{TcOSC}$ is a pull-back of the standard system for simplicial sets, $\Sys[\ascat{TcOSC}]=\sigma^\ast \Sys[\ascat{sSet}]$.
Using Proposition~\ref{prop:mapping_cylinder_pushout}, we can also show $\overline{\MapCyl(f)}$ is a pushout in $\ascat{TcOSC}$ and hence $\sigma(\overline{\MapCyl(f)})$ is the usual mapping cylinder for $\sigma(f)$.
The result then follows from the homotopy equivalence $\sigma(\overline{\MapCyl(f)})\simeq\sigma(K_2)$ in simplicial sets~\cite[Proposition~2.68]{ruschoff65lecture}.
\end{proof}

 \section{Grounded pipelines}

\subsection{Background}

This work was initially motivated by the stability analysis of grounded pipelines for weighted digraphs~\cite{Chaplin2024}.
These pipelines require two components:
\begin{enumerate}
\item a filtration map, $F:\mathrm{WDgr} \to \Obj(\Funct{\Rposet}{\Dgr})$, which assigns a filtration of digraphs to each weighted digraph; and
\item a digraph chain complex, i.e.\ a functor $C:\ascat{InclDgr} \to \ascat{Ch}$, where $\ascat{InclDgr}$ is the wide subcategory of $\ascat{WkDgr}$ restricted to morphisms which are inclusions.
\end{enumerate}
Given these complexes, one obtains the following commutative diagram for any weighted digraph $G$.
\begin{equation}
\begin{tikzcd}[row sep=small, column sep=small]
    \cdots C_3(F(G)(t) \arrow[r, "\partial_3"] &
    C_2(F(G)(t)) \arrow[rr, "\inducedch{\iota}\circ\bd_2"] \arrow[rd, "\bd_2"', dotted] & &
    C_1(G\cup F(G)(t)) \arrow[r, "\bd_1"] & 
    C_0(G\cup F(G)(t)) \cdots \\
   & & C_1(F(G)(t))\arrow[ru, "\inducedch{\iota}"', dotted, hook]
   \arrow[r, dotted, "\bd_1"']
   & C_0(F(G)(t) )\arrow[ru, dotted, "\inducedch{\iota}"', hook] &
\end{tikzcd}
\end{equation}
The top row of this diagram is a chain complex, which we denote 
\mdf{$\zb{C}_F(G)(t)$}.
By~\cite[Lemma~3.9]{Chaplin2024}, the functoriality of $C$ ensures that this can be made into a persistent chain complex $\zb{C}_F(G)\in\Funct{\Rposet}{\ascat{Ch}}$.
The persistent homology of this can then be used as a topological summary of $G$.

The initial study~\cite{Chaplin2024} considered this pipeline with $F=\mathrm{SP}$ and $C=\Omega \circ \mathcal{A}$.
\textbf{In this section, we take $F=\mathrm{SP}$ and $C=\Omega\circ\dFl$ and study the resulting pipeline.}
The assignment $G\mapsto \zb{C}_F(G)$ can be made functorial, but we must restrict to morphisms between weighted digraphs that induce triangle-collapsing morphisms $\mathrm{SP}(G)_t \to \mathrm{SP}(H)_t$.

\begin{defin}
Given a weighted digraph $G$, the \mdf{shortest-path quasimetric} is a quasimetric on $V(G)$ given for $i, j \in V(G)$ by
\begin{equation}
d_G(i, j) \defeq \inf \left\{ t\geq 0 \rmv \text{there is a path }p:i\leadsto j\text{ with }\ell(p) \leq t\right\}.
\end{equation}
Given two weighted digraphs $G, H$,
a vertex map $f: V(G) \to V(H)$ is called a \mdf{contraction} if
\begin{equation}
d_H(f(i), f(j)) \leq d_G(i, j)
\end{equation}
for all $i, j \in V(G)$.
\end{defin}

Note that if $f$ is a contraction then it induces a weak digraph map $\mathrm{SP}(G)_t \to \mathrm{SP}(H)_t$ for every $t\in \R$.

\begin{defin}
A digraph map $f: G \to H$ is \mdf{path-collapsing} if whenever there are paths $i \leadsto j \leadsto k$ and $f(i) = f(k)$ then $f(i) = f(j) = f(k)$.
\end{defin}

Note that any path-collapsing digraph map is a triangle-collapsing map.
Moreover, if $f: G \to H$ is path-collapsing and a contraction then it induces triangle-collapsing digraphs maps $\mathrm{SP}(G)_t \to \mathrm{SP}(H)_t$ for every $t\in \R$.
With these observations, we can deduce the functoriality of this pipeline.

\begin{defin}
Let \mdf{$\ascat{ContPcWDgr}$} denote the category of all simple, weighted digraphs where morphisms are vertex maps which are both contractions and path-collapsing digraph maps.
\end{defin}

\begin{lemma}
The assignment $G\mapsto\zb{C}_F(G)$ can be made into a functor $\ascat{ContPcWDgr} \to \Funct{\Rposet}{\ascat{Ch}}$.
\end{lemma}

\begin{defin}
\mdf{Grounded persistent directed flag homology} is the functor
$\mdf{\zbpershomdflagmap}:\ascat{ContPcWDgr} \to \Funct{\Rposet}{\ascat{Vec}}$ given by $\zbpershomdflagmap \defeq H_1 \circ \zb{C}_F$.
\end{defin}

In the rest of this section we study the stability of $\zbpershomdflagmap$.
Namely, we wish to understand what alterations can be made to a weighted digraph $G$ without dramatically changing $\zbpershomdflag{G}$, as measured by the interleaving distance.

\subsection{Stability theorems}

We aim to state a general stability theorem for $\zbpershomdflagmap$, along the lines of~\cite[Theorem~5.8]{Chaplin2024}.
First, we require relative version of the homotopies in $\Sys[\dFl]$ such that the induced chain homotopy is trivial on the fixed component.
We keep the presentation fairly close to~\cite[\S~5]{Chaplin2024} in order to emphasise the differences.

\begin{defin}
Given two triangle-collapsing digraph maps $f, g: G \to H$ and a subset $A\subseteq V(G)$,
if $f \simeq_{\dFl, 1}g$
and moreover
$f(v) = g(v)$ for every $v\in A$,
then
we say that $f$ and $g$ are \mdf{one-step $\Sys[\dFl]$-homotopic relative $A$}.
This determines a system of one-step homotopies, which we denote \mdf{$\Sys[\dFl, A]$}.
We denote the resulting equivalence relation by \mdf{$f\simeq_{\dFl} g\ (\mathrm{rel}\ A)$} and say that $f$ and $g$ are \mdf{$\Sys[\dFl]$-homotopic relative $A$}.
\end{defin}

By essentially the same proof as~\cite[Lemma~2.37]{Chaplin2024}, we obtain the following behaviour of the induced chain homotopy between maps that are $\Sys[\dFl]$-homotopic relative $A$.

\begin{lemma}
Suppose $f, g: G \to H$ are triangle-collapsing digraph maps that are $\Sys[\dFl]$-homotopic relative $A$, for some $A\subseteq V(G)$.
Then, $\dFl(f)$ and $\dFl(G)$ are $\Sys[\ascat{WkRPC}]$-homotopic
so Theorem~\ref{thm:homotopy_yields_chain_homotopy} induces a chain homotopy
between $\Omega(\dFl(f))$ and $\Omega(\dFl(g))$;
denote its components by
\begin{equation}
L_k: \Omega_k(\dFl(G)) \to \Omega_{k+1}(\dFl(H)).
\end{equation}
Let $G_A$ denote the induced subgraph of $G$ on the vertices in $A$.
Then $L_k(\Omega_k(\dFl(G_A))) = 0$.
\end{lemma}

We now introduce $\delta$-shifting vertex maps which will be our main method for constructing interleavings.

\begin{defin}[{\cite[Definition~5.1]{Chaplin2024}}]
Given two weighted digraphs $G, H$, a \mdf{$\delta$-shifting vertex map} is a vertex map $f: V(G) \to V(H)$ such that $f$ induces weak digraph maps
$\mathrm{SP}(G)_t \to \mathrm{SP}(H)_{t+\delta}$
and
$G\cup\mathrm{SP}(G)_t \to H\cup\mathrm{SP}(H)_{t+\delta}$,
for every $t \in \R$.
\end{defin}

\begin{defin}[{c.f.~\cite[Definition~5.5]{Chaplin2024}}]
Let $G, H$ be two weighted digraphs and let $f: V(G) \to V(H)$ and $g: V(H) \to V(G)$ be two vertex maps between them,
\begin{enumerate}
\item Construct the following sets:
\begin{align}
\mdf{V_{\mathrm{fix}} (g, f)} &\defeq \left\{ v \in V(G) \rmv (g \circ f)(v) = v \right\}, \\
\mdf{E_{\mathrm{fix}} (g, f)} &\defeq \left\{ e \in E(G) \rmv (g \circ f)(e) = e \right\}, \\
\mdf{V_{\mathrm{diff}}(g, f)}  &\defeq \left\{ v \in V(G) \rmv (g \circ f)(v) \neq v \right\}, \\
\mdf{E_{\mathrm{diff}}(g, f)}  &\defeq \left\{ e \in E(G) \rmv (g \circ f)(e) \neq e \right\}. 
\end{align}
Moreover, denote $\mdf{G_{\mathrm{diff}}(g, f)} \defeq (V(G), E_{\mathrm{diff}}(g, f))$.
\item If both
 $\id: G_{\mathrm{diff}}(g, f) \to \mathrm{SP}(G)_{2\delta}$
and 
 $g \circ f: G_{\mathrm{diff}}(g, f) \to \mathrm{SP}(G)_{2\delta}$
are triangle-collapsing digraph maps and furthermore they are $\Sys[\dFl]$-homotopic relative $V_{\mathrm{fix}}(g, f)$
then we
say that $(g, f)$ has \mdf{grounded $\dFl$-codistortion $\leq \delta$}.
\item If $f$ and $g$ are both path-collapsing, $\delta$-shifting vertex maps and the pairs $(g, f)$ and $(f, g)$ both have grounded $\dFl$-codistortion $\leq \delta$ then we say they form a \mdf{$\delta$-grounded $\dFl$-interleaving.}
\end{enumerate}
\end{defin}

With these definitions and a proof essentially identical to that of~\cite[Theorem~5.8]{Chaplin2024}, we obtain our main stability theorem for $\zbpershomdflagmap$.

\begin{theorem}
Given two weighted digraphs $G, H$, if there is a $\delta$-grounded $\dFl$-interleaving between them then
\begin{equation}
d_I(\zbpershomdflag{G}, \zbpershomdflag{H})\leq \delta.
\end{equation}
\end{theorem}

In~\cite[\S~5]{Chaplin2024}, the analogous theorem was used to study the stability of the pipeline to various perturbations of weighted digraphs. With this new theory, the results of that work go through for $\zbpershomdflagmap$ as soon as one has made the necessary constraints to ensure that the $\delta$-shifting vertex maps are path-collapsing and that the one-step homotopies belong to $\Sys[\dFl, V_{\mathrm{fix}}(g, f)]$.

To this end, note that if $G$ is a digraph and $H$ is a DAG then any weak digraph map $G \to H$ is necessarily path-collapsing.
Therefore, denoting the category of weighted DAGs with contractions for morphisms by $\ascat{ContWDag}$, we have an inclusion of categories $\ascat{ContWDag} \subseteq\ascat{ContPcWDgr}$.
As such, most of the stability theorems obtained in~\cite[\S~5]{Chaplin2024} apply automatically once we restrict to weighted DAGs.
To summarise these results, we present Table~\ref{tbl:dflag_stability_result_summary}, which is a reproduction of~\cite[Table~1]{Chaplin2024} but with additional annotations to denote when the result applies unrestricted to $\zbpershomdflagmap$, and when the result must first be restricted to DAGs.

\begin{table}[hbtp]
\begin{center}
\ars{1.2}
\tcs{0.7\tabcolsep}
\renewcommand\theadfont{\bfseries}
\newcommand*\partialthm{${}^\blacklozenge$}
\newcommand*\unrestrictedthm{${}^\checkmark$}
\newcommand*\restrictedthm{${}^\restriction$}
\begin{tabular}{ l | c c c c }
  \thead{Operation} & \thead{Locally\\Stable} & \thead{Non-locally\\Stable} & \thead{Locally\\Unstable}  & \thead{Isomorphism} \\ \hline
  Weight perturbation & Theorem~5.11\unrestrictedthm & & & \\
  Edge subdivision & Theorem~5.16\restrictedthm & & & \\
  Edge collapse & Theorem~5.28\partialthm\restrictedthm &
                & Theorem~5.35\unrestrictedthm & \\
  Edge deletion &
  Corollary~5.41\partialthm\unrestrictedthm &
  Theorem~5.38\unrestrictedthm &
  Theorem~5.42\unrestrictedthm &
Theorem~5.45\partialthm\restrictedthm
\\
  Vertex deletion & & Theorem~5.50\restrictedthm & Corollary~5.49\unrestrictedthm & Corollary~5.48\partialthm\unrestrictedthm
\end{tabular}
\caption{
  Stability and instability theorems for $\zbpershomdflagmap$, under various digraph operations.
  Theorem and Corollary numbering references results in the publication~\cite{Chaplin2024}.
  {$\blacklozenge$} Denotes a theorem which only applies to a subset of such operations.
  {$\checkmark$} Denotes a theorem which applies to the directed flag complex unrestricted.
  {$\restriction$} Denotes a theorem which applies to the directed flag complex after restricting to weighted digraphs $G$, such that $G$ and  $\wdop{T}{\theta}{G}$ are both DAGs.
}\label{tbl:dflag_stability_result_summary}
\end{center}
\end{table}
 \printbibliography 

\end{document}